\documentclass[a4paper, reqno, 12pt]{amsart}

\usepackage[usenames,dvipsnames]{color}
\usepackage{amsthm,amsfonts,amssymb,amsmath,amsxtra}
\usepackage[all]{xy}
\SelectTips{cm}{}
\usepackage{xr-hyper}
\usepackage[colorlinks=false,
   citecolor=Black,
   linkcolor=Red,
   urlcolor=Blue]{hyperref}
\usepackage{verbatim}

\usepackage[margin=1.25in]{geometry}
\usepackage{mathrsfs}

\RequirePackage{xspace}
\RequirePackage{etoolbox}
\RequirePackage{varwidth}
\RequirePackage{enumitem}
\RequirePackage{tensor}
\RequirePackage{mathtools}
\RequirePackage{longtable}
\RequirePackage{multirow}
\RequirePackage{tikz}
\usetikzlibrary{arrows.meta}
\RequirePackage{dynkin-diagrams}

\def\ge{\geqslant}
\def\le{\leqslant}
\def\a{\alpha}

\def\d{\delta}

\def\t{\tau}
\def\th{\theta}

\def\v{\vartheta}

\def\i{^{-1}}

\def\torusO#1{\mathcal T_{#1}}

\def\<{\langle}
\def\>{\rangle}

\newcommand{\BC}{\ensuremath{\mathbb {C}}\xspace}

\newcommand{\BF}{\ensuremath{\mathbb {F}}\xspace}
\newcommand{{\BG}}{\ensuremath{\mathbb {G}}\xspace}

\newcommand{{\BK}}{\ensuremath{\mathbb {K}}\xspace}

\newcommand{\BQ}{\ensuremath{\mathbb {Q}}\xspace}

\newcommand{\BZ}{\ensuremath{\mathbb {Z}}\xspace}

\newcommand{\CD}{\ensuremath{\mathcal {D}}\xspace}

\newcommand{\CO}{\ensuremath{\mathcal {O}}\xspace}

\newcommand{\CT}{\ensuremath{\mathcal {T}}\xspace}

\newcommand{\Ad}{{\mathrm{Ad}}}

\DeclareMathOperator{\Aut}{Aut}

\newcommand{\Cl}{{\mathrm{Cl}}}

\DeclareMathOperator{\diag}{diag}

\DeclareMathOperator{\End}{End}

\DeclareMathOperator{\Hom}{Hom}

\newcommand{\Ind}{{\mathrm{Ind}}}

\DeclareMathOperator{\tr}{tr}

\DeclareMathOperator{\Irr}{Irr}

\def\kk{\mathbf k}
\DeclareMathOperator{\supp}{supp}

\newtheorem{theorem}{Theorem}
\newtheorem{alphatheorem}{Theorem}

\newtheorem{proposition}[theorem]{Proposition}
\newtheorem{lemma}[theorem]{Lemma}

\newtheorem{corollary}[theorem]{Corollary}

\theoremstyle{definition}
\newtheorem{definition}[theorem]{Definition}

\newtheorem{remark}[theorem]{Remark}

\numberwithin{equation}{section}
\numberwithin{theorem}{section}

\setitemize[0]{leftmargin=*,itemsep=\the\smallskipamount}
\setenumerate[0]{leftmargin=*,itemsep=\the\smallskipamount}

\renewcommand{\to}{%
   \ifbool{@display}{\longrightarrow}{\rightarrow}%
   }
\let\shortmapsto\mapsto
\renewcommand{\mapsto}{%
   \ifbool{@display}{\longmapsto}{\shortmapsto}%
   }
\newlength{\olen}
\newlength{\ulen}
\newlength{\xlen}
\newcommand{\xra}[2][]{%
   \ifbool{@display}%
      {\settowidth{\olen}{$\overset{#2}{\longrightarrow}$}%
       \settowidth{\ulen}{$\underset{#1}{\longrightarrow}$}%
       \settowidth{\xlen}{$\xrightarrow[#1]{#2}$}%
       \ifdimgreater{\olen}{\xlen}%
          {\underset{#1}{\overset{#2}{\longrightarrow}}}%
          {\ifdimgreater{\ulen}{\xlen}%
             {\underset{#1}{\overset{#2}{\longrightarrow}}}
             {\xrightarrow[#1]{#2}}}}%
      {\xrightarrow[#1]{#2}}
   }
\makeatother
\newcommand{\xyra}[2][]{%
   \settowidth{\xlen}{$\xrightarrow[#1]{#2}$}%
   \ifbool{@display}%
      {\settowidth{\olen}{$\overset{#2}{\longrightarrow}$}%
       \settowidth{\ulen}{$\underset{#1}{\longrightarrow}$}%
       \ifdimgreater{\olen}{\xlen}%
          {\mathrel{\xymatrix@M=.12ex@C=3.2ex{\ar[r]^-{#2}_-{#1} &}}}%
          {\ifdimgreater{\ulen}{\xlen}%
             {\mathrel{\xymatrix@M=.12ex@C=3.2ex{\ar[r]^-{#2}_-{#1} &}}}
             {\mathrel{\xymatrix@M=.12ex@C=\the\xlen{\ar[r]^-{#2}_-{#1} &}}}}}%
      {\mathrel{\xymatrix@M=.12ex@C=\the\xlen{\ar[r]^-{#2}_-{#1} &}}}%
   }
\makeatletter
\newcommand{\xla}[2][]{%
   \ifbool{@display}%
      {\settowidth{\olen}{$\overset{#2}{\longleftarrow}$}%
       \settowidth{\ulen}{$\underset{#1}{\longleftarrow}$}%
       \settowidth{\xlen}{$\xleftarrow[#1]{#2}$}%
       \ifdimgreater{\olen}{\xlen}%
          {\underset{#1}{\overset{#2}{\longleftarrow}}}%
          {\ifdimgreater{\ulen}{\xlen}%
             {\underset{#1}{\overset{#2}{\longleftarrow}}}
             {\xleftarrow[#1]{#2}}}}%
      {\xleftarrow[#1]{#2}}
   }
\newcommand{\isoarrow}{%
   \ifbool{@display}{\overset{\sim}{\longrightarrow}}{\xrightarrow\sim}%
   }
\newcommand\q{\mathbf q}

\begin{document}

\title[]{Cocenter of Hecke algebras of Kac-Moody groups}

\author[Xuhua He]{Xuhua He}
\address{Department of Mathematics and New Cornerstone Science Laboratory, The University of Hong Kong, Pokfulam, Hong Kong, Hong Kong SAR, China}
\email{xuhuahe@hku.hk}

\author[Felix Schremmer]{Felix Schremmer}
\address{School of Mathematical Sciences, Xiamen University, 361005, Xiamen, P.\ R.\ China}
\email{schremmer@xmu.edu.cn}

\thanks{}



\begin{abstract}
Let $H$ be the generic Hecke algebra over $\mathbb{Z}[\mathbf{q}^{\pm 1}]$ associated to a split Kac--Moody group $G$, arising as the deformation of the group algebra of its Weyl group $W$. The cocenter $\overline{H} = H/[H, H]$ encodes the trace and character theory of $H$, playing a fundamental role in representation theory and harmonic analysis. Through deep combinatorial results on cyclic reductions in Coxeter groups, each conjugacy class $\mathcal{O}$ of $W$ determines a canonical element $T_{\mathcal{O}}$ in $\overline{H}$, and these elements are known to span the cocenter. However, establishing their linear independence has remained an open problem outside of finite and affine types.

In this paper, we solve this problem: the canonical elements form a $\mathbb{Z}[\mathbf{q}^{\pm 1}]$-basis of the cocenter $\overline{H}$. Our approach differs from earlier representation-theoretic methods in finite and affine types. To construct explicit functionals that separate all conjugacy classes, we develop a new framework based on re-normalized orbital integrals. This framework synthesizes parabolic induction, traces of infinite-dimensional bimodules, and Kac--Moody harmonic analysis into an almost-dual basis for the cocenter.

As a key local ingredient, we establish a generic duality theorem for finite groups of Lie type relating the cocenter to regular semisimple conjugacy classes, and determine precisely when this pairing is non-degenerate. Finally, we deduce the existence and uniqueness of generic class polynomials for $W$, and prove a uniform ``dimension=degree'' theorem for basic Deligne--Lusztig varieties of the split Kac--Moody group $G$.
\end{abstract}

\maketitle


\section{Introduction}
\label{sec:introduction}

\subsection{Hecke algebras and the cocenter problem}\label{sec:1.1}
Kac-Moody groups are an important class of groups in representation theory, generalizing classical notions such as Lie groups, linear algebraic groups, $p$-adic groups and loop groups within a uniform framework. Their Weyl groups $W$ are crystallographic Coxeter groups, that is, generated by a finite set $S$ of simple reflections subject to the relations $s^2 = 1$ for $s\in S$, and the braid relations
\[
\underbrace{sts\cdots}_{m_{s,t}\text{ factors}} = \underbrace{tst\cdots}_{m_{s,t}\text{ factors}} \qquad (s\neq t \text{ with } m_{s,t} < \infty),
\]
with $m_{s,t} = m_{t,s} \in \{2, 3, 4, 6, \infty\}$. The generic Iwahori--Hecke algebra $H$ associated to $(W, S)$ is the deformation of the group algebra $\mathbb{Z}[W]$ over the Laurent polynomial ring $\mathbb{Z}[\q^{\pm 1}]$. It is the associative algebra defined by generators $\{T_s \mid s\in S\}$ subject to the quadratic relations $(T_s+1)(T_s-\q) = 0$ for $s\in S$ and the deformed braid relations
\[
\underbrace{T_s T_t T_s\cdots}_{m_{s,t}\text{ factors}} = \underbrace{T_t T_s T_t\cdots}_{m_{s,t}\text{ factors}} \qquad (s\neq t \text{ with } m_{s,t} < \infty).
\]
Specializing $\q = 1$ recovers the group algebra $\mathbb{Z}[W]$. When $G$ is defined and split over a finite field $\mathbb F_q$, the $\q\mapsto q$ specialization $H_{\mathbb C}
:=
H\otimes_{\mathbb Z[\q^{\pm1}]}\mathbb C$ carries a geometric meaning: Namely, $H_{\mathbb C}$ agrees with the convolution algebra of Borel-biinvariant functions on $G(\mathbb F_q)$ which are supported on finitely many double cosets. Hecke algebras and their representations play a key role in various disciplines of mathematics, such as group theory, number theory and representation theory.

Following Dehn's work of 1911, fundamental structural questions in combinatorial group theory are often formulated as algorithmic decision problems. The first two of Dehn's three problems---the word problem and the conjugacy problem---are particularly relevant to the Weyl group considered in this paper.

Dehn's first problem for $(W,S)$ is the \emph{word problem}: determining whether two expressions in the simple generators represent the same element of $W$. The deletion condition, together with Matsumoto's theorem \cite{Ma64}, shows that every expression can be reduced using cancellations $ss\rightsquigarrow 1$ and braid relations, and that any two reduced (=minimal length) expressions of the same element are connected by braid relations. In particular, the standard elements
\[
T_w := T_{s_1} \cdots T_{s_\ell} \in H \qquad (w = s_1 \cdots s_\ell \text{ reduced})
\]
are well-defined and form a free $\mathbb{Z}[\q^{\pm 1}]$-basis $\{T_w \mid w \in W\}$ of the Hecke algebra $H$. 

Dehn's second problem is the \emph{conjugacy problem}: determining whether two elements of $W$ are conjugate. The corresponding question for the Hecke algebra is to understand the structure of its cocenter, \[
\overline{H} = H / [H, H] = H / \operatorname{span}_{\mathbb{Z}[\q^{\pm1}]}\{h_{1}h_{2}-h_{2}h_{1} \mid h_{1},h_{2}\in H\},
\]
as a module over $\mathbb Z[\q^{\pm 1}]$. Since every trace on $H$ vanishes on commutators, the cocenter is the natural domain for traces and characters and is therefore a fundamental object in representation theory.

Although Krammer \cite{Kr09} proved that the conjugacy problem for $W$ is algorithmically decidable, understanding the cocenter requires considerably finer structural information. Indeed, to construct canonical elements of $\overline H$ associated with conjugacy classes $\mathcal O\subseteq W$, one encounters a basic difficulty: the images of the standard basis elements $T_w$ in $\overline H$ are not invariant under arbitrary conjugation $w\mapsto v^{-1}wv$. They are, however, invariant under the more rigid relation of \emph{strong conjugation}. Thus, lifting conjugacy classes to the cocenter requires an explicit, length-controlled procedure for reducing arbitrary elements of $\mathcal O$ to its set of minimal-length representatives
$$
\mathcal O_{\min}
=
\{w\in\mathcal O\mid
\ell(w)=\min_{v\in\mathcal O}\ell(v)\}.
$$

This reduction phenomenon was first discovered for finite Weyl groups by Geck--Pfeiffer \cite{GP93} via a case-by-case analysis. He--Nie \cite{HN1, HN2} introduced the \emph{geodesic method}, providing a uniform proof for both finite and affine Weyl groups. Marquis \cite{Ma21, Ma25} applied the He--Nie geodesic function to the Davis complex to establish the reduction theorem in the general case. These results provide the following two fundamental properties:

\begin{enumerate}
\item Every element $w\in\mathcal O$ can be transformed into an element of $\mathcal O_{\min}$ through a sequence of cyclic shifts
$$
w\longmapsto sws,
\qquad
\ell(sws)\leq\ell(w).
$$
\item Any two elements of $\mathcal O_{\min}$ are connected by a sequence of strong conjugations.
\end{enumerate}

Since strong conjugation preserves the image of $T_w$ in the cocenter, the canonical element
$$
T_{\mathcal O}
:=
T_w+[H,H]\in\overline H
\qquad
(w\in\mathcal O_{\min})
$$
is well-defined and independent of the choice of $w\in\mathcal O_{\min}$. The first property further implies that the canonical elements
$$
\{T_{\mathcal O}\mid\mathcal O\in\Cl(W)\}
$$
form a natural $\mathbb Z[\q^{\pm1}]$-spanning set of $\overline H$.

The fundamental open question is whether these canonical spanning elements $\{T_{\mathcal{O}}\}_{\mathcal{O}\in \mathrm{Cl}(W)}$ are \textbf{linearly independent}, or equivalently, whether $\{T_{\mathcal{O}}\}_{\mathcal{O}\in \mathrm{Cl}(W)}$ forms a basis of $\overline{H}$. This in particular, implies that the cocenter $\overline{H}$ is a free $\mathbb{Z}[\q^{\pm 1}]$-module. Previously, this was known only in finite type \cite{GP93} and affine type \cite{HN2}. We solve this question in full generality:

\begin{alphatheorem}[Main Theorem]\label{thm:intro_main}
Let $G$ be a split Kac--Moody group with Weyl group $(W,S)$, and let $H$ be its generic Iwahori--Hecke algebra. Then the cocenter $\overline H$ is a free $\mathbb Z[\q^{\pm1}]$-module with basis $\{T_{\mathcal O}\mid\mathcal O\in\Cl(W)\}$.
\end{alphatheorem}

The freeness of $\overline{H}$ over the Laurent polynomial ring $\mathbb{Z}[\q^{\pm 1}]$ is a subtle feature of the generic deformation. If one instead considers the Hecke algebra ${H}_{\mathbb{Z}[\q]}$ over the polynomial ring $\mathbb{Z}[\q]$, its cocenter $\overline{H}_{\mathbb{Z}[\q]} = H_{\mathbb{Z}[\q]} / [H_{\mathbb{Z}[\q]}, H_{\mathbb{Z}[\q]}]$ need not be a free $\mathbb Z[\q]$-module. For example, when $W = S_3$ with simple reflections $(1\,2)$ and $(2\,3)$, one has $$\overline{H}_{\mathbb{Z}[\q]} \cong \mathbb Z[\q]^{\oplus 3}\oplus (\mathbb Z[\q] / (\q)).$$ The torsion summand is generated by $T_{(1~2)} - T_{(2~3)}$.

\subsection{Success and obstructions of the representation-theoretic approach}
To explain why establishing the linear independence of $\{T_{\mathcal{O}}\mid \mathcal{O} \in \Cl(W)\}$ for general Kac-Moody groups requires new insights, it is illuminating to review the classical trichotomy of Kac-Moody groups and the status of Theorem~\ref{thm:intro_main} in these cases.
\begin{equation*}
    \begin{array}{ccc}
    \begin{tikzpicture}[scale=0.85, baseline=0.35cm]
        \node[circle, fill, inner sep=1.5pt, label={[font=\scriptsize]below:$s_1$}] (1) at (0,0) {};
        \node[circle, fill, inner sep=1.5pt, label={[font=\scriptsize]below:$s_2$}] (2) at (1.4,0) {};
        \draw (1) -- (2);
    \end{tikzpicture}
    & \hspace{1.2cm}
    \begin{tikzpicture}[scale=0.85, baseline=0.35cm]
        \node[circle, fill, inner sep=1.5pt, label={[font=\scriptsize]below:$s_1$}] (1) at (0,0) {};
        \node[circle, fill, inner sep=1.5pt, label={[font=\scriptsize]below:$s_2$}] (2) at (1.4,0) {};
        \draw (1) -- node[above, font=\scriptsize, inner sep=1pt] {$\infty$} (2);
    \end{tikzpicture}
    & \hspace{1.2cm}
    \begin{tikzpicture}[scale=0.85, baseline=0.35cm]
        \node[circle, fill, inner sep=1.5pt, label={[font=\scriptsize]below:$s_1$}] (1) at (0,0) {};
        \node[circle, fill, inner sep=1.5pt, label={[font=\scriptsize]below:$s_2$}] (2) at (1.4,0) {};
        \node[circle, fill, inner sep=1.5pt, label={[font=\scriptsize]above:$s_3$}] (3) at (0.7,1.05) {};
        \draw (1) -- node[below, font=\scriptsize, inner sep=1pt] {$\infty$} (2);
        \draw (2) -- node[right, font=\scriptsize, inner sep=1pt] {$\infty$} (3);
        \draw (3) -- node[left, font=\scriptsize, inner sep=1pt] {$\infty$} (1);
    \end{tikzpicture} \\
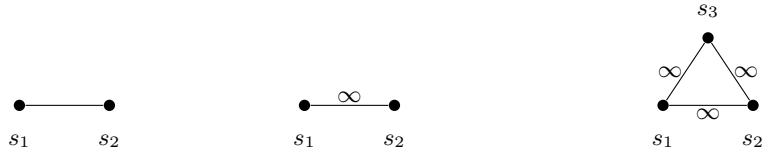
    \text{(1) Finite ($A_2$)} & \text{(2) Affine ($\widetilde{A}_1$)} & \text{(3) Indefinite (Universal rank 3)}
    \end{array}
\end{equation*}

\begin{enumerate}
    \item \textbf{Finite Type:} $G$ is a split reductive group, $W$ is finite, and the generalized Catan matrix (a generalization of the classical Cartan matrices, obtained from the Kac-Moody root datum) is positive definite (e.g., $A_2$ for $\mathrm{SL}_3$). In this setting, Tits' deformation theorem guarantees that after base change to $\BC$, $H$ becomes semisimple and isomorphic to the group algebra of $W$. With that, the linear independence of the standard generating set $\{T_{\CO}\mid \CO\in \Cl(W)\}$ of $\overline H$ follows. 
    
    Since semisimple algebras over $\BC$ are necessarily finite-dimensional, we see that $H_{\BC}$ is semisimple if and only if $W$ is finite. Hence this method has no chance to be generalized to other types.
    \item \textbf{Affine Type:} This is the case where the generalized Cartan matrix is positive semidefinite. The associated Kac-Moody group is (a central extension of) a $p$-adic group or loop group (e.g., $\widetilde{A}_1$ for $\widehat{\mathrm{SL}}_2$). Here, $H$ is no longer semisimple. Nevertheless, He and Nie \cite{HN2} established linear independence by exploiting Lusztig's asymptotic algebra ($J$-ring) and the sophisticated representation theory of affine Hecke algebras; see also Ciubotaru--He \cite{CH}.

    This approach relies on special properties of the affine setting and a very thoroughly developed theory, and is thus hardly suited for further generalization.
    \item \textbf{Indefinite Type:} This is the remaining case, where the generalized Cartan matrix is indefinite. This is the least understood case. New phenomena emerge, especially the role of imaginary roots becomes much more prominent.
\end{enumerate}

We summarize that in cases (1) and (2), Theorem \ref{thm:intro_main} was established through a \emph{representation-focused paradigm}. That is, algebraic independence is proved by evaluating the canonical basis elements against a pre-existing, comprehensive family of global representations.
For general, indefinite Kac--Moody groups and their associated infinite Coxeter groups, such an approach is infeasible.

\subsection{Overview of the proof and novel ingredients}
To overcome the absence of a generic representation theory, we develop a new \emph{orbital-integral focused architecture}. Geometrically, the naive orbital integral associated to an element $g\in G(\mathbb{F}_q)$ would be the integration
\begin{equation}\label{eq:intro_orbital_integral_geometric}
    H_{\mathbb{C}}\longrightarrow \mathbb{C}, \qquad f\longmapsto \int_{G(\mathbb{F}_q) / B(\mathbb{F}_q)} f(h^{-1} g h)\, dh.
\end{equation}
When $G$ is an infinite Kac--Moody group, such integrals often fail to converge. The purely algebraic analogue of this geometric integral is the formal trace of a bimodule endomorphism on $H$:
\begin{equation}\label{eq:intro_orbital_integral_algebraic}
    H\to \mathbb Z[\q^{\pm 1/2}],~h_1\mapsto \sum_{w\in W}(\text{$\BZ[\q^{\pm 1/2}]$-coefficient of $T_w$ in }h_1 T_w h_2),
\end{equation}
where $h_2\in H$ is a fixed element. Whenever $W$ is infinite, this expansion also yields a divergent infinite series.

The core of our strategy is to construct explicit, \emph{convergent renormalizations} of the trace functionals \eqref{eq:intro_orbital_integral_geometric} and \eqref{eq:intro_orbital_integral_algebraic} that directly separate the canonical basis elements $\{T_{\mathcal{O}} \mid \mathcal{O}\in \mathrm{Cl}(W)\}$. 

One key insight is a comparison of local versus global functionals. For any proper subset $I\subseteq S$, the standard parabolic subgroup $W_I \subseteq W$ is itself a Coxeter group with generic Hecke algebra $H_I$. We refer to linear maps $\overline{H}_I\to \mathbb{Q}[\q^{\pm 1/2}]$ as \emph{local functionals}, in contrast to \emph{global functionals} on $\overline{H}$.

\begin{itemize}
    \item \textbf{Well-Defined Parabolic Induction for Infinite Cocenters (Section \ref{sec:parabolic_induction}):} Inspired by the Harish-Chandra philosophy in finite Lie theory, we construct a normalized parabolic induction method that lifts certain local functionals $f$ to global functionals $\mathrm{Ind}_I^S(f)$. While parabolic induction on infinite groups typically produces divergent series, we prove that by imposing strict support conditions on local functionals, the induced infinite sum collapses to a well-defined \textbf{finite sum} when evaluated on any given element. This allows us to lift local trace maps to the global cocenter, reducing the separation problem to elliptic classes and spherical (finite-order) classes.
    
    \item \textbf{Elliptic Classes via Infinite-Dimensional Bimodules (Section \ref{sec:elliptic}):} For elliptic conjugacy classes, we construct infinite-dimensional $H$-bimodules $M^b$ built on Krammer's straight elements \cite{Kr09} and Marquis's normal forms \cite{Ma21, Ma25}. By analyzing the symmetry and vanishing properties of endomorphisms on $M^b$, we define explicit bimodule trace maps $\mathrm{tr}_{b,w}: \overline{H} \longrightarrow \mathbb{Z}[\q^{\pm 1}]$ that isolate elliptic conjugacy classes as convergent realizations of \eqref{eq:intro_orbital_integral_algebraic}.
    
    \item \textbf{Kac--Moody Orbital Integrals for Spherical Classes (Sections \ref{sec:orbital_integrals} \& \ref{sec:spherical}):} To separate spherical classes, we define renormalized orbital integrals on $G(\mathbb{F}_q)$ for regular semisimple classes as geometric instances of \eqref{eq:intro_orbital_integral_geometric}. We prove that these orbital integrals \emph{vanish identically on all non-spherical elements} $T_{\CO}\in \overline H$. Consequently, orbital integrals serve as the exact geometric complement to parabolic induction. Via rigid polynomial interpolation across infinitely many prime powers $q \gg 0$, these point counts assemble into generic trace maps on $\overline{H}$.
    
    \item \textbf{Surgical $E_7$ Parabolic Induction (Section \ref{sec:E7}):} We encounter an intrinsic deficiency that only arises for type $E_7$ and $E_8$ root systems (see Theorem~\ref{thm:intro_perfect_pairing} below). To overcome it, we introduce a particular renormalization of \eqref{eq:intro_orbital_integral_algebraic} that is tailor-made to lift very specific local functionals $\overline H_I\to \mathbb Z[\q^{\pm 1/2}]$ to global functionals whenever the Dynkin diagram of $I$ contains a connected component of type $E_7$.

    We remark that, while a similar deficiency occurs in type $E_8$, our parabolic induction method is already sufficient to handle these cases.
\end{itemize}

With these ingredients in place, the proof of Theorem \ref{thm:intro_main} proceeds by induction on the rank $\#S$ of $W$ (Section \ref{sec:proofOfMainThm}). More precisely, our construction yields an {\it ``almost dual basis''} for the cocenter: a family of functionals $\{F_{\mathcal{O}} : \overline{H} \longrightarrow \mathbb{Q}[\q^{\pm 1/2}] \mid \mathcal{O}\in \mathrm{Cl}(W)\}$ satisfying
\[
F_{\mathcal{O}_1}(T_{\mathcal{O}_2}) \equiv \delta_{\mathcal{O}_1,\mathcal{O}_2} \pmod{\q-1}.
\]
In other words, evaluating the matrix $[F_{\mathcal{O}_1}(T_{\mathcal{O}_2})]$ modulo $\q-1$ yields the identity matrix, whose non-zero determinant establishes that the canonical elements $\{T_{\mathcal{O}}\}$ are linearly independent over $\mathbb{Q}[\q^{\pm 1/2}]$. This immediately implies the freeness of $\overline{H}$ over $\mathbb{Z}[\q^{\pm 1}]$. 

The logical dependencies between the main sections of the proof are summarized in the following roadmap:

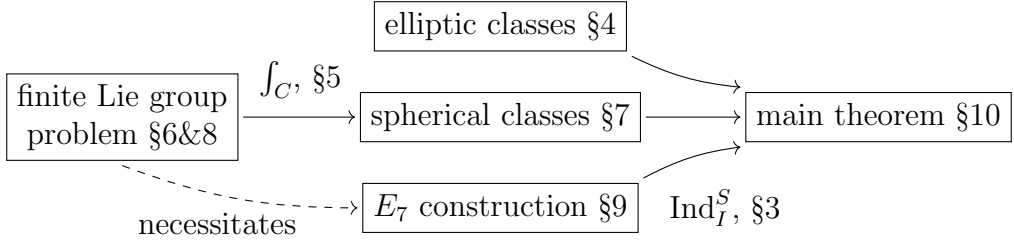
\begin{center}
    \begin{tikzpicture}[every node/.style={outer sep=2}]
        \node[align=center] (fglt) at (0,0) [draw] {finite Lie group\\ problem \S\ref{sec:lusztig_pairing}\&\ref{sec:perfect_pairing}};
        \node (sph) at (5,0) [draw] {spherical classes \S\ref{sec:spherical}};
        \node (main) at (10,0) [draw] {main theorem \S\ref{sec:proofOfMainThm}};
        \node (ell) at (5,1.2) [draw] {elliptic classes \S\ref{sec:elliptic}};
        \node (E7) at (5,-1.2) [draw] {$E_7$ construction \S\ref{sec:E7}};
        \draw[->] (sph.east) -- (main.west);
        \draw (ell.south east) edge[bend right=10,->] (main.north west);
        \draw (E7.north east) edge[bend left=10,->] (main.south west);
        \path (E7.east) node[right]{~~$\mathrm{Ind}_I^S$, \S\ref{sec:parabolic_induction}};
        \draw[->] (fglt.east) -- (sph.west);
        \path (fglt.north east) node[right,yshift=-4] {$\int_C$, \S\ref{sec:orbital_integrals}};
        \draw (fglt.south) edge[bend right=10,->,dashed] (E7.west);
        \path (fglt.south) node[below right,yshift=-12] {necessitates};
    \end{tikzpicture}
\end{center}

\subsection{Local analysis: Finite groups of Lie type}
Consider the classical setting where $W$ is a finite Weyl group and $G$ is a split connected reductive group over a finite field $\mathbb{F}_q$, with complex Iwahori--Hecke algebra $H_{\mathbb{C}}$. Evaluating orbital integrals on $G(\mathbb{F}_q)$ leads to the following question, which is of independent interest in the representation theory of finite groups of Lie type:
\begin{quote}
\emph{To what extent can the cocenter $\overline{H}_{\mathbb{C}}$ be separated by orbital integrals over regular semisimple conjugacy classes? Explicitly, what is the joint kernel of the linear functionals}
\[
H_{\mathbb{C}}\longrightarrow \mathbb{C}, \qquad h\longmapsto \sum_{g\in G(\mathbb{F}_q)/B(\mathbb{F}_q)} h(g^{-1} s g)
\]
\emph{as $s$ ranges over all regular semisimple elements of $G(\mathbb{F}_q)$?}
\end{quote}

This problem arises naturally from three distinct perspectives:
\begin{itemize}
    \item \textbf{Geometric Duality:} Kazhdan and Lusztig \cite{KL88, Lu-conj} related the conjugacy classes of $W$ to unipotent conjugacy classes in $G(\mathbb{F}_q)$. However, because the number of unipotent classes generally differs from (and is typically smaller than) $\#\mathrm{Cl}(W)$, this relation does not provide a square, non-degenerate pairing. In contrast, $G(\mathbb{F}_q)$-conjugacy classes of $\mathbb{F}_q$-rational maximal tori are in natural bijection with $\mathrm{Cl}(W)$, making regular semisimple classes the canonical candidates to yield a perfect duality.
    
    \item \textbf{Character Separation:} Following the Harish-Chandra theory of $p$-adic groups, it is natural to ask if a principal series unipotent representation of $G(\mathbb F_q)$ (which is the same as an $H_{\mathbb C}$-representation) is uniquely determined by character values on regular semisimple elements. 
    
    \item \textbf{Local Engine for Kac--Moody Separation:} Evaluating global Kac--Moody orbital integrals for regular semisimple classes reduces to point counts over finite groups of Lie type. Understanding the non-degeneracy of these local orbital integrals is essential for constructing separation functionals on the global generic cocenter.
\end{itemize}

In Sections \ref{sec:lusztig_pairing} and \ref{sec:perfect_pairing}, we answer the above question completely by establishing a generic duality theorem over $A = \mathbb{Q}[\q^{\pm 1/2}]$ that holds uniformly across all finite fields.

\begin{alphatheorem}[Cf. Theorems \ref{thm:lusztigPairingOverview} \& \ref{thm:perfectPairing}]\label{thm:intro_perfect_pairing}
Let $G/\mathbb Z$ be a simple Chevalley group with Weyl group $W$, and set $A = \mathbb{Q}[\q^{\pm 1/2}]$. Evaluating orbital integrals over regular semisimple conjugacy classes $C \subseteq G(\mathbb{F}_q)$ for finite fields $\mathbb F_q$ induces a generic, $A$-bilinear duality pairing
\[
\psi: A[\mathrm{Cl}(W)] \otimes_A A[\mathrm{Cl}(W)] \longrightarrow A.
\]
This pairing satisfies the following structural properties:
\begin{enumerate}
    \item \textbf{Perfect Duality Outside $E_7/E_8$:} The pairing $\psi$ is perfectly non-degenerate if and only if $G$ is not of type $E_7$ or $E_8$.
    \item \textbf{Exceptional Defect:} The dimension of the kernel of $\psi$ is $1$ in type $E_7$ and $2$ in type $E_8$.
\end{enumerate}
\end{alphatheorem}

The structural mechanism underlying Theorem \ref{thm:intro_perfect_pairing} reveals a deep connection with Deligne--Lusztig theory. By expressing the values of regular semisimple orbital integrals in terms of Green functions and Deligne--Lusztig characters, we show that the matrix representing $\psi$ factorizes as
\[
[\psi] = \q^{-\ell(w_0)} N F C,
\]
where $C$ is the character table of the Hecke algebra, $F$ is Lusztig's non-abelian Fourier transform matrix on unipotent character sheaves \cite{Lu-Characters}, and $N$ is an invertible matrix. 

Consequently, the non-degeneracy of $\psi$ reduces to the non-degeneracy of Lusztig's Fourier transform matrix on the relevant families of unipotent symbols. We prove this non-degeneracy on a type-by-type basis: for classical types, we establish the non-vanishing of $\det(F)$ using the combinatorics of Lusztig symbols; for exceptional types, the result is verified using the \textsc{Chevie} computer algebra system \cite{Chevie}. Finally, the explicit $1$-dimensional kernel in type $E_7$ pinpoints the precise local obstruction that necessitates the surgical truncated induction developed in Section \ref{sec:E7}.

\subsection{Applications to Kac-Moody Deligne-Lusztig varieties}

Establishing the basis property of $\{T_{\mathcal{O}} \mid \mathcal{O}\in \mathrm{Cl}(W)\}$ in Theorem \ref{thm:intro_main} has immediate, fundamental consequences for both the representation theory and geometry of Kac--Moody groups. The linear independence of the canonical elements guarantees that for any $w \in W$, the transition coefficients in the cocenter $\overline{H}$ yield uniquely determined class polynomials $f_{w, \mathcal{O}} \in \mathbb{N}[\q-1]$ satisfying
\[
T_{w} \equiv \sum_{\mathcal{O}\in \mathrm{Cl}(W)} f_{w,\mathcal{O}} T_{\mathcal{O}} \pmod{[H,H]}.
\]

When $G$ is of finite type, the (classical) Deligne--Lusztig varieties play a central role in constructing representations of $G(\mathbb F_q)$ via their $\ell$-adic cohomology. When $G$ is of affine type, the resulting \emph{affine Deligne--Lusztig varieties} are fundamental objects in arithmetic geometry and the Langlands program, modeling the mod $p$ special fibers of Shimura varieties and moduli spaces of shtukas.

While the dimension of a classical Deligne--Lusztig variety in finite type is elementary ($\dim X_w = \ell(w)$), transitioning to infinite Kac--Moody flag varieties introduces a profound shift. Due to the failure of Lang's theorem, Kac--Moody Deligne--Lusztig varieties
\[
X_w(b) = \{g B \in G(\overline{\mathbb{F}}_q) / B(\overline{\mathbb{F}}_q) \mid g^{-1} b \sigma(g) \in B(\overline{\mathbb{F}}_q) w B(\overline{\mathbb{F}}_q)\}
\]
are indexed by two parameters $w\in W$ and $b\in G(\overline{\mathbb{F}}_q)$, where $\sigma$ is the Frobenius endomorphism. Even for the basic case $b=1$, these varieties can be empty, and their geometric dimensions are highly non-trivial. In the affine setting, the first author \cite{He-Ann, He-ICM} discovered that this geometric complexity is precisely captured by the cocenter, establishing the ``dimension=degree'' theorem that computes the dimension of $X_w(1)$ from the degrees of the class polynomials.

With the unique existence of generic class polynomials now established by Theorem \ref{thm:intro_main}, we extend this ``dimension=degree'' correspondence to split Kac--Moody flag varieties over finite fields \cite{Ka20}.

\begin{alphatheorem}[Cf. Theorem 11.3]\label{thm:intro_dim_degree}
Let $w\in W$, and let $f_{w,[1]} \in \mathbb{N}[\q-1]$ be the polynomial defined by
\[
f_{w,[1]} = \sum_{\mathcal{O} \text{ finite order}} \q^{\ell(\mathcal{O})} f_{w,\mathcal{O}},
\]
where the sum is taken over all conjugacy classes of finite-order elements in $W$, and $\ell(\mathcal{O})=\min\{\ell(v) \mid v\in\mathcal{O}\}$. Then the dimension of the basic Kac--Moody Deligne--Lusztig variety $X_w(1)$ is explicitly given by
\[
\dim X_{w}(1) = \deg_{\q} f_{w,[1]},
\]
with the convention that $\dim \emptyset = -\infty = \deg_{\q} 0$.
\end{alphatheorem}

Theorem \ref{thm:intro_dim_degree} extends the dimension formulas of \cite{He-Ann, He-ICM} from affine flag varieties to the full Kac--Moody setting, demonstrating that the algebraic structure of the Hecke algebra cocenter directly captures the intricate geometry of Kac--Moody flag varieties. We conclude our article by pointing out fruitful further research directions.

{\bf Acknowledgements} XH is partially supported by the New Cornerstone Science Foundation through the New Cornerstone Investigator Program. Part of this project was conducted while FS was employed at the University of Hong Kong, during which he also was partially supported by the New Cornerstone Science Foundation through the New Cornerstone Investigator Program awarded to XH. We thank George Lusztig, Jean Michel and Dinakar Mutiah for useful discussions. 

\section{Hecke Algebra and its cocenter} \label{sec:hecke_algebra}

\subsection{Weyl groups and Coxeter combinatorics}

Throughout this paper, let $G$ be a split Kac--Moody group
associated with a fixed Kac--Moody root datum, and let $(W,S)$
be its Weyl group. Thus $(W,S)$ is a crystallographic Coxeter
system. We use the standard Coxeter-theoretic notation recalled
below. 

The {\it Coxeter matrix} associated with the Kac--Moody root datum is a matrix $(m_{s, s'})_{(s, s') \in S \times S}$ with entries in $\{2, 3, 4, 6, \infty\}$ such that $m_{s, s}=1$ for all $s \in S$ and $m_{s, s'}=m_{s', s} \ge 2$ for all $s \neq s'$. The associated {\it Coxeter group} $W$ is the group defined by the presentation with generators $S$ and relations $(s s')^{m_{s, s'}}=1$ for all $s, s' \in S$ with $m_{s, s'}<\infty$. The set $\Cl(W)$ denotes the set of conjugacy classes of $W$. 

The elements in $S$ are called the {\it simple reflections} of $W$. For any $w \in W$, the {\it length} $\ell(w)$ of $w$ is the smallest integer $l$ such that $w$ can be written as a product of $l$ simple reflections. Any expression $w=s_{1} \cdots s_{\ell(w)}$ with $s_{1}, \ldots, s_{\ell(w)} \in S$ is called a {\it reduced expression} of $w$. In this case, we write $w'\leq w$ for $w'\in W$ if there is a (reduced) subexpression $w' = s_{i_1} \cdots s_{i_\ell}$ with $1\leq i_1<\cdots<i_\ell\leq \ell(w)$. This partial order is known as \emph{Bruhat order}.

For any subset $I$ of $S$, let $W_J$ be the parabolic subgroup of $W$ generated by $I$. Given subsets $I,J\subseteq S$, we write 
\begin{align*}
{}^I W^J := \{w\in W\mid (\forall s\in I:~sw>w)\text{ and }(\forall s\in J:~ ws>w)\}.
\end{align*}
Each double coset $W_I\setminus W/W_J$ contains exactly one element in ${}^I W^J$, which is its unique minimal length representative. If $I=\emptyset$, we write $W^J := {}^\emptyset W^J$. In this case, each element $w\in W$ has a unique decomposition $w = w^J w_J$ with $w^J\in W^J$ and $w_J\in W_J$. This product is length additive.

A subset $J\subseteq S$ is called \emph{spherical} if $W_J$ is finite. This is equivalent to the condition that there exists some $w\in W$ with $J = \{s\in S\mid ws<w\}$, and also equivalent to the claim that $W_J$ contains an element of maximal length. Such a maximal length element is denoted $w_{0,J}$ and is uniquely determined. It satisfies $w_{0,J}^2 = 1$ and conjugation by $w_{0,J}$ induces an automorphism of the set $J$.

We call a Coxeter group $(W, S)$ \emph{crystallographic} if for all $s, s'\in S$, the order of $ss'\in W$ lies in $\{1,2,3,4,6,\infty\}$. The crystallographic Coxeter groups arise as the Weyl groups of Kac-Moody groups. The finite and affine Weyl groups are examples of crystallographic Coxeter groups. 

\subsection{Root systems}
Such crystallographic Coxeter groups arise naturally out of Kac-Moody root data
\[
(X_\ast, X^\ast, \{\alpha_s\}_{s\in S}, \{\alpha^\vee_s\}_{s\in S}),
\] consisting of finitely generated free abelian groups $X_\ast, X^\ast$ equipped with a perfect pairing $\langle \cdot,\cdot\rangle : X_\ast \otimes X^\ast\to\mathbb Z$ together with simple roots $0\neq \alpha_s\in X^\ast$ and simple coroots $0\neq \alpha_s^\vee\in X_\ast$ for all $s\in S$ satisfying: 
\begin{enumerate}
    \item $\langle \alpha_s^\vee,\alpha_s\rangle=2$ for all $s\in S$ and
    \item $\langle \alpha_s^\vee,\alpha_{s'}\rangle \leq 0$ for all distinct $s,s'\in S$.
\end{enumerate}
Then the group generated by the reflections $\rho_s : X^\ast\to X^\ast,~v\mapsto v-\langle \alpha_s^\vee,v\rangle \alpha_s$ is naturally identified with a crystallographic Coxeter group $(W, S)$ in a way that identifies $\rho_s$ with $s\in S$.

This group $W$ also acts on $X_\ast$ and $V = X^\ast\otimes\mathbb R$ is the Tits geometric representation of $(W, S)$ with integral structure constants. The set of (real) roots is $\Phi=W \cdot \{\a_s\}_{s \in S}$. Each root $\alpha\in \Phi$ can be written as $\alpha = \sum_{s\in S} c_s \alpha_s$ with all $c_s$ non-negative (positive roots) or all $c_s$ negative (negative roots). Positive roots are in bijection with the \emph{reflections} of $W$ (the conjugates of simple reflections) by sending the reflection $w s w^{-1}$ for $w\in W^{\{s\}}, s\in S$ to $w\alpha_s$. 

We write $\Phi^+$ for the set of positive roots and $\Phi^-$ for the set of negative roots. For $J\subseteq S$, set
\begin{align*}
    \Phi_J := \Phi\cap \sum_{s\in J} \mathbb Z\alpha_s,\quad
    \Phi_J^{\pm} := \Phi_J\cap \Phi^{\pm}.
\end{align*}

\subsection{Kac-Moody groups}\label{subsec:kac-moody}
Recall that we have chosen a Kac-Moody root datum $\CD=(X^\ast, X_\ast, \{\alpha_s\}_{s\in S}, \{\alpha_s^\vee\}_{s\in S})$. The matrix $A = (\langle \alpha_s^\vee,\alpha_{s'}\rangle)_{s,s'\in S}$ is a \emph{generalized Cartan matrix} in the sense of \cite[Definition~3.7]{Ma18}. 

Let $\kk$ be any field. Set $T(\kk)=X^\ast \otimes_{\BZ} \kk^\times$. Let $G^{\min}$ be the minimal Kac-Moody group over $\kk$ associated to the Kac-Moody root datum $\CD$ (see \cite[Definition 7.47]{Ma18}). Roughly speaking, this means that $G(\kk)$ is generated by the split torus $T(\kk)$ and the root subgroups $U_\a(\kk) \cong \kk$ for all $\a\in \Phi$, subject to the usual Serre relations we know from groups of Lie type.

Let $U^{\min}(\kk)$ (resp. $U^-(\kk)$) be the subgroup of $G^{\min}(\kk)$ generated by $U_\a(\kk)$ for $\a \in \Phi_+$ (resp. $\a \in \Phi_-$). Let $B^{\min}(\kk)=T(\kk) U^{\min}(\kk)$ be the positive Borel subgroup. Let $N_{G^{\min}(\kk)}(T(\kk))$ be the normalizer of $T(\kk)$ in $G^{\min}(\kk)$. Then $N_{G^{\min}(k)}(T(k))/T(k)$ is the Weyl group of $G^{\min}(k)$, which is naturally isomorphic to the crystallographic Coxeter group $W$ we started with. By \cite[Corollary 7.70]{Ma18}, $(G^{\min}(\kk), N_{G^{\min}(\kk)}(T(\kk)), U^{\min}(\kk), U^-(\kk), T(\kk), S)$ is a symmetric refined Tits system in the sense of \cite[Definition B.30]{Ma18}. 

We similarly get a \emph{maximal} Kac-Moody group $G^{\max}(\kk)$, essentially obtained from $G^{\min}(\kk)$ by replacing the subgroups $B^{\min}(\kk)$ and $U^{\min}(\kk)$ by their profinite completions $B^{\max}(\kk)$ and $U^{\max}(\kk)$ (see \cite[Definition 8.13]{Ma18} for details).  Then by \cite[Proposition 8.15]{Ma18}, we get another symmetric refined Tits system $(G^{\max}(\kk), N_{G^{\min}(\kk)}(T(\kk)) = N_{G^{\max}(\kk)}(T(\kk)), U^{\max}(\kk), U^-(\kk), T(\kk), S)$.

In particular, for $\bullet \in \{\min,\max\}$, we get the Bruhat decomposition
\begin{align*}
    G^\bullet(\kk) = \bigsqcup_{w\in W} B^{\bullet}(\kk) w B^{\bullet}(\kk).
\end{align*}
These Schubert cells $B^\bullet(\kk) w B^\bullet(\kk)$ satisfy the usual relations from BN-pairs.

Both minimal and maximal Kac-Moody groups have their merits, as is explained in Marquis' book \cite{Ma18}, and both notions are studied with similar interest in the literature.

\subsection{Hecke algebras}\label{sec:Hecke} Let $H$ be the {\it Hecke algebra} of a Coxeter group $W$. It is the $\BZ[\q^{\pm 1}]$-algebra defined by the generators $T_s$ for $s \in S$ and the relations 
\begin{itemize}
    \item $(T_s-\q)(T_s+1)=0$ for $s \in S$; 
    \item $T_s T_{s'} T_s \cdots=T_{s'} T_s T_{s'} \cdots$ (both products have $m_{s, s'}$ factors) for $s \neq s'$ in $S$ with $m_{s, s'}<\infty$. 
\end{itemize}

For $w \in W$, we define $T_w=T_{s_{i_1}} \cdots T_{s_{i_l}} \in H$, where $w=s_{i_1} \cdots s_{i_l}$ is a reduced expression of $w$. It is known that $T_w$ is independent of the choice of reduced expressions of $w$ (see \cite[\S 3.2]{Lu-Hecke}). Moreover, $T_w$ is invertible in $H$ for all $w \in W$. 

Given a finite field $\kk = \mathbb F_q$ and a Kac-Moody group $G^\bullet(\kk)$ (with $\bullet\in \{\min,\max\}$), then we define its Hecke algebra to be $H_{G^\bullet(\kk)} = \mathbb C[B^\bullet(\kk)\backslash G^\bullet(\kk) / B^\bullet(\kk)]$. Explicitly, $H_{G^\bullet(\kk)}$ consists of all functions $f : G^\bullet(\kk)\to \mathbb C$ which satisfy $f(b_1 g b_2) = f(g)$ for all $g\in G^\bullet(\kk), b_1, b_2\in B^\bullet(\kk)$ and are supported on finitely many $B^\bullet(\kk)$-double cosets. We turn $H_{G^{\bullet}(\kk)}$ into an algebra by defining the convolution product of two such functions $f_1, f_2\in H_{G^\bullet}(\kk)$ to be the function
\begin{align*}
    (f_1 f_2) : G^\bullet(\kk) \to \mathbb C,~g\mapsto \sum_{h\in G^\bullet(\kk) / B^\bullet(\kk)} f_1(gh) f_2(h^{-1}).
\end{align*}
This yields an associative and unital algebra over $\mathbb C$.

Let $H_{\mathbb C}$ denote the base change of our Coxeter-theoretic Hecke algebra $H$ along the ring homomorphism $\BZ[\q^{\pm 1}]\to \mathbb C$, sending $\q$ to $q = \#\kk$. Then we get an isomorphism of $\mathbb C$-algebras $H_{\mathbb C}\to H_{G^\bullet(\kk)}$, which sends $T_w$ for $w\in W$ to the characteristic function of $B^\bullet(\kk) w B^\bullet(\kk)$. In particular, the minimal and maximal Kac-Moody groups have the same Hecke algebra. This allows us to specialize to maximal Kac-Moody groups later.

Let $[H, H]$ be the $\BZ[\q^{\pm 1}]$-submodule of $H$ generated by $[h, h']:=h h'-h' h$ for $h, h' \in H$. The {\it cocenter} $\bar H$ of the Hecke algebra $H$ is defined to be $\bar H=H/[H, H]$.

Define a linear map $\tau: H \to \BZ[\q^{\pm 1}]$ by $\t(T_1)=1$ and $\t(T_w)=0$ for $w \neq 1$ in $W$. It is known (see, e.g., \cite[Proposition 8.1.1]{GP-book}) that $$\t(T_w T_{w'})=\begin{cases} \q^{\ell(w)}, & \text{ if } w\i=w'; \\ 0, & \text{ otherwise }. \end{cases}$$

Then $\t$ induces a linear map $\bar H \to \BZ[\q^{\pm 1}]$, which we still denote by $\t$. We call it the {\it trace function}. One may regard it as a natural generalization of the trace for the group algebra (see \cite[Corollary 8.4.7]{GP-book}. Any element $h \in H$ can be written as a linear combination of the $T_{w}$: $$h=\sum_{w \in W} \q^{-\ell(w)} \t(T_{w \i} h) T_w.$$

We also have that $$\t(h_1 h_2)=\sum_{w \in W} \q^{-\ell(w)} \t(h_1 T_2) \t(T_{w \i} h_2).$$

\subsection{Spanning set of the cocenter} Let $w, w' \in W$ and $s \in S$. We write $w \to_s w'$ if $w'=s w s$ and $\ell(w') \le \ell(w)$. Let $w, w' \in W$. We write $w \to w'$ if there exists a sequence $w=w_1, w_2, \ldots, w_n=w'$ of elements of $W$ such that for any $1 \le i\le n-1$, $w_i \to_{s_i} w_{i+1}$ for some $s_i \in S$. We write $w \approx w'$ if $w \to w'$ and $w' \to w$. 

Let $w, w', x \in W$. We write $w \sim_x w'$ if $w'=x w x \i$, $\ell(w)=\ell(w')$ and we have 
\begin{itemize}
\item $\ell(x w)=\ell(x)+\ell(w)=\ell(w' x)$, or 

\item $\ell(w x \i)=\ell(x)+\ell(w)=\ell(x \i w')$. 
\end{itemize}
 
We write $w \sim w'$ if there exists a sequence $w=w_1, w_2, \ldots, w_n=w'$ of elements of $W$ such that for any $1 \le i\le n-1$, $w_i \to_{x_i} w_{i+1}$ for some $x_i \in W$. 

It is easy to see that $w \approx w' \Rightarrow w \sim w' \Rightarrow \ell(w)=\ell(w')$. 

\smallskip

The following remarkable property was first discovered by Geck and Pfeiffer \cite{GP93} for finite Weyl groups via a case-by-case analysis. A conceptual proof for finite and affine Weyl groups was obtained by the first named author together with Sian Nie \cite{HN2}. The general case was proved by Marquis in \cite{Ma21}. 
 
 \begin{theorem}\label{min}
 Let $\CO$ be a conjugacy class of $W$. Then 
 
 (1) For any $w \in W$, there exists $w' \in \CO_{\min}$ such that $w \to w'$. 
 
 (2) If $w, w' \in \CO_{\min}$, then $w \sim w'$. 
 \end{theorem}
 
 Now we discuss some consequences on the cocenter of the Hecke algebra $H$. 

 \begin{corollary}
 Let $\CO \in \Cl(W)$. Then for any $w, w' \in \CO_{\min}$, the images of $T_w$ and $T_{w'}$ in $\bar H$ coincide. 
 \end{corollary}
 \begin{proof}
     It suffices to show that $T_w - T_{w'}\in \bar H$ whenever $w\sim_x w'$. This means $\ell(xw) = \ell(x)+\ell(w) = \ell(w'x)$ or $\ell(wx^{-1}) = \ell(x)+\ell(w) = \ell(x^{-1} w')$. In the first case, we get
     \begin{align*}
         T_w = T_x^{-1} T_{xw} \equiv T_{xw} T_x^{-1} = T_{w'x} T_x^{-1} = T_{w'}\pmod{[H, H]}.
     \end{align*}
     We argue similarly in the second case.
 \end{proof}
 
 \begin{definition}\label{def:to}
 We denote the common image of all $T_w$ in $\bar H$ for $w\in\mathcal O_{\min}$ by $T_\CO$. 
 \end{definition}

\begin{proposition}\label{prop:cocenterSpanning}
The cocenter $\bar H$ is spanned by $\{T_{\CO}\}_{\CO \in \Cl(W)}$.
\end{proposition}
\begin{proof}
    Let $M\subseteq \bar H$ be the $\mathbb Z[\q^{\pm 1}]$-submodule spanned by the set $\{T_{\CO}\}_{\CO \in \Cl(W)}$. For $w\in W$, we denote the image of $T_w\in H$ in $\overline H$ by $T_w$ again. It remains to show that $T_w\in M$ for all $w\in W$. We prove this using induction on $\ell(w)$.

    If $w$ has minimal length in its conjugacy class $\CO\in \Cl(W)$, then $T_w = T_{\CO}$ in $\overline H$, so we are done by definition of $M$.

    Suppose now that $w$ does not have minimal length in its conjugacy class, and that the claim is proved for all elements of smaller length. By Theorem~\ref{min}, we find a sequence
    \begin{align*}
        w = w_1 \rightarrow_{s_1} w_2\rightarrow_{s_2}\cdots\rightarrow_{s_{n-1}} w_n \rightarrow_{s_n} w_{n+1}
    \end{align*}
    with $\ell(w_1) = \ell(w_2) = \cdots = \ell(w_{n}) > \ell(w_{n+1})$. Then, as we saw above, $T_w = T_{w_{n}}$ in $\overline H$. We calculate in $\overline H$ that
    \begin{align*}
        T_{w_n} &= T_{s_n w_{n+1} s_n} = T_{s_n} T_{w_{n+1}} T_{s_n}
        \\&= T_{s_n}^2 T_{w_{n+1}} = \q T_{w_{n+1}} + (\q-1) T_{s_n w_{n+1}}.
    \end{align*}
    Observe that both $\ell(s_n w_{n+1}) \leq 1+\ell(w_{n+1}) = \ell(w)-1$. By the inductive assumption, both $T_{w_{n+1}}$ and $T_{s_n w_{n+1}}$ lie in $M$. This finishes the induction and the proof.
\end{proof}

\section{Parabolic Induction}\label{sec:parabolic_induction}

The purpose of this section is to develop a powerful method for constructing linear functionals on the cocenter $\bar{H}$. The idea, inspired by parabolic induction for finite groups of Lie type, allows us to ``inflate" \textbf{certain} functions $f$ defined on the cocenter of a parabolic subalgebra $\bar{H}_I$ to a function $F$ on the whole cocenter $\bar{H}$, with precise control over its support. This will be a crucial inductive tool in our proof of the main theorem.

\subsection{Motivation: The finite group case} 
To motivate the construction, recall the parabolic induction for finite groups of Lie type. 

We consider a split reductive group $G/\mathbb{F}_q$. Let $P = MU_P$ be a rational standard parabolic subgroup containing a fixed Borel subgroup $B = TU_B$ with maximal torus $T \subseteq M$. Consider the parabolic induction functor $R_M^G: \mathbb{C}[M(\mathbb{F}_q)]\text{-Mod} \to \mathbb{C}[G(\mathbb{F}_q)]\text{-Mod}$. Let $\mathcal H \cong \End_{\mathbb C[G(\mathbb F_q)]\text{-Mod}}(\mathbb C[G(\mathbb F_q)/B(\mathbb F_q)])$ be the Hecke algebra of $G$ and $\mathrm{Rep}_G^B$ denote the category of $\mathbb C[G(\mathbb F_q)]$-representations which are generated by their $B(\mathbb F_q)$-fixed vectors. We get an equivalence of categories
\begin{align*}
    \mathrm{Rep}_G^B \to \mathcal H\text{-Mod},~V\mapsto V^B.
\end{align*}
Now, consider the composition of functors
\begin{align*}
\mathcal H_M\text{-Mod} \xrightarrow\sim \mathrm{Rep}_M^{M\cap B}\hookrightarrow \mathbb C[M(\mathbb F_q)]\text{-Mod}\xrightarrow{R_M^G} \mathbb C[G(\mathbb F_q)]\text{-Mod}\xrightarrow{V\mapsto V^B} \mathcal H\text{-Mod}.
\end{align*}
A direct calculation shows that for any finite-dimensional $\mathcal{H}_M$-module $V$ and any $h \in \mathcal{H}$, the trace of $h$ on the image of $V$ under this composition is given by
\begin{align*}
    \sum_{\substack{w_1\in W^M\\ w_2\in W_M}} \tr(T_{w_2}\mid V) q^{-\ell(w_1)-\ell(w_2)} \tau(T_{(w_1 w_2)^{-1}} h T_{w_1}).
\end{align*}
Here, $\tau$ is the standard trace function on $\mathcal{H}$ (see \S\ref{sec:Hecke}), $W_M$ is the Weyl group of $M$ and $W^M$ is the set of minimal length representatives of $W/W_M$. This formula serves as the blueprint for our general construction in this section.

\subsection{Parabolic induction for Hecke algebras}\label{sec:par-ind}
We now extend this construction to our setting. Let $J \subseteq I \subseteq S$ such that the normalizer $N_W(W_J)$ is contained in $W_I$. This condition ensures that the Levi subgroup we are inducing from is ``well-behaved" inside $W_I$.

Let $f: \overline{H}_I \longrightarrow \mathbb{Z}[\q^{\pm 1}]$ be a linear functional that vanishes on the commutator $[H_I, H_I]$. We will view $f$ as a function on $H_I$ that factors through the cocenter. Our goal is to build a function $F: H \to \mathbb{Z}[\q^{\pm 1}]$ from $f$.

Consider the following candidate: 
\def\indFunctionName{{\Ind^S_I(f)}}
\begin{equation}\label{eq:defining-F}
    \indFunctionName : H\to\mathbb Z[\q^{\pm 1}],\qquad h\mapsto \sum_{\substack{v_1\in W^I\\ v_2\in W_I}} f(T_{v_2}) \q^{-\ell(v_1)-\ell(v_2)} \tau(T_{(v_1 v_2)^{-1}} h T_{v_1})
\end{equation}

The sum is over all $v_1$ in the set of minimal length representatives of $W/W_I$ and $v_2 \in W_I$.

The main result of this section is the following theorem, which is a direct analogue of the classical trace formula above.

\begin{theorem}\label{thm:classPolynomialLifting} Let $J \subseteq I \subseteq S$ such that the normalizer $N_W(W_J)$ is contained in $W_I$. Let $f: \overline{H}_I \longrightarrow \mathbb{Z}[\q^{\pm 1}]$. Assume furthermore that for any $J' \subset I$ that does not contain a $W_I$-conjugate of $J$, we have $f(T_w)=0$ for any $w \in W_{J'}$. Then 
\begin{enumerate}
\item (Well-definedness) For any $h \in H$, the sum defining $\indFunctionName(h)$ has only finitely many non-zero terms. Thus $\indFunctionName: H \to \mathbb{Z}[\q^{\pm 1}]$ in \eqref{eq:defining-F} is a well-defined map.
\item (Descent to the cocenter) $\indFunctionName$ vanishes on the commutator $[H, H]$, hence induces a linear map $\bar{H} \to \mathbb{Z}[\q^{\pm 1}]$, which we still denote by $\indFunctionName$.
\item (Compatibility with $f$) If, in addition, for every subset $K \subseteq S$ that is $W$-conjugate to $J$, we have $N_W(K) \subseteq W_I$, then the restriction of $\indFunctionName$ to $H_I$ coincides with $f$.
\item (Support) If $K \subseteq S$ is such that no subset of $K$ is $W$-conjugate to $J$, then $\indFunctionName(T_w) = 0$ for all $w \in W_K$.
\item ($\q=1$ specialization) For any $\CO\in \Cl(W)$, we get
\begin{align*}
    \indFunctionName(T_{\CO}) \equiv \sum_{\substack{\CO'\in \Cl(W_I)\\ \CO'\subseteq \CO}}c_{\CO'}f(T_{\CO'})\pmod{\q-1}
\end{align*}
for some constants $c_{\CO'}\in\mathbb Z_{>0}$ depending on $\CO'\subseteq \CO$.
\end{enumerate}
\end{theorem}

The remainder of this section is devoted to the proof of this theorem. Let $J, I, N_W(J)$ and $f$ be as in the theorem statement.

\subsection{Finiteness and well-definedness of $F$}
\label{subsec:finiteness}

We start with the following lemma. 

\begin{lemma}\label{lem:liftingFiniteness1}
Let $w\in W,v_1\in W^I$ and $v_2\in W_I$ with $\tau(T_{(v_1 v_2)^{-1}} T_w T_{v_1}) \neq 0$. Then
    \begin{enumerate}
        \item $v_1 v_2 v_1 \i \le w$;
        \item $v_1 s v_1 \i \le w$ for any $s \in \supp(v_2)$.
    \end{enumerate}
\end{lemma}
\begin{proof}
    By the definition of the Hecke algebra, we see that $T_w T_{v_1}$ is a linear combination of terms $T_{zv_1}$ for $z\leq w$. So the condition $\tau(T_{(v_1 v_2)^{-1}} T_w T_{v_1})\neq 0$ implies that $v_1 v_2 = z v_1$ for some $z\leq w$. Comparing lengths and using that $v_1 v_2$ is necessarily a length-additive product, we get the claim (1).

    Let $u_1\lessdot u_2$ be a Bruhat cover of elements in $W_I$. Then the condition $v_1\in W^I$ implies that $v_1 u_1 v_1^{-1} < v_1 u_2 v_1^{-1}$. In other words, conjugation by $v_1$ preserves the Bruhat order in $W_I$. Hence $v_1 s v_1^{-1} \leq v_1 v_2 v_1^{-1} = z\leq w$. This proves part (2). 
\end{proof}

We now prove that the sum defining $\indFunctionName(h)$ is finite for every $h\in H$, thereby showing that $\indFunctionName$ is a well-defined map. By linearity it suffices to consider $h=T_w$ for some $w\in W$.  Set
\[
\mathcal{S}_w:=\bigl\{(v_1,v_2)\in W^{I}\times W_{I}\;\big|\;
f(T_{v_2})\neq 0\;\text{and}\;\tau(T_{(v_1v_2)^{-1}}T_wT_{v_1})\neq0\bigr\}.
\]

By Lemma \ref{lem:liftingFiniteness1}, there exists $z\leq w$ such that $v_1v_2=z\,v_1$.  We obtain
\[
\ell(v_2)=\ell(v_1v_2)-\ell(v_1)=\ell(z v_1)-\ell(v_1) \le \ell(z)\le\ell(w).
\]
Thus for every $(v_1,v_2)\in\mathcal{S}_w$, we have $\ell(v_2)\le\ell(w)$.  Hence, the projection of $\mathcal{S}_w$ onto the second factor is contained in the finite set
\[
V_2:=\{v_2\in W_{I}\mid\ell(v_2)\le\ell(w)\}.
\]

Fix $v_2\in V_2$ such that $(v_1, v_2) \in \mathcal{S}_w$ for some $v_1$. In particular, $f(T_{v_2})\neq0$.  By the hypothesis of Theorem~\ref{thm:classPolynomialLifting}, the support $\supp(v_2)$ contains a subset $J'$ which is $W_{I}$‑conjugate to $J$.  

By Lemma \ref{lem:liftingFiniteness1} (2), for any simple reflection $s\in J'$ one has $v_1s v_1^{-1}\le w$.  Consequently
$v_1$ belongs to the set
\[
X_{v_2}:=\{v_1\in W^{I}\;\big|\; \ell(v_1s v_1^{-1})\le\ell(w) \text{ for all } s \in J'\}.
\]

Consider the map
\[
\Phi\colon W^{I}\longrightarrow\Hom_{\mathrm{Set}}(J',W),\qquad
v_1\longmapsto\bigl(s\mapsto v_1s v_1^{-1}\bigr).
\]
If $\Phi(v_1)=\Phi(v_1')$ then $v_1^{-1}v_1'\in N_W(W_{J'})$.  By the hypothesis of Theorem~\ref{thm:classPolynomialLifting} we have $N_W(W_{J'})\subseteq W_I$; because $v_1,v_1'\in W^{I}$ this forces $v_1=v_1'$.  Hence $\Phi$ is injective.

The set $J'$ is finite and $W$ contains only finitely many elements of length $\le\ell(w)$.  Therefore, there are only finitely many functions $c\colon J'\to W$ with $\ell(c(s))\le\ell(w)$ for all $s\in J'$.  By injectivity of $\Phi$, its preimage $X_{v_2}$ (which is exactly $\Phi^{-1}$ of that finite set of functions) is finite.

Therefore, $\mathcal{S}_w\subseteq\bigcup_{v_2\in V_2}\bigl(X_{v_2}\times\{v_2\}\bigr)$ is a finite union of finite sets.  Hence $\mathcal{S}_w$ is finite, completing the proof that $F$ is well‑defined. 

\subsection{Descent to the cocenter}
Now we establish part (2) of Theorem \ref{thm:classPolynomialLifting}.

It is enough to prove $\indFunctionName(h_1h_2)=\indFunctionName(h_2h_1)$ for all $h_1,h_2\in H$.  By definition, $\indFunctionName(h_1 h_2)=\sum_{v_1\in W^{I}}\sum_{v_2\in W_{I}} \q^{-\ell(v_1)-\ell(v_2)} \tau(T_{(v_1 v_2)^{-1}} h_1 h_2 T_{v_1}) f(T_{v_2})$. We expand the product $h_2 T_{v_1}$ in the basis $\{T_{u_1u_2}\mid u_1\in W^{I},\;u_2\in W_{I}\}$ and obtain $$h_2 T_{v_1}=\sum_{u_1\in W^{I}}\sum_{u_2\in W_{I}} \q^{-\ell(u_1)-\ell(u_2)}\,
\tau(T_{(u_1u_2)^{-1}}h_2 T_{v_1})\,T_{u_1u_2}.$$  

Since $\t$ vanishes on $[H, H]$, we have
\begin{align*} &\tau(T_{(v_1 v_2)^{-1}} h_1 h_2 T_{v_1}) \\&=\sum_{u_1\in W^{I}}\sum_{u_2\in W_{I}} \q^{-\ell(u_1)-\ell(u_2)}\, \tau(T_{(v_1 v_2)^{-1}} h_1 T_{u_1} T_{u_2}) \tau(T_{(u_1u_2)^{-1}}h_2 T_{v_1}) \\ &=\sum_{u_1\in W^{I}}\sum_{u_2\in W_{I}} \q^{-\ell(u_1)-\ell(u_2)}\, \tau(T_{v_1^{-1}} h_1 T_{u_1} T_{u_2} T_{v_2 \i}) \tau(T_{(u_1u_2)^{-1}}h_2 T_{v_1}).
\end{align*}

We have $\tau(T_{v_1^{-1}} h_1 T_{u_1} T_{u_2} T_{v_2 \i})=\sum_{z\in W_{I}} \q^{-\ell(z)}\, \tau(T_{v_1^{-1}} h_1 T_{u_1} T_z) \tau(T_{z \i} T_{u_2} T_{v_2 \i})$.

Since $f$ vanishes on $[H_I, H_I]$, we have 
\begin{align*}
   \sum_{v_2\in W_{I}} \q^{-\ell(v_2)} \tau(T_{z \i} T_{u_2} T_{v_2 \i}) f({T_{v_2}}) &=f(\sum_{v_2\in W_{I}} \q^{-\ell(v_2)}  \tau(T_{z \i} T_{u_2} T_{v_2 \i}) T_{v_2}) \\ &=f(T_{z \i} T_{u_2})=f({T_{u_2} T_{z \i}}) \\ &=f(\sum_{v_2\in W_{I}} \q^{-\ell(v_2)}  \tau( T_{u_2} T_{z \i}T_{v_2 \i}) T_{v_2}) \\&=\sum_{v_2\in W_{I}}  \q^{-\ell(v_2)} \tau(T_{u_2} T_{z \i}T_{v_2 \i}) f(T_{v_2}).
\end{align*}

Now we reverse the procedure. We have 
\begin{align*}
\t(T_{u_1 \i} h_2 T_{v_1} T_{z \i} T_{v_2 \i}) &=\sum_{u_2 \in W_I} \q^{-\ell(u_2)} \tau(T_{u_2} T_{z \i}T_{v_2 \i}) \t(T_{u_1 \i} h_2 T_{v_1} T_{u_2 \i}) \\&=\sum_{u_2 \in W_I} \q^{-\ell(u_2)} \tau(T_{u_2} T_{z \i}T_{v_2 \i}) \t(T_{u_2 \i} T_{u_1 \i} h_2 T_{v_1}).
\end{align*}
and
\begin{align*}
\t(T_{(u_1 v_2) \i} h_2 h_1 T_{u_1}) &=\sum_{v_1 \in W^I} \sum_{z \in W_I} \q^{-\ell(v_1)-\ell(z)}\t(T_{(u_1 v_2) \i} h_2 T_{v_1} T_{z \i}) \t(T_z T_{v_1 \i} h_1 T_{u_1}) \\ &=\sum_{v_1 \in W^I} \sum_{z \in W_I} \q^{-\ell(v_1)-\ell(z)}\t(T_{u_1\i} h_2 T_{v_1} T_{z \i} T_{v_2 \i}) \t(T_{v_1 \i} h_1 T_{u_1} T_z).
\end{align*}

Putting all these equalities together, we have
\begin{align*}
    & \indFunctionName(h_1 h_2)=\sum \q^{-\ell(v_1)-\ell(v_2)} \tau(T_{(v_1 v_2)^{-1}} h_1 h_2 T_{v_1}) f(T_{v_2}) \\ &=\sum \q^{-\ell(v_1)-\ell(v_2)-\ell(u_1)-\ell(u_2)} \tau(T_{v_1^{-1}} h_1 T_{u_1} T_{u_2} T_{v_2 \i}) \tau(T_{(u_1u_2)^{-1}}h_2 T_{v_1}) f(T_{v_2}) \\&=\sum \q^{-\ell(v_1)-\ell(v_2)-\ell(u_1)-\ell(u_2)-\ell(z)} \tau(T_{v_1^{-1}} h_1 T_{u_1} T_z) \tau(T_{z \i} T_{u_2} T_{v_2 \i}) \tau(T_{(u_1u_2)^{-1}}h_2 T_{v_1}) f(T_{v_2}) \\ &=\sum \q^{-\ell(v_1)-\ell(v_2)-\ell(u_1)-\ell(u_2)-\ell(z)} \tau(T_{v_1^{-1}} h_1 T_{u_1} T_z) \tau(T_{u_2} T_{z \i} T_{v_2 \i}) \tau(T_{(u_1u_2)^{-1}}h_2 T_{v_1}) f(T_{v_2}) \\ &=\sum \q^{-\ell(v_1)-\ell(v_2)-\ell(u_1)-\ell(z)} \tau(T_{v_1^{-1}} h_1 T_{u_1} T_z) \t(T_{u_1 \i} h_2 T_{v_1} T_{z \i} T_{v_2 \i}) f(T_{v_2}) \\ &=\sum \q^{-\ell(v_2)-\ell(u_1)} \t(T_{(u_1 v_2) \i} h_2 h_1 T_{u_1}) f(T_{v_2}) \\ &=\indFunctionName(h_2 h_1). \qedhere
\end{align*}

\subsection{Support of conjugacy classes}\label{sec:LS}
In this section, we review some properties of conjugacy classes, and conjugating elements, in Coxeter groups.

The following facts are frequently used in the literature, although not often explicitly explained. They can be found, with slightly different focus and notation, e.g.\ in \cite[Section~3]{Kr09}.

\begin{lemma}\label{lem:parabolicConjugation} 
    Let $I_1, I_2\subseteq S$ be two subsets and $v\in {}^{I_1} W^{I_2}$.
    \begin{enumerate}
    \item If $u_1<u_2\in W_{I_1}$, then $v^{-1} u_1 v<v^{-1} u_2 v$.
    \item Set $J = I_2\cap v^{-1} I_1 v$, then $W_{I_2}\cap v^{-1} W_{I_1} v = W_J$ and
    \[W_{v J v^{-1}}\to W_J,~w\mapsto v^{-1} w v\]
    is a bijective map preserving length and Bruhat order.
    \item Let $u_1\in W$ be of minimal length in its conjugacy class and $I_1 = \supp(u_1)$. Then every $W_{I_1}$-conjugate of $u_1$ has support equal to $I_1$.
    \item Suppose we are given minimal length elements $u_1, u_2\in W$ that are conjugate, i.e.\ there exists $w\in W$ with $w^{-1} u_1 w = u_2$. Assume that $I_1 = \supp(u_1), I_2 = \supp(u_2)$ and $v=\min(W_{I_1} w W_{I_2})$. Then the map \begin{align*}
        W_{I_1} \to W_{I_2},~u\mapsto v^{-1} u v.
    \end{align*}  is bijective and preserves both the length and the Bruhat order.
    \end{enumerate}
\end{lemma}

This simplifies the question to checking if two minimal length elements $w_1, w_2$ in $W$ are conjugate. First, one may check if their supports $I_1 = \supp(w_1), I_2 = \supp(w_2)$ are conjugate. If they are not, $w_1$ and $w_2$ are not conjugate. If yes, one immediately reduces to the case where both $w_1$ and $w_2$ have the same support, say $I_1$. Now one needs to compute the group
\begin{align*}
    \Aut(W, I_1) := \{\varphi : I_1\to I_1\mid \exists v\in W\mid \varphi(s) = v^{-1} s v\text{ for all }s\in I_1\},
\end{align*}
which is finite since $I_1$ is. Note that $\Aut(W, I_1)$ not only acts on $I_1$, but this action extends to an action on $W_{I_1}$, its Hecke algebra etc. By the above lemma, $w_1$ is conjugate to $w_2$ in $W$ if and only if $w_1$ is conjugate, inside $W_{I_1}$ to some element in the $\Aut(W, I_1)$-orbit of $w_2$.

It remains to explain how to check if two subsets of $S$ are conjugate, and how to compute the group $\Aut(W, I_1)$. It is done via the Lusztig-Spaltenstein algorithm \cite[Lemma~2.12]{LS79} (see also \cite[Proposition~5.5]{De82}).

\begin{definition}
    Let $I_1\subseteq S$ be any subset, and $s\in S\setminus I_1$ be a simple reflection such that the irreducible component $C$ of $I_1\sqcup\{s\}$ containing $s$ is spherical. Then $w = w_{0,C\cap I_1}w_{0,C} \in {}^{I_1}W$ satisfies $I_2 := w^{-1} I_1 w\subseteq S$. We may call $w$ the \emph{elementary conjugator} associated with $(I_1, I_2, s)$.
\end{definition}
\begin{proposition}[Lusztig-Spaltenstein algorithm]\label{prop:LSalgo}
Let $I_1\subseteq S$ be a subset and $w\in {}^{I_1}W$ be an element such that $w^{-1} I_1 w \subseteq S$. Then there exists a decomposition of $w$ as a length additive product
\begin{align*}
    w = w_1 \cdots w_n
\end{align*}
together with elements $s_1,\dotsc,s_n\in S$ and subsets $I_2,\dotsc, I_{n+1}\subseteq S$ such that $w_i$ is the elementary conjugator associated with $(I_i, I_{i+1}, s_i)$ for all $i=1,\dotsc,n$.
\end{proposition}
This proposition makes it very easy to determine the conjugacy classes of subsets of $S$ and the groups $\Aut(W, I)$, since we only have to consider the (iterated) actions of elementary conjugators. In particular, 

(a) If $I_1$ is as in Proposition~\ref{prop:LSalgo} and $C$ is an irreducible component of $I_1$ not of type $A$, then we get  $w^{-1} C w = C$. If moreover $C$ is not of type $A$ nor $D$, then $w$ centralizes $W_C$.

\subsection{Compatibility with $f$}
Let $w \in W_I$. We have 
\begin{align*}
    \indFunctionName(T_w) = \sum_{v_1\in W^I}\sum_{v_2\in W_I} f(T_{v_2})\q^{-\ell(v_1)-\ell(v_2)} \tau(T_{(v_1 v_2)^{-1}} T_w T_{v_1}),
\end{align*}

The contribution of the term with $v_1 = 1$ is precisely $f(T_w)$, as $\q^{-\ell(v_2)}\tau(T_{v_2^{-1}} T_w) = \delta_{v_2, w}$. We claim that all terms with $v_1 \neq 1$ vanish under the hypothesis of Theorem \ref{thm:classPolynomialLifting} (3). 

Suppose that $v_1 \neq 1$ in $W^I$ and $f(T_{v_2}) \neq 0$, $\tau(T_{(v_1 v_2)^{-1}} T_w T_{v_1}) \neq 0$.
By Lemma \ref{lem:liftingFiniteness1}, $v_1 v_2 v_1^{-1} = z$ for some $z \leq w$. Hence $v_1 v_2 v_1^{-1} \in W_I$. Write $v_1 = v_3 v_4$ with $v_3 \in W_I$ and $v_4 \in {}^{I}W^{I}$. Then $v_4^{-1} v_2 v_4 \in W_I$ as well. By Lemma \ref{lem:parabolicConjugation}, $v_4^{-1} \operatorname{supp}(v_2) v_4 \subseteq I$.
Since $f(T_{v_2}) \neq 0$, the support of $v_2$ contains a subset $J'$ which is $W_I$-conjugate to $J$. Therefore, $J'' := v_4^{-1} J' v_4 \subseteq I$ is a subset $W$-conjugate to $J$.

By the Lusztig-Spaltenstein algorithm (Proposition~\ref{prop:LSalgo}), we find subsets $J' = J_1,\dotsc,J_{n+1} = J''$ and elements $s_1,\dotsc,s_{n}\in S$ such that $v_4$ decomposes as the length additive product $v_4 = w_1\cdots w_n$, where $w_i$ is the elementary conjugator associated to $(J_i, J_{i+1}, s_i)$. Then each $J_i$ is $W$-conjugate to $J'$ and hence to $J$, thus $J_i\subseteq N_W(J_i)\subseteq W_I$. We claim that this implies $w_i\in W_I$: If $J_{i+1}\neq J_i$, then we get $s_i\in J_i\cup J_{i+1}$ and hence $w_i\in W_{J_i\cup J_{i+1}}\subseteq W_I$. If $J_{i+1} = J_i$, then $w_i\in N_W(J_i)\subseteq W_I$.

This argument shows that $v_4\in W_I$. Since we assumed $v_4\in {}^IW^I$, we get $v_4=1$. Consequently, $v_1 = v_3 \in W_I$. But $v_1 \in W^I$, so $v_1 = 1$, a contradiction. This proves the claim. 

\subsection{Support of $\indFunctionName$}
Assume $\indFunctionName(T_w) \neq 0$ for $w \in W_K$. Then there exists a pair $(v_1, v_2)\in W^I\times W_I$ with $f(T_{v_2}) \neq 0$ and $\tau(T_{(v_1 v_2)^{-1}} T_w T_{v_1}) \neq 0$. By Lemma \ref{lem:liftingFiniteness1}, the element $v_1 v_2 v_1^{-1}$ is contained in $W_K$.

Write $v_1 = v_3 v_4$ with $v_3\in W_K$ and $v_4\in {}^K W^I$, so that $z = v_4 v_2 v_4^{-1}$ is contained in $W_K$.
Then by Lemma \ref{lem:parabolicConjugation}, $\operatorname{supp}(z) = v_4 \supp(v_2) v_4^{-1} \subseteq K$. Since $f(T_{v_2}) \neq 0$, $\operatorname{supp}(v_2)$ contains a subset $J'$ which is $W_I$-conjugate to $J$. Thus $\operatorname{supp}(z) \subseteq K$ also contains a subset which is $W$-conjugate to $J$. This contradicts the hypothesis of part (4). Hence $\indFunctionName(T_w)$ must be zero.
\subsection{Specialization at $\q=1$}
Let $\CO\in \Cl(W)$ and $w\in \CO_{\min}$. Under the specialization $\q=1$, the value of $\indFunctionName(T_w)$ becomes, by definition,
\begin{align*}
    \indFunctionName(T_w)_{\q=1} = \sum_{\substack{(v_1, v_2)\in W^I\times W_I\\v_1 v_2 = w v_1}} f(T_{v_2})_{\q=1}.
\end{align*}
For $\CO'\in \Cl(W_I)$, we set
\begin{align*}
    c_{\CO'} := \#\{(v_1, v_2)\in W^I\times W_I\mid v_1 v_2 = wv_1\text{ and }v_2\in \CO'\}.
\end{align*}
Then 
\begin{align*}\indFunctionName(T_{\CO})_{\q=1} = \sum_{\CO'\in \Cl(W_I)} f(T_{\CO'})_{\q=1} c_{\CO'}.
\end{align*}
Certainly, $c_{\CO'}=0$ unless $\CO'\subseteq \CO$. Conversely, if $\CO'\subseteq \CO$, we find $v\in W$ with $v^{-1} w v \in \CO'$. This condition does not change if $v$ is replaced by any element in $vW_I$, so we may assume $v\in W^I$. Then the tuple $(v, v^{-1} w v)$ is in the set defining $c_{\CO'}$, so $c_{\CO'}>0$. We summarize: $c_{\CO'}\neq 0$ if and only if $\CO'\subseteq \CO$. This finishes the proof.

\section{Class polynomials for elliptic conjugacy classes}\label{sec:elliptic}

\subsection{Elliptic conjugacy classes} By definition, a conjugacy class $\CO$ of $W$ is called {\it elliptic} if $\CO \cap W_J=\emptyset$ for any $J \subsetneq S$. The class polynomials for the elliptic conjugacy classes does not arise from the parabolic induction method as in Theorem \ref{thm:classPolynomialLifting}. For these classes, it is possible to give a direct and explicit construction of (a weaker version of) class polynomials. 

\begin{theorem}\label{thm:almostEllipticClassPolynomial}
    Assume that $(W, S)$ is irreducible and
    let $\mathcal O\in\Cl(W)$ be an elliptic conjugacy class. Then there exists a function $F : \overline H\to \BZ[\q^{\pm 1}]$ and an integer $z>0$ such that
    \begin{align*}
        F(T_{\CO'}) \equiv z\delta_{\CO,\CO'}\pmod{\q-1}
    \end{align*}
    holds for all $\CO'\in \Cl(W)$.
\end{theorem}

The case where $(W, S)$ is of finite type is classically known, and can be derived from Tits' deformation theorem. This result states that after a base change to an algebraic closure of $\mathbb Q(\q)$, the Hecke algebra $H_{\overline{\BQ(\q)}}$ becomes isomorphic to the group algebra $\overline{\BQ(\q)}[W]$. In particular, our spanning set of the cocenter $\overline H$ from Proposition~\ref{prop:cocenterSpanning} is linearly independent, and thus $\overline H$ is a free module over this spanning set. Then the statement of Theorem~\ref{thm:almostEllipticClassPolynomial} becomes straightforward.

The affine type is similarly established by the first author and Nie \cite{HN2}, giving an explicit basis of the cocenter $\overline H$. In the rest of this section, we therefore assume that $(W, S)$ is irreducible and of indefinite type. We will construct an infinite-dimensional $H$-module with an endomorphism that allows us to get well-defined \emph{trace maps}. This construction was motivated by a classical construction from affine Deligne-Lusztig varieties \cite[Section~11]{GHKR}.

\subsection{The $H$-module $M^b$}

By definition, an element $b\in W$ is called \emph{straight} if $\ell(b^n) = n\ell(b)$ holds for all integers $n\geq 1$. The notion of straight elements was introduced by Krammer \cite{Kr09}. It played a fundamental role in the study of conjugacy classes in Coxeter groups \cite{HN2, Ma21, Ma25}.

\begin{proposition}
\label{prop:Mbmodule}
Let $b$ be a straight element. Set $C_H(T_b)= \{h\in H\mid h T_b = T_b h\}$. Let $M^b$ be the free $\mathbb Z[\q^{\pm 1}]$-module with basis $\{M^b_z\mid z\in W\}$.
For integers $n\geq 1$, we let $f_n:M^b\to H$ be the unique $\mathbb Z[\q^{\pm 1}]$-linear map sending $M^b_z$ to $T_{zb^n}$ for all $z\in W$. Then there exists a unique $H$-$C_H(T_b)$-bimodule structure of $M^b_z$ such that for all $h\in H, m\in M$ and $c\in C_H(T_b)$, there exists some constant $N\geq 1$ with
        \begin{align*}
            f_n(hmc) = h f_n(m) c\text{ for all }n\geq N.
        \end{align*}
\end{proposition}
\begin{proof}
We first show the following auxiliary claim:
For every $m\in M$, there exists some constant $N\geq 1$ with
        \begin{align}\label{eq:Mbmoduleproof}
            f_{n_2}(m) = f_{n_1}(m) T_{b^{n_2-n_1}}\text{ for all }n_2\geq n_1\geq N.
        \end{align}
By linearity, it suffices to consider the case $m = M^b_z$ for some $z\in W$. Consider the sequence
    \begin{align*}
        a_n = \ell(zb^n) - \ell(b^n) = \ell(zb^n) -n\ell(b),~n\geq 1.
    \end{align*}
    It is bounded from below by $-\ell(z)$. It is non-increasing since
    \begin{align*}
        a_{n+1} - a_n = \ell(zb^{n+1}) - \ell(zb^n) - \ell(b) \leq 0.
    \end{align*}
    Hence the sequence eventually stabilizes. Let $N\gg 0$ with $a_N = a_{N+1} = a_{N+2}=\cdots$. Then for $n_2\geq n_1\geq N$, we get
    \begin{align*}
        \ell(zb^{n_2}) = a_{n_2} + n_2 \ell(b) = a_{n_1} + n_2\ell(b) = \ell(zb^{n_1}) + (n_2-n_1)\ell(b) = \ell(zb^{n_1}) + \ell(b^{n_2-n_1}).
    \end{align*}
    Hence, the product $(zb^{n_1})(b^{n_2-n_1})$ is length additive. We conclude
    \begin{align*}
        f_{n_2}(m) = T_{zb^{n_2}} = T_{zb^{n_1}} T_{b^{n_2-n_1}} = f_{n_1}(m) T_{b^{n_2-n_1}}.
    \end{align*}
    Our auxiliary claim \eqref{eq:Mbmoduleproof} is proved.
    
    Note that each map $f_n$ is a $\mathbb Z[\q^{\pm 1}]$-linear endomorphism. We first claim that for each $h\in H, m\in M$ and $c\in C_H(T_b)$, the sequence
    \begin{align*}
        m_n := f_n^{-1}(hf_n(m) c)\in M^b,~n\geq 1
    \end{align*}
    stabilizes: By \eqref{eq:Mbmoduleproof}, we find $N_1\geq 1$ with $f_n(m) = f_{N_1}(m) T_{b^{n-N_1}}$ for all $n\geq N_1$. Set now $h_2 := h f_{N_1}(m) c$ and let $L\geq 1$ be a constant such that
    \begin{align*}
        h_2 \in \sum_{\substack{z\in W\\\ell(z)\leq L}} \mathbb Z[\q^{\pm 1}]T_z.
    \end{align*}
    Then for $n\geq 0$, we get
    \begin{align*}
        h_2 T_{b^n} \in \sum_{\substack{z\in W\\\ell(z)\leq L}} \mathbb Z[\q^{\pm 1}]T_{zb^n} = f_n(M_2),
    \end{align*}
    where $M_2 := \sum_{\substack{z\in W\\\ell(z)\leq L}} \mathbb Z[\q^{\pm 1}]M^b_z$. By \eqref{eq:Mbmoduleproof}, we can find $N_2\geq 1$ such that for all $m\in M_2$ and $n\geq N_2$, we get
    \begin{align*}
        f_n(m) = f_{N_2}(m) T_{b^{n-N_2}}.
    \end{align*}
    Let $m_2\in M_2$ with $h_2 T_{b^{N_2}} = f_{N_2}(m_2)$. Then for all $n\geq N_1+N_2$, we get
    \begin{align*}
        &f_n^{-1}(h f_n(m) c) = f_n^{-1}(h f_{N_1}(m) (T_b)^{n-N_1} c) = f_n^{-1}(h f_{N_1}(m) c (T_b)^{n-N_1})
        \\=&f_n^{-1}(h_2 T_{b^{N_2}} T_{b^{n-N_1-N_2}}) = f_n^{-1}(f_{N_2}(m_2) T_{b^{n-N_1-N_2}})
        =f_n^{-1}(f_{n-N_1}(m_2)).
    \end{align*}
    Using the definition of the $f_\bullet$, it is easy to see that this sequence stabilizes.

    We now define our bimodule structure on $M^b$: For $m\in M, h\in H$ and $c\in C_H(T_b)$, we let $hm$ be the limit value of the sequence $f_n^{-1}(hf_n(m))$ for $n\to \infty$, and $mc$ the limit value of $f_n^{-1}(f_n(m)c)$. The defining properties of a $H$-$C_H(T_b)$-bimodule are now easy to verify, using that $H$ is naturally a $H$-$C_H(T_b)$-bimodule.
\end{proof}

By definition, each $m\in M^b$ can be written as a unique $\mathbb Z[\q^{\pm 1}]$-linear combination of the basis elements $M^b_z$. We therefore get well-defined values $m_z\in \mathbb Z[\q^{\pm 1}]$ for all $z\in W$ such that
\begin{align*}
    m = \sum_{z\in W} m_z M^b_z.
\end{align*}

\begin{definition}
    Let $b,w, z\in W$ such that $b$ is straight and $T_b T_w = T_w T_b$. Let $h\in H$. We write
$M^b(h; z; w) := (h M^b_z T_w)_z$ for the coefficient of $M^b_{z}$ in $h M^b_{z}T_w$.
\end{definition}
Using the $H$-$C_H(T_b)$-bimodule automorphism $M^b\to M^b,~M^b_z\mapsto M^b_{zb}$, we see that the value of $M^b(h; z; w)$ only depends on the right coset $z\langle b\rangle\in W/\langle b\rangle$.

\subsection{The trace map}
We want to consider \emph{traces} of pairs $(h, h')\in H\times C_H(T_b)$ on $M^b$. More concretely:
\begin{definition}
    Let $b,w\in W$ such that $b$ is straight and $T_b T_w = T_w T_b$. For $h\in H$, we define
    \begin{align*}
        \mathrm{tr}_{b,w}(h) := \sum_{z\in W/\langle b\rangle} M^b(h; z; w),
    \end{align*}
    with the convention that $\mathrm{tr}_{b,w}(h)$ is undefined if the defining sum has infinitely many non-zero terms.
\end{definition}

\begin{proposition}
\label{prop:trace-properties}
Suppose that $b, w\in W$ are given such that $b$ is straight and $T_b T_w = T_w T_b$. Assume that $\mathrm{tr}_{b,w}(h)$ is well-defined for all $h\in H$. Then
\begin{enumerate}
\item The map $\operatorname{tr}_{b,w}$ vanishes on $[H,H]$; hence it factors through $\bar{H}$.
\item Under the specialization $\q=1$, we get for all $v\in W$ that
\[
\operatorname{tr}_{b,w}(T_v)\equiv [C_W(w):\langle b\rangle]\cdot\varepsilon\pmod{\q-1},
\]
where $\varepsilon$ is $1$ if $v$ is $W$-conjugate to $w^{-1}$, and $0$ otherwise.
\end{enumerate}
\end{proposition}

\begin{proof}
    \begin{enumerate} 
    \item
    Let $v_1, v_2\in W$. Then we get
    \begin{align*}
        \mathrm{tr}_{b,w}(T_{v_1} T_{v_2}) &= \sum_{z\in W/\langle b\rangle} (T_{v_1} T_{v_2} M^b_z T_w)_z.
        \end{align*}
        Recall that $T_{v_2} M^b_z\in M^b$ is a $\BZ_{\geq 0}[\q^{\pm 1}]$-linear combination of the basis vectors $\{M^b_y\mid y\in W\}$. Explicitly, $T_{v_2} M^b_z = \sum_{y\in W}(T_{v_2}M^b_z)_y M^b_y$, where $(T_{v_2} M^b_z)_y$ is a $\mathbb Z_{\geq 0}$-linear combination of $\{\q^e\mid e\in\mathbb Z\}$. Setting $g(y,z) := (h_1 M^b_y T_w)_z(h_2 M^b_z)_y\in \mathbb Z_{\geq 0}[\q^{\pm 1}]$, we thus get
        \begin{align*}
        \mathrm{tr}_{b,w}(T_{v_1} T_{v_2}) &= \sum_{\substack{z\in W/\langle b\rangle\\y\in W}} g(y,z).
        \end{align*}
        By assumption, this is a finite sum, with each summand in $\mathbb Z_{\geq 0}[\q^{\pm 1}]$. This implies that $g(y,z)=0$ for all but finitely many orbits $(y,z) \in (W\times W) / \langle (b,b^{-1})\rangle$ (note that $g(y,z) = g(yb^{-1},zb)$). Thus, we get
        \begin{align*}
        \mathrm{tr}_{b,w}(T_{v_1} T_{v_2})
        = \sum_{\substack{z\in W\\y\in W/\langle b\rangle}} g(y,z).
        \end{align*}
        Now we use that $T_{v_1} M^b_y T_w = \sum_{z\in W}(T_{v_1} M^b_y T_w)_z M^b_z$ to get
        \begin{align*}
        \mathrm{tr}_{b,w}(T_{v_1]} T_{v_2}) &= \sum_{\substack{z\in W\\y\in W/\langle b\rangle}} (T_{v_2} T_{v_1} M^b_y T_w)_y
        =\mathrm{tr}_{b,w}(T_{v_2} T_{v_1}).
    \end{align*}
    By linearity, we get $\tr_{b,w}(h_1 h_2) = \tr_{b,w}(h_2 h_1)$ for all $h_1, h_2\in H$.
    \item Specializing $\q\to1$ sends $T_x$ to $x$ in the group algebra $\mathbb{Z}[W]$.  Then
\[
\operatorname{tr}_{b,w}(T_v)\equiv\sum_{z\in W/\langle b\rangle}\lim_{n\to\infty}\delta_{z b^n,\;v z b^n w}
      =\#\{z\in W/\langle b\rangle\mid z^{-1}vz=w^{-1}\}.
\]
If $v$ is not $W$-conjugate to $w^{-1}$, this quantity is zero as claimed. Otherwise, it agrees with the index of $\langle b\rangle$ in $C_W(w)$, which hence must be finite.  This gives the claimed formula. 
    \qedhere\end{enumerate}
\end{proof}
\begin{remark}
    Well-definedness of $\mathrm{tr}_{b,w}(h)$ for all $h\in H$ is a strong condition. While
    trivially satisfied for all finite Coxeter groups, it is never true for all $(b, w)$ in a non-trivial affine Coxeter group.

    The affine case still serves as our motivation, for the following reason: If $\tilde W$ is an \emph{extended} affine Weyl group of type $A$, it contains additional length zero elements. Such a length zero element $b\in \tilde W$ is called \emph{superbasic} if its centralizer is generated by $b$ itself. Then the (analogously defined) trace map $\mathrm{tr}_{b, b}$ is well-defined for all $h\in H$. For $w\in \tilde W$, the value of $\mathrm{tr}_{b,b}(T_w)$ is the \emph{normalized cardinality} of the affine Deligne-Lusztig variety $X_w(b)$, cf.\ \cite[Section~11]{GHKR}, \cite[Section~8]{He-Ann} and \cite[Section~3]{HNY24}.
\end{remark}

\subsection{Application to elliptic classes}
We now specialize to the situation where the conjugacy class $\mathcal{O}$ is elliptic.  The following result  provides a convenient normal form for elements of $\mathcal{O}_{\min}$. It was first established by the joint work of Nie and the first-named author \cite{HN2} in the affine type and generalized by Marquis \cite[Theorem C]{Ma25}  for arbitrary type. 

\begin{theorem}
\label{thm:MarquisIndefiniteMinLength}
    Let $\mathcal O \in \Cl(W)$ and $w\in \mathcal O_{\min}$. Then there exists an element $w'\in\mathcal O_{\min}$ with $w'\approx w$ such that $w'$ admits a length additive decomposition $w' = ab$ where $b\in W$ is a straight element, $a\in W_J$ for a spherical subset $J\subseteq S$ which is normalized by $b$. If $(W, S)$ is irreducible and not spherical, then $\CO$ is elliptic if and only if $b$ is elliptic.
\end{theorem}

For elliptic classes in indefinite type, the centralizer of such an element is ``as small as possible”.

\begin{lemma}[{\cite[Corollary~6.3.10]{Kr09}}]
    \label{lem:KrammerFiniteIndex}Let $\mathcal O$ be an elliptic conjugacy class in the irreducible and indefinite Coxeter group $(W, S)$. Let $w\in \mathcal O$. Then the cyclic subgroup $\langle w\rangle\subseteq W$ generated by $w$ has finite index in the centralizer $Z_W(w)$.
\end{lemma}

These facts allow us to construct the trace maps used for Theorem~\ref{thm:almostEllipticClassPolynomial}.
\begin{proof}[Proof of Theorem~\ref{thm:almostEllipticClassPolynomial}]
If $(W, S)$ is finite or affine, the claim is clear from the explicitly known basis of $\overline H$, cf.\ \cite{GP93} respectively \cite{HN2}. So we only have to consider the indefinite case. Consider the conjugacy class $\CO^{-1} = \{w^{-1}\mid w\in \CO\}$, which is elliptic again. By Theorem~\ref{thm:MarquisIndefiniteMinLength}, we find $w\in (\CO^{-1})_{\min}$ with a length additive decomposition $w = ab'$ such that $b'\in W$ is straight and $a\in W_J$ for a spherical subset $J\subseteq S$ normalized by $b$. If $n\geq 1$ is the order of the action of $b'$ on $W_J$, then we set $b := (b')^n$. Hence, $wb = bw$ has length $\ell(w)+\ell(b)$. In particular, $T_w T_b = T_{wb} = T_{bw} = T_b T_w$.

It remains to show that the trace map $\mathrm{tr}_{b,w}(h)$ is well-defined for all $h\in H$, as then the function $F_{\CO} := \mathrm{tr}_{b,w}$ satisfies the required conditions by using Proposition~\ref{prop:trace-properties}.

    We thus have to show for each $h\in H$ that the set
    \begin{align*}
        \{z\in W\mid M^b(h; z; w)\neq 0\}\subseteq W
    \end{align*}
    is a finite union of right $\langle b\rangle$-cosets. By linearity of $M^b(\bullet, z; w)$, it suffices to show this in case $h= T_v$ for some $v\in W$.

    So fix $v\in W$ and suppose $z\in W$ is given with $M^b(T_v; z; w)\neq 0$. By Proposition~\ref{prop:Mbmodule}, this means that for all sufficiently large $n\gg 0$, the coefficient of $T_{zb^n}$ in $T_v T_{zb^n} T_w$ is non-zero. Note that
    \begin{align*}
        T_v T_{zb^n} T_w = T_v T_{zb^n} T_{b'} T_a \in \sum_{\substack{v'\leq v\\a'\leq a}}\BZ[\q^{\pm 1}] T_{v'zb^n b' a'}.
    \end{align*}
    Thus, there exist $v'\leq v$ and $a'\leq a^{-1}$ with $zb^n = v' zb^n b' (a')^{-1}$, or in other words $v' = (zb^n) b' a' (zb^n)^{-1} = z (b' a') z^{-1}$.

    We summarize that
    \begin{align*}
        \{z\in W\mid M^b(T_v; z; w)\neq 0\}\subseteq \bigcup_{\substack{v'\leq v\\ a'\leq a^{-1}}}\{z\in W\mid z (b' a')z^{-1}=v'\}.\tag{$\ast$}
    \end{align*}
    There are only finitely many possibilities for $v'\leq v$ and $a'\leq a^{-1}$. For each $(v', a')$, the set $\{z\in W\mid z (b' a')z^{-1}=v'\}$ is either empty or a single right $C_W(b'a')$-coset. We note that $(b')^m\in W$ is straight and has support equal to $S$ for all $m\geq 1$. It follows that each $(b')^m$ is elliptic. Since $(b' a')^m = (b')^m$ for some $m\geq 1$, also $b' a'$ must be in an elliptic $W$-conjugacy class. Thus, $\langle b\rangle$ has finite index in $C_W(b' a')$ by Lemma~\ref{lem:KrammerFiniteIndex}.

    We summarize that each of the finitely many sets $\{z\in W\mid z (b' a')z^{-1}=v'\}$ in the right-hand side of $(\ast)$ is a finite union of right $\langle b\rangle$-cosets. This finishes the proof.
\end{proof}

\section{Orbital Integrals of Kac-Moody groups}\label{sec:orbital_integrals}

\subsection{Haar measure and orbital integrals}
Consider a finite field $\kk = \mathbb F_q$, and consider the maximal Kac-Moody group $G(\mathbb F_q) := G^{\max}(\mathbb F_q)$ as in Section~\ref{subsec:kac-moody}. We similarly abbreviate $B(\mathbb F_q) := B^{\max}(\mathbb F_q)$ and $U(\mathbb F_q) := U^{\max}(\mathbb F_q)$.

Equip $G(\mathbb F_q)$ with the structure of a topological group, where the finite index subgroups of $B(\mathbb F_q)$ form a neighbourhood basis of the identity. 
Then the topological group $G(\mathbb F_q)$ is a locally compact, totally disconnected and unimodular, with $B(\mathbb F_q)$ a compact subgroup \cite[§1]{CG99}. 

Let $\mu$ denote the Haar measure of $G(\mathbb F_q)$, normalized so that $\mu(B(\mathbb F_q)) = 1$. Then for any finite index subgroup $H\leq B(\mathbb F_q)$, we have $$\mu(H) = \frac 1{[B(\mathbb F_q):H]}.$$

As a side remark, we note that this setup is only possible for the maximal Kac-Moody groups. Indeed, if $W$ is infinite, then $B^{\min}(\mathbb F_q)$ is a countably infinite group, so it can never be compact and hence cannot have Haar measure equal to $1$.

We consider the ring homomorphism $\mathbb Z[\q^{\pm 1}]\to\mathbb C$ sending $\q$ to $q$, and let $H_{\mathbb C}$ denote the base change of the Hecke algebra $H$ under this ring homomorphism.
\begin{definition}
    Let $\mathcal C\subseteq G(\mathbb F_q)$ be a measurable subset that is stable under conjugation. The \emph{orbital integral of $\mathcal C$} is the linear map
    \begin{align*}
        O_{\mathcal C} : H_{\mathbb C}\to\mathbb C,\qquad h\mapsto \int_{\mathcal C} h(g)\, dg,
    \end{align*}
    where the integral is taken with respect to the Haar measure $\mu$. 
\end{definition}

\begin{proposition}\label{prop:orbital_integral_cocenter}
   For any measurable conjugation‑stable subset $\mathcal C\subseteq G(\mathbb F_q)$, the orbital integral $O_\mathcal C$ vanishes on the commutator $[H_{\mathbb C},H_{\mathbb C}]$.
\end{proposition}

\begin{proof}
    It suffices to prove $$O_{\mathcal C}(T_{w_1} T_{w_2}) = O_{\mathcal C}(T_{w_2} T_{w_1})$$ for all $w_1, w_2\in W$. 
    
    View elements of $H_{\mathbb C}$ as $B(\mathbb F_q)$-biinvariant functions on $G(\mathbb F_q)$. For any $g_2\in G(\mathbb F_q)$, we thus get
    \begin{align*}
        (T_{w_1} T_{w_2})(g_2) = \int_{G(\mathbb F_q)} T_{w_1}(g_1) T_{w_2}(g_1^{-1} g_2)\, dg_1.
    \end{align*}
All functions involved are non‑negative, so we may freely apply Fubini’s theorem:
    \begin{align*}
        O_{\mathcal C}(T_{w_1} T_{w_2}) &= \int_{\mathcal C}\int_{G(\mathbb F_q)} T_{w_1}(g_1) T_{w_2}(g_1^{-1} g_2)\, dg_1\, dg_2
        \\&=\int_{G(\mathbb F_q)} T_{w_1}(g_1) \int_{\mathcal C} T_{w_2}(g_1^{-1} g_2)\, dg_2\, dg_1.
    \end{align*}
    By unimodularity of $G$ and the conjugation‑invariance of $\mathcal C$, conjugation by $g_1$ preserves the measure. Hence for any $g_1\in G(\mathbb F_q)$,
    \begin{align*}
        \int_{\mathcal C} T_{w_2}(g_1^{-1} g_2)\, dg_2 = \int_{\mathcal C} T_{w_2}(g_2g_1^{-1})\, dg_2.
    \end{align*}
    Using this substitution in our above calculation, we can use Fubini's Theorem again to get
    \begin{align*}
        O_{\mathcal C}(T_{w_1} T_{w_2})&=\int_{G(\mathbb F_q)} T_{w_1}(g_1) \int_{\mathcal C} T_{w_2}(g_2g_1^{-1} )\, dg_2\, dg_1
        \\&=\int_{\mathcal C} \int_{G(\mathbb F_q)} T_{w_1}(g_1) T_{w_2}(g_2g_1^{-1} )\, dg_1\, dg_2.
    \end{align*}
   Now observe that the Haar measure is invariant under the inversion map $g\mapsto g^{-1}$. Consequently, for any fixed $g_2\in G(\mathbb F_q)$, the map $g_1\mapsto g_2 g_1^{-1}$ is measure‑preserving. Making the change of variables $\tilde g_1=g_2 g_1 \i$, we obtain
    \begin{align*}
        \int_{G(\mathbb F_q)} T_{w_1}(g_1) T_{w_2}(g_2g_1^{-1} )\, dg_1 = 
        \int_{G(\mathbb F_q)} T_{w_1}(\tilde g_1^{-1} g_2) T_{w_2}(\tilde g_1 )\, d\tilde g_1.
    \end{align*}
    Continuing our above calculation, we get
        \begin{align*}
        O_{\mathcal C}(T_{w_1} T_{w_2})&=\int_{\mathcal C} \int_{G(\mathbb F_q)} T_{w_1}(\tilde g_1^{-1} g_2) T_{w_2}(\tilde g_1)\, d\tilde g_1\, dg_2
        \\&=O_{\mathcal C}(T_{w_2} T_{w_1}).
    \end{align*}
    This finishes the proof.
\end{proof}

\subsection{Lifting semisimple classes}\label{subsec:lifting-ss}
In this subsection we explain how to compute orbital integrals of certain semisimple classes in $G(\mathbb F_q)$. It is known from the $p$-adic theory of orbital integrals that regular semisimple classes are best behaved for this purpose. Therefore, we study certain unions of conjugacy classes that are as regular as possible in $G(\mathbb F_q)$.

Mimicking the theory of finite groups of Lie type, we might call an element $s\in T(\overline{\mathbb F}_q)$ \emph{regular} if its centralizer in $G(\overline{\mathbb F}_q)$ is just $T(\overline{\mathbb F}_q)$ itself. 

Sadly, such a notion would be void for Lie groups of non finite type: If $s\in T(\overline{\mathbb F}_q)$ is any element, then there exists a finite subfield $\mathbb F_{q'}\subseteq \overline{\mathbb F}_q$ with $s\in T(\mathbb F_{q'})$. Note that $W$ acts on the finite set $T(\mathbb F_{q'})$. Thus, if $W$ is infinite, the $W$-stabilizer of $s$ is infinite. Lifting an element $w\in W$ which stabilizes $s$ to $N_{\mathbf G}(\mathbf T)(\mathbb F_{q'})$, we obtain an element in the $G(\overline{\mathbb F}_q)$-centralizer of $s$ that does not lie in $T$.

Therefore, what we do in this section is to require regularity of elements in $T(\overline{\mathbb F}_q)$ with respect to Levi subgroups $M_J\subseteq G$ for spherical $J\subseteq S$. In the affine case, this recovers the usual notion of a \emph{regular semisimple element}, which denotes elements whose centralizer is the loop group of a torus.

Let $J\subseteq S$ be a spherical subset and $M_J$ be the associated Levi subgroup. Set $\underline M_J=M_J/Z(M_J)^{\circ}$, where $Z(M_J)^{\circ}$ is the identity component of its center. Let $\underline C\subseteq \underline M_J(\mathbb F_q)$ a regular semisimple conjugacy class inside $\underline M_J(\mathbb F_q)$, which we assume to be elliptic (so its intersection with any proper rational parabolic subgroup of $\underline M_J(\mathbb F_q)$ is empty). We call $(J, \underline C)$ a {\it spherical elliptic pair}. 

Let $(J, \underline C)$ be a spherical elliptic pair. Let $C \subseteq M_J(\mathbb F_q)$ be the inverse image of $\underline C$ under the natural surjection $M_J(\mathbb F_q)\twoheadrightarrow \underline M_J(\mathbb F_q)$. Let $\hat C$ be the set of all elements that are $G(\mathbb F_q)$-conjugate to some element of $C U_J(\mathbb F_q)$. Since $U_J$ is open, so is $C U_J$. Hence $\hat C$ too is open and thus measurable. 

For subsets $J_1, J_2$ of $S$, we write
\begin{align*}
    \mathrm{Iso}_W(J_1, J_2) := \{\tau : J_1\to J_2\text{ is a bijection}\mid \exists w\in W, \forall s\in J_1:~\tau(s) = w^{-1} s w\}.
\end{align*}

Each element $\tau\in \mathrm{Iso}_W(J_1, J_2)$ induces isomorphisms between the parabolic subgroups $W_{J_1}\cong W_{J_2}$ of $W$, Hecke algebras etc.\ associated with $J_1$ and $J_2$. We write $(J, \underline C) \sim (J', \underline C')$ if there exists $\t \in \mathrm{Iso}_W(J, J')$ with $\t(\underline C)=\underline C'$. We show that $\hat C$ depends only on the $\sim$-equivalence class of $(J, \underline C)$.

The first step in this proof is to study the spherical case. Here, our construction of $\hat{C}$ matches a union of classical objects, known as Steinberg fibers. Let us briefly recall the setup for finite reductive groups. Whenever $W$ is finite, $G(\mathbb{F}_q)$ is a finite group of Lie type. Each element $g \in G(\mathbb{F}_q)$ admits a unique Jordan decomposition $g=su=us$, where $s \in G(\mathbb{F}_q)$ is semisimple and $u \in G(\mathbb{F}_q)$ is unipotent. The Steinberg map is the morphism that sends $g \in G(\overline{\mathbb{F}}_q)$ to the $G(\overline{\mathbb{F}}_q)$-conjugacy class of its semisimple part. This map can be identified with the GIT quotient map $G(\overline{\mathbb{F}}_q) \to G(\overline{\mathbb{F}}_q)//G(\overline{\mathbb{F}}_q)$. The fibers of this map are precisely the Steinberg fibers.

\begin{lemma}\label{lem:Philemma}
    Suppose that $W$ is finite and let $(J, \underline C)$ be a spherical elliptic pair. Let $g\in G(\mathbb F_q)$ with Jordan decomposition $g = su = us$, where $s$ is semisimple and $u$ is unipotent. Then $g\in \hat C$ if and only if $s$ is $G(\mathbb F_q)$-conjugate to an element in $C$. In other words, if $G$ is simply connected and semisimple, then $\hat C$ is a union of Steinberg fibers. 
\end{lemma}
\begin{proof}
It suffices to show that the following are equivalent.
    \begin{enumerate}
        \item The element $g$ is $G(\mathbb F_q)$-conjugate to an element of $CU_J(\mathbb F_q)$.
        \item The element $s$ is $G(\mathbb F_q)$-conjugate to an element of $C$.
    \end{enumerate}
    
    (1) $\implies$ (2): By \cite[Proposition~7.1]{DM91}, after conjugation by a suitable element in $G(\mathbb F_q)$, we get $g=mn$ for some $m \in C$ and $n\in U_{J}(\mathbb F_q)$ with $mn = nm$. Since $n$ commutes with $g$, it commutes with both parts of its Jordan decomposition, so $s, u, n$ are pairwise commutative. In particular, $un^{-1} = n^{-1} u$ is unipotent, and it commutes with the semisimple element $s$. Thus, $m = gn^{-1} = s(un^{-1})$ is the Jordan decomposition of $m$. By the uniqueness of the Jordan decomposition and $m\in C$ (so $m$ is semisimple), we get $s\in C$.

    (2) $\implies$ (1): By \cite[Proposition~2.5]{DM91}, $g\in C_{G}(s)^\circ$. Then $P_{J}\cap C_{G}(s)^{\circ}$ is a parabolic subgroup of $C_{G}(s)^{\circ}$. By \cite[Corollary~3.20]{DM91}, it contains a $C_{G}(s)^{\circ}(\mathbb F_q)$-conjugate of $u$. Thus, we may assume that $g, s$ and $u$ all lie in $P_J(\mathbb F_q)$. The image of $u$ along the projection $P_J(\mathbb F_q)\twoheadrightarrow M_J(\mathbb F_q)$ is a unipotent element centralizing $s$. By regularity of $s$, this projected image is trivial, hence $u\in U_J(\mathbb F_q)$. We proved (1).

    Finally, let us prove the ``in other words'' part. If $G$ is simply connected and semisimple, Steinberg's theorem asserts that all centralizers of semisimple elements are connected. Thus, an element $s\in G(\mathbb F_q)$ is $G(\mathbb F_q)$-conjugate to an element of $C$ if and only if $s$ is $G(\overline{\mathbb F}_q)$-conjugate to an element of $C$, completing the proof.
\end{proof}

\begin{lemma}\label{lem:semisimpleLiftingIndependence}
Suppose that $(J,\underline C) \sim (J',\underline C')$. Then $\hat C = \hat C'$.
\end{lemma}

\begin{proof}
By the Lusztig-Spaltenstein algorithm (Proposition \ref{prop:LSalgo}), it suffices to consider the case where there is a spherical set $K = J\sqcup\{s\}\subseteq S$ with $J' = (w_{0,K} w_{0,J}) J (w_{0,K} w_{0,J})^{-1}$ and $\tau$ comes from the conjugation action of $w_{0,K}w_{0,J}$.

Now Lemma~\ref{lem:Philemma} implies that every element in $(CU_J(\mathbb F_q))\cap M_K(\mathbb F_q)$ is $M_K(\mathbb F_q)$-conjugate to some element in $(C'U_{J'}(\mathbb F_q))\cap M_K(\mathbb F_q)$. Since $M_K(\mathbb F_q)$ normalizes $U_K(\mathbb F_q)$, we see that every element in
\begin{align*}
    CU_J(\mathbb F_q) = [(CU_J(\mathbb F_q))\cap M_K(\mathbb F_q)] U_K(\mathbb F_q)
\end{align*}
is $M_K(\mathbb F_q)$-conjugate to an element of
\begin{align*}
[(C'U_{J'}(\mathbb F_q))\cap M_K(\mathbb F_q)] U_K(\mathbb F_q) = C'U_{J'}(\mathbb F_q).
\end{align*}
This finishes the proof.
\end{proof}

\subsection{Orbital integrals of lifts}
We have $\mu(U_K(\mathbb F_q)) = [B(\mathbb F_q):U_K(\mathbb F_q)]^{-1} = (q-1)^{-\# S} q^{-\ell(w_{0,K})}$. 
In view of Proposition~\ref{prop:orbital_integral_cocenter} and the constructive proof of Proposition~\ref{prop:cocenterSpanning}, the following proposition reduces the computation of any orbital integral $\mathcal O_{\hat C}(h)$ for $h\in H$ to a point-counting exercise in finite groups of Lie type.

\begin{proposition}\label{prop:orbitalIntegralLift}
Let $(J, \underline C)$ be a spherical elliptic pair. Let $w\in W$ have minimal length in its conjugacy class, and set $K := \supp(w)\subseteq S$. Then 
    \begin{enumerate}
        \item If $K$ is not spherical, then $\mathcal O_{\hat C}(T_w)=0$.
        \item Assume that $K$ is spherical. Then
        \begin{align*}
            \mathcal O_{\hat C}(T_w) = \mu(U_K(\mathbb F_q)) \sum_g T_w(g),
        \end{align*}
        where the sum is taken over all elements $g\in M_K(\mathbb F_q)$ such there exist a pair $(J',\underline C')$ equivalent to $(J, \underline C)$ such that $J'\subseteq K$ and $g$ is $M_K(\mathbb F_q)$-conjugate to an element of $C' U_{J'}(\mathbb F_q)$.
    \end{enumerate}
\end{proposition}
\begin{remark}
    From Proposition~\ref{prop:orbitalIntegralLift}, we see that if $K\subseteq S$ does not have a subset that is $W$-conjugate to $J$, then $\mathcal O_{\hat C}\vert_{H_K} = 0$.
\end{remark}

\begin{proof}
(1) Note that every element $g\in \hat C$ is contained in a compact subgroup (conjugate to $P_J(\mathbb F_q)$). Thus, $g$ is topologically periodic, in the sense that there is a sequence of integers $(a_n)$ converging to $\infty$ such that $g^{a_n}$ converges to $1\in G(\mathbb F_q)$.
    
Suppose that $K$ is not spherical. Since the value of $\mathcal O_{\hat C}(T_w)$ does only depend on the value of $T_w$ in $\overline H_{\mathbb C}$, we can replace $w$ by an element that is as in Theorem~\ref{thm:MarquisIndefiniteMinLength}. Explicitly, there exists a straight element $b\in W$ normalizing a spherical subset $J\subseteq S$ with $w'\in b W_J$.

For $n\geq 1$, we note
    \begin{align*}
        (B(\mathbb F_q) w B(\mathbb F_q))^n\subseteq \bigcup_v B(\mathbb F_q) v B(\mathbb F_q),
    \end{align*}
    where the union is taken over all $v\in W$ such that the coefficient of $T_v$ in $T_w^n$ is non-zero. Now observe
    \begin{align*}
        T_w^n \in \sum_{a\in W_J} \BZ[\q^{\pm 1}]T_{b^n a}.
    \end{align*}
    Note that $\ell(b^n a) \geq n \ell(b) - \ell(w_{0,J})$, and that $\ell(b)>0$ since $K$ is not spherical.
    Thus, we see that $B(\mathbb F_q) w B(\mathbb F_q)$ does not contain topologically periodic elements.
    This finishes the proof of (1).

(2) We wish to show
\begin{align*}
    \hat C\cap P_K(\mathbb F_q) = \bigcup_{\substack{(J',\underline C')\sim (J,\underline C)\\J'\subseteq K}} (M_K(\mathbb F_q)\cdot (C' U_{J'}(\mathbb F_q))).\tag{$\ast$}
\end{align*}

The inclusion $\supseteq$ follows from Lemma~\ref{lem:semisimpleLiftingIndependence}: if $(J',\underline C')\sim (J,\underline C)$, then $C' U_{J'}(\mathbb F_q) \subseteq \hat C$. Since $\hat C$ is closed under $G(\mathbb F_q)$-conjugation, it is closed under $M_K(\mathbb F_q)$-conjugation. Furthermore, both $M_K(\mathbb F_q)$ and $C' U_{J'}(\mathbb F_q)$ are contained in $P_K(\mathbb F_q)$, so their product is contained in $\hat C\cap P_K(\mathbb F_q)$.

Let us now prove the converse inclusion. Pick an element $g\in \hat C\cap P_K(\mathbb F_q)$. We must show that $g$ belongs to the right-hand side of $(\ast)$.

Among all subsets $J'\subseteq K$, pick a minimal one such that $P_{J'}(\mathbb F_q)$ has a non-empty intersection with the $P_K(\mathbb F_q)$-conjugacy class of $g$. Let $g_1\in P_{J'}(\mathbb F_q)$ be $P_K(\mathbb F_q)$-conjugate to $g$.

Because $\hat C$ is conjugation-invariant, $g\in \hat C$ implies $g_1\in \hat C$. Thus, there exists $h\in G(\mathbb F_q)$ with $h g_1 h^{-1}\in CU_J(\mathbb F_q)$. By the Bruhat decomposition, we can write $h = h_1 h_2 h_3$ with $h_1\in P_J(\mathbb F_q), h_2\in {}^J W^{J'}$ and $h_3\in P_{J'}(\mathbb F_q)$. 

Define $g_2 := h_3 g_1 h_3^{-1}$. Since $h_3\in P_{J'}(\mathbb F_q)$, $g_2$ remains in $P_{J'}(\mathbb F_q)$ and is still $P_K(\mathbb F_q)$-conjugate to $g$. Define $g_3 := h_2 g_2 h_2^{-1} = h_1^{-1} (h g_1 h^{-1}) h_1$. Because $C U_J(\mathbb F_q)$ is stable under conjugation by $P_J(\mathbb F_q)$, and $h_1 \in P_J(\mathbb F_q)$, we have $g_3 \in C U_J(\mathbb F_q)$. Let $w_2, w_3\in W$ be the unique elements such that
\begin{align*}
    g_2\in B(\mathbb F_q) w_2 B(\mathbb F_q)\text{ and }g_3\in B(\mathbb F_q) w_3 B(\mathbb F_q).
\end{align*}

We claim that $\supp w_2 = J'$. Since $g_2\in P_{J'}(\mathbb F_q)$, we clearly have $\supp w_2 \subseteq J'$. If $J'' := \supp w_2 \subsetneq J'$, then $g_2\in P_{J''}(\mathbb F_q)$. However, since $g_2$ is $P_K(\mathbb F_q)$-conjugate to $g$, this contradicts the minimality of $J'$. Thus, $\supp w_2 = J'$.

Next, we claim that $\supp w_3 = J$. Using $g_3\in CU_J(\mathbb F_q)$, we can write $g_3 = c_3 u_3$ with $c_3\in C$ and $u_3\in U_J(\mathbb F_q)$. Projection to the Levi quotient $M_J$ yields $c_3\in (B\cap M_J)(\mathbb F_q) w_3 (B\cap M_J)(\mathbb F_q)$. Since the image of $c_3$ in $\underline M_J(\mathbb F_q)$ lies in the elliptic regular class $\underline C$, it cannot be contained in any proper parabolic subgroup of $\underline M_J$. This forces $\supp w_3 = J$.

Recall that $h_2 g_2 = g_3 h_2$. Because $h_2\in {}^J W^{J'}$ and $w_2 \in W_{J'}$, we have $l(h_2 w_2) = l(h_2)+l(w_2)$ and hence $B(\mathbb F_q) h_2 B(\mathbb F_q) w_2 B(\mathbb F_q) = B(\mathbb F_q) h_2 w_2B(\mathbb F_q)$. Similarly, since $\supp w_3 = J$, $w_3 \in W_J$, and $h_2 \in {}^J W$, we have $l(w_3 h_2) = l(w_3) + l(h_2)$, giving $B(\mathbb F_q) w_3 B(\mathbb F_q) h_2 B(\mathbb F_q) = B(\mathbb F_q) w_3 h_2 B(\mathbb F_q)$. Comparing the double cosets, we deduce $h_2 w_2 = w_3 h_2$.

It is a standard fact that the conjugation maps $W_J\to W, w \mapsto h_2^{-1} w h_2$ and $W_{J'}\to W, w\mapsto h_2 w h_2^{-1}$ are Bruhat order preserving. For any $s\in J$, we have $s\leq w_3$ (as $J = \supp w_3$), and hence $s h_2 \leq w_3 h_2 = h_2 w_2$. Since $l(s h_2) > l(h_2)$, we must have $h_2^{-1} s h_2\leq w_2$ with $\ell(h_2^{-1} s h_2) = 1$. Thus, $h_2^{-1} s h_2\in J'$. By symmetry, $h_2 J' h_2^{-1}\subseteq J$.

Conjugation by $h_2$ thus induces a bijection $[h_2] : J\to J'$ and an isomorphism $[h_2] : \underline{M_J}\to \underline{M_{J'}}$. Setting $C' := [h_2](\underline C)$, we see $(J,\underline C)\sim (J',\underline C')$. Note that $h_2^{-1} c_3 h_2\in C'\subseteq P_{J'}(\mathbb F_q)$. Hence 
\begin{align*}
    h_2^{-1} u_3 h_2 = h_2^{-1} c_3^{-1} g_3 h_2 = (h_2^{-1} c_3 h_2)^{-1} g_2\in P_{J'}(\mathbb F_q).
\end{align*}

The subgroup $\tilde U := P_{J'}\cap h_2^{-1} U_J h_2$ of $P_{J'}$ is normalized by $M_{J'}$ since $h_2 M_{J'} h_2^{-1} = M_J$ normalizes $U_J$. Projecting $\tilde U$ along the quotient map $P_{J'}\twoheadrightarrow M_{J'}$ thus produces a normal unipotent subgroup of $M_{J'}$. Since $M_{J'}$ is reductive, any such normal unipotent subgroup is trivial, and hence $\tilde U$ is contained in $U_{J'}$.

Thus, $h_2^{-1} u_3 h_2\in \tilde U(\mathbb F_q)\subseteq U_{J'}(\mathbb F_q)$, which implies $g_2\in C' U_{J'}(\mathbb F_q)$.
Then $g_1\in C' U_{J'}(\mathbb F_q)$, which is $M_K(\mathbb F_q)$-conjugate to $g$. This completes the proof of $(\ast)$.

With $(\ast)$ established, we can now evaluate the orbital integral to verify condition (2) of Proposition \ref{prop:orbitalIntegralLift}. By definition, $\mathcal O_{\hat C}(T_w)$ is the Haar measure of the set
$$ \begin{aligned}
    D :={}& (B(\mathbb F_q) w B(\mathbb F_q)) \cap \hat C \\
    ={}& (B(\mathbb F_q) w B(\mathbb F_q)) \cap P_K(\mathbb F_q) \cap \hat C \\
    ={}& \bigcup_{\substack{(J',\underline C')\sim (J,\underline C)\\J'\subseteq K}} (B(\mathbb F_q) w B(\mathbb F_q))\cap \big(M_K(\mathbb F_q)\cdot C' U_{J'}(\mathbb F_q)\big),
\end{aligned} $$
where the second equality follows from $\supp(w) = K$ (which implies the double coset is contained in $P_K$), and the third uses $(\ast)$.

Notice that $B(\mathbb F_q) w B(\mathbb F_q)$ is stable under right multiplication by $B(\mathbb F_q)$, and thus by its subgroup $U_K(\mathbb F_q)$. Similarly, for each $J' \subseteq K$, we have $U_K(\mathbb F_q) \subseteq U_{J'}(\mathbb F_q)$, so the set $C' U_{J'}(\mathbb F_q)$ is also stable under right multiplication by $U_K(\mathbb F_q)$. Consequently, the set $D$ is a union of right $U_K(\mathbb F_q)$-cosets. We conclude that
$$ \mathcal O_{\hat C}(T_w) = \mu(D) = \mu(U_K(\mathbb F_q)) \cdot \#(D/U_K(\mathbb F_q)). $$

To compute this cardinality, we count the number of cosets $gU_K(\mathbb F_q) \in P_K(\mathbb F_q) / U_K(\mathbb F_q)$ contained in $D$. Since $P_K(\mathbb F_q) = M_K(\mathbb F_q) U_K(\mathbb F_q)$, we can represent each such coset by a unique element $m \in M_K(\mathbb F_q)$. A coset $m U_K(\mathbb F_q)$ lies in $D$ if and only if $m \in B(\mathbb F_q) w B(\mathbb F_q)$ and $m$ belongs to the set $M_K(\mathbb F_q) \cdot C' U_{J'}(\mathbb F_q)$ for some pair $(J', \underline C') \sim (J, \underline C)$. However, the latter condition is precisely stating that $m$ is $M_K(\mathbb F_q)$-conjugate to an element of $C' U_{J'}(\mathbb F_q)$. Summing over all such elements $m \in M_K(\mathbb F_q)$ recovers exactly the formula claimed in Proposition \ref{prop:orbitalIntegralLift} (2).
\end{proof}

\section{Lusztig's pairing for finite groups of Lie types}\label{sec:lusztig_pairing}
To establish the linear independence of the canonical elements $\{T_{\mathcal{O}}\}_{\mathcal{O} \in Cl(W)}$ in the cocenter $\overline{H}$ for arbitrary Kac-Moody types, one key step in our approach relies on lifting trace functionals from spherical parabolic subalgebras. Consequently, we require a highly structured family of trace functionals on the Hecke algebras of finite Weyl groups that are amenable to this globalization. While the linear independence of the cocenter basis for finite types is classically known---following easily, for instance, from Tits' deformation theorem---that algebraic fact alone does not easily lift to indefinite types. 

Instead, we turn to the representation theory of finite groups of Lie type to construct these local building blocks. Studying orbital integrals of regular semisimple elements, like in Section~\ref{sec:orbital_integrals}, provides a rich family of trace functions that posess exactly the lifting property we need. Our main result of this section is the following.
\begin{theorem}\label{thm:lusztigPairingOverview}
    Let $G / \mathbb Z$ be a split Chevalley group with Weyl group $(W, S)$.
    Let $A = \mathbb Q[\q^{\pm 1/2}]$. There exists a uniquely determined pairing $\psi : A[\Cl(W)]\otimes_A A[\Cl(W)]\to A$ satisfying the following conditions:
    \begin{enumerate}
    \item (Orbital integrals of regular semisimple elements) For every finite field $\mathbb F_q$ and every regular semisimple element $s\in G(\mathbb F_q)$, the orbital integral of the $G(\mathbb F_q)$-conjugacy class of $s$ agrees with the map
    \begin{align*}
        \overline H_{\mathbb C}\to \mathbb C,~T_{\CO_1}\mapsto \frac{\# W}{\# C_{G(\mathbb F_q)}(s) \# \CO_2} \q^{\ell(w_0)} \psi(\CO_1,\CO_2)_{\q=q},
    \end{align*}
    where $\CO_2\in \Cl(W)$ is the conjugacy class that determines the type of the rational torus $C_{G(\overline{\mathbb F}_q)}(s)$.
    \item (Orbital integrals of Steinberg fibres) For every finite field $\mathbb F_q$ and every conjugation-stable set of semisimple element $C\subseteq G(\mathbb F_q)$, the map
    \begin{align*}
        H_{\mathbb C}\to\mathbb C,~h\mapsto \sum_{\substack{s\in C\\u\in G(\mathbb F_q)\text{ unip.}\\su=us}} h(su)
    \end{align*}
    is equal to the composition of the projection map $H_{\mathbb C}\to\overline H_{\mathbb C}$ and the map
    \begin{align*}
        \overline H_{\mathbb C}\to \mathbb C,~T_{\CO_1}\mapsto \q^{\ell(w_0)} \sum_{\CO_2\in \Cl(W)} \frac{\#(\torusO{\CO_2}(\mathbb F_q)\cap C)}{\#\torusO{\CO_2}(\mathbb F_q)} \psi(\CO_1,\CO_2)_{\q=q},
    \end{align*}
    where $\torusO{\CO_2}$ denotes any choice of a rational maximal torus of $G_{\overline{\mathbb F}_q}$ of type $\CO_2$.
    \item (Compatibility with Levi subgroups) For any subset $J\subseteq S$, denote the Levi subgroup by $M_J$ and its corresponding pairing by $\psi^{M_J}$. For $\CO_1\in \Cl(W_J)$ and $\hat \CO_1,\hat \CO_2\in \Cl(W)$ with $\CO_1\subseteq \hat{\CO}_1$, we have
    \begin{align*}
        \sum_{\substack{\CO_2\in \Cl(W_J)\\\CO_2\subseteq \hat \CO_2}} \psi^{M_J}(\CO_1,\CO_2) = \psi(\hat{\CO_1},\hat{\CO_2}).
    \end{align*}
    \item (Factorization) The matrix representing $\psi$ can be factorized as $\q^{-\ell(w_0)} NFC$ for three matrices $N, F, C\in \mathbb Q[\q^{\pm 1/2}]^{\Cl(W)\times \Cl(W)}$ such that
    \begin{itemize}
        \item $N$ (normalization) only has entries from $\mathbb Q$ and moreover it is invertible,
        \item $F$ (Fourier transform) only has entries from $\mathbb Q$ and moreover it is symmetric and
        \item $C$ (character table of character table of $H$) is invertible.
    \end{itemize}
    \item (Detecting support) If $\mathcal O_1,\mathcal O_2\in \Cl(W)$ are two conjugacy classes such that $\psi(\CO_1,\CO_2)\not\equiv 0\pmod{\q-1}$, then there exist elements $w_1\in (\CO_1)_{\min}$ and $w_2\in (\CO_2)_{\min}$ with $\supp(w_1) = \supp(w_2)$.
    \end{enumerate}
\end{theorem}

\subsection{The Deligne-Lusztig character}
In this section, we assume that $W$ is a finite Weyl group. Then $W$ is the Weyl group of a split finite group of Lie type $G(\mathbb{F}_{q})$. We regard $G(\mathbb{F}_{q})$ as the subgroup of fixed points in $G(\overline{\mathbb{F}}_{q})$ under the Frobenius automorphism $\sigma$. 

It is known that there is a natural bijection between the $G(\mathbb{F}_{q})$-conjugacy classes of $\sigma$-stable maximal tori of $G(\overline{\mathbb{F}}_{q})$ and $Cl(W)$ (see e.g.\ \cite[Application 3.23]{DM91}). This bijection is explicitly given as follows. Let $T$ be a fixed split maximal torus of $G(\mathbb{F}_{q})$. For any $\sigma$-stable maximal torus $T^{\prime}$, we have $T^{\prime}=gTg^{-1}$ for some $g\in G(\overline{\mathbb{F}}_{q})$. The condition that $T^{\prime}=\sigma(T^{\prime})$ implies that $g^{-1}\sigma(g)\in N_{G}(T)$. The type of $T^{\prime}$ is the image of $g^{-1}\sigma(g)$ under the map
$$ N_{G}(T)\longrightarrow N_{G}(T)/T\cong W\longrightarrow Cl(W). $$

For each conjugacy class $\mathcal{O}\in Cl(W)$, choose a $\sigma$-stable maximal torus $\mathcal{T}_{\mathcal{O}}$ of type $\mathcal{O}$. Then any character $\th$ of $\torusO{\CO}(\mathbb F_q)$ gives rise to the Deligne-Lusztig character $R_{\mathcal{O}}^{\theta}:=R_{\mathcal{T}_{\mathcal{O}}}^{\theta}$ of $G(\mathbb{F}_{q})$ in the sense of \cite{DL76}. The function $R_{\mathcal{O}}^{\theta}:G(\mathbb{F}_{q})\rightarrow\mathbb{C}$ is a virtual character. The pairing $\langle R_{\mathcal{O}}^{\theta},R_{\mathcal{O}^{\prime}}^{\theta^{\prime}}\rangle$ is zero unless $\mathcal{O}=\mathcal{O}^{\prime}$ and $\theta$ can be conjugated to $\theta^{\prime}$ through the normalizer of $\mathcal{T}_{\mathcal{O}}$ in $G(\mathbb{F}_{q})$ \cite[Theorem 6.8]{DL76}.

We denote the trivial character of $\torusO{\CO}(\mathbb F_q)$ by $1$. Then $R_{\CO}^1$ may have several irreducible constituents. Those irreducible representations of $G(\mathbb F_q)$ that occur in some $R^1_{\CO}$ are called the \emph{unipotent representations}.

The character $R^1_{\{1\}}$ of the trivial conjugacy class is of special interest. Recall that $T = \torusO{\{1\}}$ is our fixed split maximal torus, and we may choose a $\sigma$-stable Borel subgroup $B\subseteq G$ containing it. According to \cite[Chapter~11]{DM91}, we get the explicit description $R^1_{\{1\}} = \mathbb C[G(\mathbb F_q)/B(\mathbb F_q)]$.

Denote by $H_{\mathbb C}$ the base change of our Hecke algebra along the map $\mathbb Z[\q^{\pm 1}]\to \mathbb C$ which sends $\q$ to $q$. Then we can identify $H_{\mathbb C}$ with the space of $B(\mathbb F_q)$-biinvariant functions from $G(\mathbb F_q)$ to $\mathbb C$ (where multiplication becomes the convolution product).
This leads to an equivalence of categories between $G(\mathbb F_q)$-representations which are generated by their $B(\mathbb F_q)$-fixed points, and the category of $H$-representations. This yields the identity
        \begin{align*}
            \sum_{g\in B(\mathbb F_q) w B(\mathbb F_q)}\tr(g\mid V) = \tr(T_w\mid V^{B(\mathbb F_q)}).
        \end{align*}
for all $G(\mathbb F_q)$-representations $V$ and all $w\in W$. 

Let $C\subseteq G(\mathbb F_q)^{\mathrm{rs}}$ be a union of semisimple conjugacy classes. A standard exercise (cf.\ \cite[Exercise~12.21]{DM91}) shows that the function
\begin{align}
    \sum_{\CO\in \Cl(W)} \frac{\#\CO}{\# W \# \torusO{\CO}(\mathbb F_q)}\sum_{\substack{\theta\in \Irr(\torusO{\CO}(\mathbb F_q))\\ s\in \torusO{\CO}(\mathbb F_q)\cap C}} \theta(s)^{-1} R_{\CO}^\theta : G(\mathbb F_q)\to \mathbb C\label{eq:semisimplePartIndicatorFunction}
\end{align}
is one on all elements $g\in G(\mathbb F_q)$ whose Jordan decomposition $g = su = us$ satisfies $s\in C$, and zero on all other elements.

Note that the centralizer of any regular semisimple element $\tilde s\in G(\mathbb F_q)$ is a torus, hence connected. It follows that any element in $G(\mathbb F_q)$ is $G(\mathbb F_q)$-conjugate to $\tilde s$ (i.e.\ rationally conjugate) if and only if it is $G(\overline{\mathbb F}_q)$-conjugate (i.e.\ geometrically conjugate). In the case where $C$ is the $G(\mathbb F_q)$-conjugacy class of $\tilde s$, there is only one $\CO\in\Cl(W)$ with $\torusO{\CO}(\mathbb F_q)\cap C\neq\emptyset$ (namely, $\mathcal O$ is the type of the maximal torus $C_G(\tilde s)$). Suppose, without loss of generality, that this intersection contains $\tilde s$. For any $\theta\in \Irr(\torusO{\CO}(\mathbb F_q))$, we get
\begin{align*}
    \sum_{s\in \torusO{\CO}(\mathbb F_q)\cap C} \theta(s)^{-1} = \sum_{w\in N_G(\torusO{\CO})(\mathbb F_q)/T(\mathbb F_q)} (\theta\circ \Ad(w))(\tilde s)^{-1}.
\end{align*}
It follows that the formula in \eqref{eq:semisimplePartIndicatorFunction} simplifies to
\begin{align}
    \frac 1{\#\torusO{\CO}(\mathbb F_q)} \sum_{\theta\in \Irr(\torusO{\CO}(\mathbb F_q))} \theta(\tilde s)^{-1} R_{\CO}^\theta.\label{eq:rssIndicatorFunction}
\end{align}

\subsection{The pairing $\varphi$}

\begin{proposition}\label{prop:rsExists}
    Assume that $q$ is sufficiently large (depending on the Coxeter system $(W, S)$). Then each $\sigma$-stable maximal torus of $G(\overline{\mathbb F}_q)$ contains a regular semisimple element of $G(\mathbb F_q)$. 
\end{proposition}

\begin{remark}
Although Proposition \ref{prop:rsExists} requires $q$ to be large, this restriction does not compromise our main theorem. In Sections \ref{sec:6.3} and \ref{sec:6.4}, we promote this geometric pairing over $\mathbb{F}_q$ to a generic pairing over the ring $A = \mathbb{Q}[\q^{\pm 1/2}]$. Because a polynomial identity over $A$ is completely determined by its evaluations at infinitely many prime powers $q$, it suffices to establish the properties of the pairing for $q \gg 0$.
\end{remark}

\begin{proof}
    From \cite[Proposition~3.3.5]{Ca93}, we see that for every conjugacy class $\mathcal O\in \Cl(W)$, there is a monic polynomial $\chi_{\mathcal O}\in \mathbb Z[\q]$ of degree $\mathrm{rk}\,G$ such that $\# T_{\CO}(\mathbb F_q) = \chi_{\mathcal O}(q)$ for all prime powers $q>1$. Thus, for $q$ sufficiently large, we can ensure that $\chi_{\mathcal O}(q)\geq \frac 12 q^r$, where $r$ is the rank of $G$. 

    We note that for each non-regular element $s\in \CT_{\CO}(\mathbb F_q)$, the identity component of its centralizer $L := C_G(s)^\circ$ is connected reductive and contains $\CT_{\CO}$ as a proper subgroup.

Consider the root system of $(L,\mathcal{T}_{\mathcal{O}})$ over $\overline{\mathbb{F}}_{q}$. This is a $\sigma$-stable subsystem of the root system of $(G,\mathcal{T}_{\mathcal{O}})$, and it uniquely determines $L$. Thus, the number of possibilities for $L$ is bounded by $2^{\#\Phi}$, and its connected center $Z(L)^\circ$ is always an $\mathbb{F}_{q}$-rational torus whose rank is strictly less than $r$. Thus, by the above result, we can ensure that $\# Z(L)^\circ(\mathbb F_q) \leq 2 q^{r-1}$ for $q$ sufficiently large. Applying \cite[Proposition~2.3 and Corollary~3.13]{DM91} to any regular element in $Z(L)^\circ(\overline{\mathbb F}_q)$, we see that $\# Z(L)(\mathbb F_q) \leq \# W \# Z(L)^\circ(\mathbb F_q)$.

Under these assumptions, the number of non-regular (i.e., singular) elements in $\mathcal{T}_{\mathcal{O}}(\mathbb{F}_{q})$ is at most
$$ \#\{s\in \mathcal{T}_{\mathcal{O}}(\mathbb{F}_{q}) \mid s \text{ singular}\} \le \sum_{\substack{\mathcal T_{\mathcal O}\subseteq L=\sigma(L)\subseteq G \\ L \text{ reductive}}} \#Z(L)(\mathbb{F}_{q}) \le \# W\cdot 2^{\#\Phi} \cdot 2q^{r-1}. $$
For sufficiently large $q$, this number is strictly less than $\frac{1}{2}q^{r}\le\#\mathcal{T}_{\mathcal{O}}(\mathbb{F}_{q})$, completing the proof. 
\end{proof}

\begin{proposition}\label{prop:HtraceOnRS}
    Let $\CO\in\Cl(W)$ and $s\in \torusO{\CO}(\mathbb F_q)$ be a regular element. Let $\{s\}$ be the $G(\mathbb F_q)$-conjugacy class of $s$. Let $h\in H_{\mathbb C}$. Then
    \begin{align*}
        \sum_{g\in \{s\}} h(g) = \frac 1{\#\torusO{\CO}(\mathbb F_q)} \tr(h\mid (R^1_{\CO})^{B(\mathbb F_q)}).
    \end{align*}
\end{proposition}

\begin{remark}
    In particular, the $h$-average on $\{s\}$ is independent of the particular choice of $s$ or $T_{\CO}$. Moreover, it only depends on the image of $h$ in $\overline H$.
\end{remark}
\begin{proof}
    Using \eqref{eq:rssIndicatorFunction} we can write the left sum as
    \begin{align*}
    \frac 1{\#\torusO{\CO}(\mathbb F_q)} \sum_{\theta\in \Irr(\torusO{\CO}(\mathbb F_q))} \theta(\tilde s)^{-1} \sum_{g\in G(\mathbb F_q)}h(g)R_{\CO}^\theta(g)
    \end{align*}
    By our initial discussion of the algebra $H_{\mathbb C}$, the sum
    \begin{align*}
        \sum_{g\in G(\mathbb F_q)} R^\theta_{\CO}(g) h(g)
    \end{align*}
    is just $\tr(h\mid (R^\theta_{\CO})^{B(\mathbb F_q)})$. We note that the irreducible constituents of the $H$-representation $(R^\theta_{T'})^{B(\mathbb F_q)})$ corresponds to the irreducible $G(\mathbb F_q)$-representations that occur in both $R^1_{\{1\}}$ and $R^\theta_{T'}$. By \cite[Corollary~6.3]{DL76}, this set is empty unless $\theta=1$. So our above sum simplifies to
    \begin{align*}
    \frac 1{\#\torusO{\CO}(\mathbb F_q)} \tr(h\mid (R^1_{\CO})^{B(\mathbb F_q)}).
    \end{align*}
    This finishes the proof.
\end{proof}
\begin{definition}\label{def:phi}
    We define a pairing $\varphi^G : \Cl(W)\times\Cl(W)\to\mathbb C$ using the formula
    \begin{align*}
        \varphi^G(\mathcal O_1,\mathcal O_2) =  \frac{\#\CO_2}{\#W}q^{-\ell(w_0)}\tr(T_{\mathcal O_1}\mid (R^1_{\mathcal O_2})).
    \end{align*}
    Here, $T_{\CO_1}\in \overline H$ is the canonical basis element of the cocentre of $H$ associated with $\CO_1$. If there is no ambiguity, we will write $\varphi = \varphi^G$.
\end{definition}

    If $\torusO{\mathcal O_2}(\mathbb F_q)$ contains a regular element $s$, then by Proposition \ref{prop:HtraceOnRS},     
    \begin{equation}\label{eq:phi}
       \varphi(\mathcal O_1,\mathcal O_2)= (\# \torusO{O_2}(\mathbb F_q)) \frac{\#\CO_2}{\#W}q^{-\ell(w_0)}\#(B(\mathbb F_q) w B(\mathbb F_q)\cap \{s\}).
    \end{equation}
    
    Denoting the free $\mathbb C$-vector space over $\Cl(W)$ by $\mathbb C[\Cl(W)]$, we extend $\varphi$ to a $\mathbb C$-linear pairing $\mathbb C[\Cl(W)]\otimes \mathbb C[\Cl(W)]\to\mathbb C$. Let $\Cl(G(\mathbb F_q)^{\mathrm{rs}})$ be the set of regular-semisimple conjugacy classes of $G(\mathbb F_q)$. We get a natural map
\begin{align*}
    \mathbb C[\Cl(W)]\cong \overline H_{\mathbb C}\to \mathbb C[\Cl(G(\mathbb F_q)^{\mathrm{rs}})]^{\ast},
\end{align*}
sending $h\in H_{\mathbb C}$ and $\{s\}\in \Cl(G(\mathbb F_q)^{\mathrm{rs}})$ to $\sum_{g\in \{s\}} h(g)$. This map is computed explicitly using the pairing $\varphi$. In particular, the matrix defining $\varphi$ is non-degenerate if and only if the cocenter $\overline H_{\mathbb C}$ can be separated using regular semisimple conjugacy classes.

\subsection{Compatibility of $\varphi$ with Levi Subgroups}

The chosen normalization of $\varphi$ gives us compatibility with Levi subgroups.
\begin{lemma}\label{lem:phicompat}
    Let $J\subseteq S$ be any subset. Let $\CO_1\in \Cl(W_J)$ and $\hat{\CO}_1\in\Cl(W_J)$ with $\CO_1\subseteq \hat{\CO}_1$. Then for any $\hat{\CO}_2\in\Cl(W)$,
    \begin{align*}
        \sum_{\substack{\CO_2\in\Cl(W_J)\\\CO_2\subseteq \hat{\CO}_2}}\varphi^{M_J}(\mathcal O_1,\mathcal O_2) = \varphi^G(\hat{\mathcal O}_1,\hat{\mathcal O}_2).
    \end{align*}
\end{lemma}

\begin{proof}
    Let $w\in (\CO_1)_{\min}\subseteq (\hat{\CO}_1)_{\min}$ and set $M= M_J$.
    We wish to compare the actions of $T_w$ on the $G(\mathbb F_q)$-representation $R^1_{\CO_2}$ and its restriction to $M_J(\mathbb F_q)$. For any $G(\mathbb F_q)$-representation $V$, let ${}^\ast R^G_M V$ be the \emph{Harish-Chandra restriction}. It is a $M(\mathbb F_q)$-representation and the trace of $g\in M(\mathbb F_q)$ is given by (see \cite[Example 4.6 (iii)]{DM91})
    \begin{align*}
        \mathrm{tr}(g\mid {}^\ast R^G_M V) = \frac 1{\# U_J(\mathbb F_q)} \sum_{u\in U_J(\mathbb F_q)} \mathrm{tr}(gu\mid V).
    \end{align*}
The multiplication induces a bijection $U_J(\BF_q) \times (M\cap B)(\mathbb F_q) w (M\cap B)(\mathbb F_q) \to B(\mathbb F_q) w B(\mathbb F_q)$. Let $T^M_w \in H_M$ and $T^G_w \in H$ be the standard element associated with the element $w$ in the Hecke algebra $H_M$ and $H$ respectively. It follows that
    \begin{align*}
        \mathrm{tr}(T^M_w\mid ({}^\ast R^G_M V)^{B(\mathbb F_q)})=& \sum_{g\in (M\cap B)(\mathbb F_q) w (M\cap B)(\mathbb F_q)} \mathrm{tr}(g\mid {}^\ast R^G_M V)
        \\=&\frac 1 {\# U_J(\mathbb F_q)} \sum_{g\in B(\mathbb F_q) w B(\mathbb F_q)} \mathrm{tr}(g\mid V)
        \\ =& q^{-(\ell(w_0) - \ell(w_{0,J}))} \mathrm{tr}(T_w\mid V).
    \end{align*}

Now we specialize to $V = R^1_{\hat{\CO}_2}$. By the Mackey formula (see \cite[Theorem~11.13]{DM91}), we have 
    \begin{align*}
        {}^\ast R^G_M R^1_{\hat{\CO}_2} = \sum_{\substack{v\in M(\mathbb F_q)\backslash G(\mathbb F_q) / \torusO{\hat{\CO}_2}(\mathbb F_q)\\ v \torusO{\hat{\CO}_2} v^{-1}\subseteq M}} R^1_{v \torusO{\hat{\CO}_2}v^{-1},M},
    \end{align*}
    where $R^1_{T',M}$ denotes the Deligne-Lusztig character of $M$ with respect to the $\mathbb F_q$-rational torus $T'\subseteq M$.

    Given any $v\in G(\mathbb F_q)$ with $T' := v \torusO{\hat{\CO}_2}v^{-1}\subseteq M$, the torus $T'$ is an $\mathbb F_q$-rational maximal torus of $M$, and its type $\CO_2\in \Cl(W_J)$ necessarily satisfies $\CO_2\subseteq \hat{\CO}_2$.

Let now $\CO_2\in \Cl(W_J)$ with $\CO_2\subseteq \hat{\CO}_2$ and $\torusO{\CO_2}\subseteq M$ be a $\mathbb F_q$-rational maximal torus of type $\CO_2$. Then $\torusO{\CO_2}$ is $G(\mathbb F_q)$-conjugate to $\torusO{\hat{\CO}_2}$, and we get
    \begin{align*}
        &\#\{v\in M(\mathbb F_q)\backslash G(\mathbb F_q)/\torusO{\hat{\CO}_2}(\mathbb F_q)\mid v \torusO{\hat{\CO}_2} v^{-1} \subseteq M\text{ is of type }\CO_2\}
        \\&=\#\{v\in M(\mathbb F_q)\backslash G(\mathbb F_q)/\torusO{\CO_2}(\mathbb F_q)\mid v \torusO{\CO_2} v^{-1} \subseteq M\text{ is of type }\CO_2\}
        \\&=\frac{\#\{v\in G(\mathbb F_q)/\torusO{\CO_2}(\mathbb F_q)\mid v \torusO{\CO_2} v^{-1} = \torusO{\CO_2}\}}{\#\{v\in M(\mathbb F_q)/\torusO{\CO_2}(\mathbb F_q)\mid v \torusO{\CO_2} v^{-1} = \torusO{\CO_2}\}}.
    \end{align*}
    
    In this final expression, the computation of denominator and numerator use identical steps. For the numerator, we use \cite[Corollary~3.13, Proof of Corollary~11.16]{DM91} to compute
    \begin{align*}
        \#\{v\in G(\mathbb F_q)/\torusO{\CO_2}(\mathbb F_q)\mid v \torusO{\CO_2} v^{-1} = \torusO{\CO_2}\}
        &=\#(N_G(\torusO{\CO_2}(\mathbb F_q))/\torusO{\CO_2}(\mathbb F_q))
        \\&=\#((N_G(\torusO{\CO_2}) / \torusO{\CO_2})(\mathbb F_q))
        \\&=\#C_W(w_{\CO_2}) = \frac{\#W}{\#\hat{\CO}_2}.
    \end{align*}
    Here, $w_{\CO_2}$ denotes any representative of $\CO_2$, and $C_W(w_{\CO_2})$ is its centralizer in $W$.

    We summarize that
    \begin{align*}
        {}^\ast R^G_M R^1_{\hat{\CO}_2} = \sum_{\substack{\CO_2\in \Cl(W_J)\\\CO_2\subseteq \hat{\CO}_2}}\frac{\# W}{\#\hat{\CO}_2}\frac{\#\CO_2}{\# W_J}R^1_{\CO_2,M}.
    \end{align*}
By definition, 
    \begin{gather*}
        q^{-\ell(w_0) + \ell(w_{0,J})} \mathrm{tr}(T_w\mid R^1_{\hat{\CO}_2})=q^{\ell(w_{0,J})} \frac{\# W}{\# \hat{\CO_2}} \varphi(\hat{\CO}_1, \hat{\CO}_2), \\
        \sum_{\substack{\CO_2\in \Cl(W_J)\\\CO_2\subseteq \hat{\CO}_2}} \frac{\# W}{\#\hat{\CO}_2}\frac{\#\CO_2}{\# W_J}\mathrm{tr}(T^M_w\mid R^1_{\CO_2,M})=\sum_{\substack{\CO_2\in \Cl(W_J)\\\CO_2\subseteq \hat{\CO}_2}} \frac{\# W}{\#\hat{\CO}_2}q^{\ell(w_{0,J})} \varphi(\CO_1, \CO_2).
    \end{gather*}
The lemma is proved. 
\end{proof}

\begin{remark}
    In the case where $T_{\hat{\CO}_2}(\mathbb F_q)$ contains a regular element,  Lemma~\ref{lem:phicompat} follows from \eqref{eq:phi} and more elementary arguments, i.e.\ which don't use the general version of the Mackey formula: Note that for every regular semisimple conjugacy class $\{s\}\subseteq G(\mathbb F_q)$, we get
    \begin{align*}
        \{s\}\cap B(\mathbb F_q) w B(\mathbb F_q) = (\{s\}\cap (B\cap M_J)(\mathbb F_q) w (B\cap M_J)(\mathbb F_q)) U_J(\mathbb F_q).
    \end{align*}
    Now it remains to study the intersection $\{s\}\cap M_J(\mathbb F_q)$, which is a union of regular semisimple $M_J(\mathbb F_q)$-conjugacy classes. If $\CO_2\in \Cl(W_J)$ is any class, then
    \begin{align*}
    &\#\{\{s'\}\subseteq M_J(\mathbb F_q)\text{ rss class}\mid \{s'\}\subseteq \{s\}\text{ and }\mathrm{type}(C_{M_J}(s')) = \CO_2\}
    \\&=\#\{v\in M_J(\mathbb F_q) \backslash G(\mathbb F_q) / C_G(s')(\mathbb F_q)\mid\\&\qquad\qquad vs'v^{-1}\in M_J(\mathbb F_q)\text{ and }\mathrm{type}(C_{M_J}(v s' v^{-1})) = \CO_2\}.
    \end{align*}
    The condition $v s' v^{-1}\in M_J(\mathbb F_q)$ is equivalent to $v C_J(s') v^{-1}\subseteq M$, and if it is true, then $v C_J(s') v^{-1}$ is equal to $C_{M_J}(v s' v^{-1})$. So if we set $T' = C_J(s')$, then
    \begin{align*}
    &\#\{\{s'\}\subseteq M_J(\mathbb F_q)\text{ rss class}\mid \{s'\}\subseteq \{s\}\text{ and }\mathrm{type}(C_{M_J}(s')) = \CO_2\}
    \\&=\{v\in M_J(\mathbb F_q) \backslash G(\mathbb F_q) / T'(\mathbb F_q)\mid v T' v^{-1}\subseteq M\text{ is of type }\CO_2\}.
    \end{align*}
    This cardinality was computed in the proof of Lemma~\ref{lem:phicompat}. If we set $\hat{\CO}_2 = \mathrm{type}(T')$, then this cardinality is equal to
    \begin{align*}
        \begin{cases} \frac{\# W \# \CO_2}{\# W_J \# \hat{\CO}_2},&\CO_2\subseteq \hat{\CO}_2,\\
        0,&\CO_2\not\subseteq \hat{\CO}_2
        \end{cases}
    \end{align*}
    This gives an alternative, elementary proof of Lemma~\ref{lem:phicompat} for sufficiently large $q$.
\end{remark}

\subsection{Evaluation of the pairing}\label{sec:6.3}\label{subsec:evaluation-phi}
We give a general recipe for evaluating the pairing  $\varphi$. Our main goal is to determine its kernel (which will often be trivial).

By Tits' deformation theorem, given any algebraically closed field $k$ and a ring homomorphism $\BZ[\q^{\pm 1}]\to k$, the base change $H_k$ is isomorphic to the group algebra $k[W]$. In particular, $H_k$ is semisimple and we can express the values of irreducible characters on the standard basis elements of $\overline H_k$ as formulas involving the image of $\q$ in $k$. Over types $E_7$ and $E_8$, these formulas involve taking square roots of the image of $\q$ in $k$. Therefore, when studying the character table of $H$, we generally consider the base change $H_A$ of $H$ to $A$, where $A = \mathbb Q[\q^{\pm 1/2}]\supseteq \BZ[\q^{\pm 1}]$.

As in \cite[Section~3.3]{Lu-Characters}, every $\mathbb Q[W]$-module $E$ lifts to a canonical $H_A$-module $E(\q)$ whose specialization under $\q=1$ recovers $E$. Now the base change according to $A\to \mathbb C,~\q^{1/2}\mapsto \sqrt q$ turns $E(\q)$ into a $H_{\mathbb C}$-module, which we denote by $E_{\mathbb C}$. This gives an explicit bijection between irreducible $\mathbb Q[W]$-modules and irreducible $H_{\mathbb C}$-modules.

Consequently, for any virtual $H_{\mathbb C}$-representation $V$ and every $h\in H_{\mathbb C}$, we have
\begin{align*}
    \tr(h\mid V) = \sum_{E\in\Irr(\mathbb Q[W])} \tr(h\mid E_{\mathbb C}) [V: E_{\mathbb C}],
\end{align*} where $[V: E_{\mathbb C}]$ is the multiplicity of $E_{\mathbb C}$ in $V$. 

To compute these multiplicities for the Deligne-Lusztig characters, we follow Lusztig's parametrization of unipotent characters \cite{Lu-Characters}. For each $E\in \operatorname{Irr}(\mathbb{Q}[W])$, we define the \emph{almost character} $R_E$ as the $\mathbb{Q}$-linear combination of $G(\mathbb F_q)$-characters
$$ R_E := \sum_{\mathcal{O}\in Cl(W)} \operatorname{tr}(w_{\mathcal{O}}\mid E) \frac{\#\mathcal{O}}{\#W} R^1_{\mathcal{O}}, $$
where $w_{\mathcal{O}} \in \mathcal{O}$ is any representative. The orthogonality relations for the character table of $W$ give the inversion formula
$$ R_{\mathcal{O}}^1 = \sum_{E\in \Irr(\mathbb{Q}[W])} \tr(w_{\mathcal{O}}\mid E) R_E. $$

It follows that
\begin{align*}
\frac{\#W}{\#\CO_2}q^{\ell(w_0)}\varphi(\CO_1, \CO_2) =& \tr(T_{\CO_1}\mid R_{\CO_2}^1)\\ = &\sum_{E\in \Irr(\mathbb Q[W])} \tr(w_{\CO_2}\mid E) \tr(T_{\CO_1}\mid R_E)
\\=&\sum_{E, E'\in \Irr(\mathbb Q[W])} \tr(w_{\CO_2}\mid E) \tr(T_{\CO_1}\mid E'_{\mathbb C}) [R_E: E'_{\mathbb C}].
\end{align*}

The computation of these multiplicities was established by Lusztig in \cite{Lu-Characters}. He partitions $\operatorname{Irr}(\mathbb{Q}[W])$ into \emph{families}, which correspond precisely to the two-sided cells of $W$ defined via Kazhdan-Lusztig theory (or equivalently, Joseph's primitive ideals). The disjointness theorem \cite[Theorem 6.17]{Lu-Characters} asserts that given $E, E'\in \operatorname{Irr}(\mathbb{Q}[W])$, the virtual representations $R_E$ and $R_{E'}$ share an irreducible constituent if and only if $E$ and $E'$ belong to the same family.

To each family $\mathfrak c$, Lusztig associates a finite group $\mathcal G_{\mathfrak c}$ (often an elementary abelian 2‑group, or a small symmetric group in exceptional types). Then the irreducible representations $E\in \mathfrak c$ are parametrized by a finite set $\mathcal M(\mathcal G_{\mathfrak c})$ built from conjugacy data of $\mathcal G_{\mathfrak c}$.

By \cite[Proposition 3.9 and Theorem 4.23]{Lu-Characters}, the multiplicity for $E, E' \in \mathfrak{c}$ is governed by a Fourier transform matrix on $\mathcal{M}(\mathcal{G}_{\mathfrak{c}})$. Explicitly,
$$ [R_E : E'_{\mathbb{C}}] = \Delta(\bar{x}_\rho)\{\bar{x}_{\rho}, x_E\}, $$
where $\{\cdot,\cdot\}$ denotes the Fourier pairing on $\mathcal{M}(\mathcal{G}_{\mathfrak{c}})$, and $\Delta : \mathcal{M}(\mathcal{G}_{\mathfrak{c}})\to \{\pm 1\}$ is a highly constrained sign function introduced by Lusztig. This function $\Delta$ is identically $1$ unless $W$ contains a parabolic subgroup of type $E_7$, in which case it is explicitly known. The element $x_E \in \mathcal{M}(\mathcal{G}_{\mathfrak{c}})$ parametrizes the representation $E$. By \cite[Proposition 12.6]{Lu-Characters}, the element $\bar{x}_{\rho}$ (for $\rho = \rho_{E'}$) parametrizing the dual character is exactly $x_{E'}$.

Hence, evaluating our pairing reduces to understanding the kernel of the block-diagonal matrix 
$$ (\{x_{E'}, x_{E}\})_{E, E'\in\operatorname{Irr}(\mathbb{Q}[W])}\in \mathbb{Q}^{\operatorname{Irr}(\mathbb{Q}[W])\times \operatorname{Irr}(\mathbb{Q}[W])}, $$
which is completely independent of $q$. Here, we set $\{x_{E'}, x_E\}=0$ if $E$ and $E'$ belong to different families.

\subsection{Generalized pairing and specialization at $\q=1$}\label{sec:6.4}
Using the framework established above, we can define a generic version of our pairing over the ring $A = \mathbb{Q}[\q^{\pm 1/2}]$ that treats all possible choices of base fields simultaneously.

\begin{definition} \label{def:psi_generic}
Let $A[Cl(W)]$ be the free $A$-module with basis $Cl(W)$. We define the $A$-bilinear pairing
$$ \psi: A[Cl(W)] \otimes_A A[Cl(W)] \longrightarrow A $$
by the formula
$$ \psi(\mathcal{O}_1 \otimes \mathcal{O}_2) := \frac{\#\mathcal{O}_2}{\#W} q^{-l(w_0)} \sum_{E,E'\in \operatorname{Irr}(\mathbb{Q}[W])} \operatorname{tr}(w_{\mathcal{O}_2}\mid E) \operatorname{tr}(T_{\mathcal{O}_1}\mid E'(q)) \{x_{E'}, x_E\} \Delta(x_{E'}). $$
Let $\psi_1: \mathbb{Q}[Cl(W)] \otimes_{\mathbb{Q}} \mathbb{Q}[Cl(W)] \to \mathbb{Q}$ denote its specialization at $q=1$.
\end{definition}

By definition, the specialization of $\psi$ under $A\xrightarrow{\q^{1/2}\mapsto \sqrt{q}}\mathbb C$ equals $\varphi$.

We define the (left) kernel of the generic pairing $\psi$ as the $A$-submodule
$$ K := \{v \in A[Cl(W)] \mid \psi(v \otimes \mathcal{O}) = 0 \text{ for all } \mathcal{O} \in Cl(W)\}, $$
and we similarly denote the kernel of the specialized pairing $\psi_1$ by $K_1 \subseteq \mathbb{Q}[Cl(W)]$. The structural properties of these pairings, and in particular the dimensions of their kernels, are completely governed by the Fourier transform matrix.

\begin{proposition} \label{prop:lusztigPairingConsequences}
Let $M$ be the square matrix indexed by $\operatorname{Irr}(\mathbb{Q}[W])$ with entries given by Lusztig's Fourier pairing $\{x_{E'}, x_E\}$. The kernels $K$ and $K_1$ are determined by the kernel of $M$:
\begin{enumerate}
    \item If $M$ is invertible, then $K = 0$ and $K_1 = 0$ (i.e., the pairings are non-degenerate).
    \item If $\ker(M)$ is one-dimensional, generated by the difference of basis vectors $[E_1] - [E_2]$, then $K$ and $K_1$ are also one-dimensional. They are spanned by the element whose $\q=1$ specialization has the form
        \begin{align*}
            \sum_{\mathcal O\in \Cl(W)} (\tr(w_{\CO}\mid E_1) - \tr(w_{\CO}\mid E_2))\mathbb Q^{\ast} \CO\in \mathbb Q[\Cl(W)].
        \end{align*}
\end{enumerate}
\end{proposition}

\begin{proof}
Let $C = (\operatorname{tr}(T_{\mathcal{O}}\mid E'(\q)))_{\mathcal{O}, E'}$ be the character table of the Hecke algebra $H_A$, and let $D$ be the diagonal matrix of signs $\Delta(x_{E'})$. By Definition \ref{def:psi_generic}, the matrix representing $\psi$ factors as $C^T D M C'$ (where $C'$ is the character table of $W$). 

At $\q=1$, both $C$ and $C'$ specialize to the standard character table of $W$, which is invertible. Thus, both $C$ and $C'$ are invertible matrices over the fraction field of $A$. Because $\Delta$ takes values in $\{\pm 1\}$, $D$ is also invertible. Therefore, the kernel of the pairing matrix is naturally isomorphic to $\ker(M)$. The explicit generator in part (2) follows directly from inverting the character table of $W$ via the standard orthogonality relations.
\end{proof}

\subsection{Proof of Theorem~\ref{thm:lusztigPairingOverview}}
It is clear from Proposition~\ref{prop:rsExists} that property (1) alone determines the abstract pairing $\psi$. By definition of the pairing $\psi$ we constructed and Proposition~\ref{prop:HtraceOnRS}, our pairing $\psi$ indeed satisfies property (1).

To show that our pairing $\psi$ satisfies property (2) of Theorem~\ref{thm:lusztigPairingOverview}, we argue just as in the proof of Proposition~\ref{prop:HtraceOnRS} together with formula~\eqref{eq:semisimplePartIndicatorFunction}.

Condition (3) is a consequence of Lemma~\ref{lem:phicompat} together with the observation from Proposition~\ref{prop:rsExists} that the pairings $\varphi^G / \mathbb F_q$ for infinitely many choices of $q$ indeed determine $\psi$.

The factorization claimed in condition (4) is constructed explicitly in Section~\ref{subsec:evaluation-phi}. As noted in \cite[\S 4.14]{Lu-Characters}, the Fourier matrix $F$ embeds into a Hermitian matrix, implying it is symmetric.

Let us prove (5). From the explicit construction of the factorization in (4), it follows, that matrix representing $\psi_1 := \psi_{\q=1}$ is also symmetric up factors in $\mathbb Q^\ast$; i.e.\ $\psi_1(\mathcal{O}_1 \otimes \mathcal{O}_2) \ne 0$ implies $\psi_1(\mathcal{O}_2 \otimes \mathcal{O}_1) \ne 0$.

Suppose $\psi_1(\mathcal{O}_1 \otimes \mathcal{O}_2) \ne 0$. This implies the generic pairing $\psi$ is non-zero, meaning its geometric realization over $\mathbb{F}_q$ evaluates to non-zero for infinitely many prime powers $q$. Pick a sufficiently large $q$ and a representative $w_1 \in (\mathcal{O}_1)_{\min}$ with support $J = \operatorname{supp}(w_1)$. By \eqref{eq:phi}, the Schubert cell $B(\mathbb{F}_q)w_1 B(\mathbb{F}_q) \subseteq P_J(\mathbb{F}_q)$ must intersect the conjugacy class of a regular semisimple element $s$ of type $\mathcal{O}_2$.

Because $s \in P_J(\mathbb{F}_q)$ is regular, its conjugation action on the unipotent radical $U_J(\mathbb{F}_q)$ is fixed-point free. Thus, $s$ is $U_J(\mathbb{F}_q)$-conjugate to an element $s' \in M_J(\mathbb{F}_q)$. The relative type of $s'$ inside the Levi subgroup $M_J$ defines a class $\mathcal{O}'_2 \in Cl(W_J)$, which forces $\mathcal{O}_2 \cap W_J \ne \emptyset$. This guarantees we can find a minimal length representative $w_2 \in (\mathcal{O}_2)_{\min}$ such that $\operatorname{supp}(w_2) \subseteq J = \operatorname{supp}(w_1)$.

Since $F$ is symmetric, applying the exact same geometric argument to $\psi_1(\mathcal{O}_2 \otimes \mathcal{O}_1) \ne 0$ yields a minimal length element $w_1' \in (\mathcal{O}_1)_{\min}$ with $\operatorname{supp}(w_1') \subseteq \operatorname{supp}(w_2)$. By Proposition \ref{prop:LSalgo}, minimal length elements in the same conjugacy class have supports of equal cardinality. 

Since $\operatorname{supp}(w_1') \subseteq \operatorname{supp}(w_2) \subseteq \operatorname{supp}(w_1)$ and $|\operatorname{supp}(w_1')| = |\operatorname{supp}(w_1)|$, we conclude $\operatorname{supp}(w_2) = \operatorname{supp}(w_1)$.

This finishes the proof of Theorem~\ref{thm:lusztigPairingOverview}.

\section{Class polynomials for spherical conjugacy classes}\label{sec:7}\label{sec:spherical}
\subsection{Spherical conjugacy classes}
By definition, a conjugacy class $\mathcal{O} \in Cl(W)$ is called \emph{spherical} if $\mathcal{O}\cap W_{J}\ne\emptyset$ for some spherical subset $J\subseteq S$. This is equivalent to the condition that the elements of $\mathcal{O}$ have finite order.

To each conjugacy class $\mathcal{O} \in Cl(W)$, we associate an equivalence class of supports $[J]=[J]_{S}$. This class consists of all subsets $J\subseteq S$ such that there exists $w\in\mathcal{O}_{\min}$ with $J=\operatorname{supp}(w)$. The equivalence relation is defined as $J^{\prime}\sim J$ if $\operatorname{Iso}_{W}(J,J^{\prime})\ne\emptyset$. This relation is explicitly described by the Lusztig-Spaltenstein algorithm (cf.\ Proposition \ref{prop:LSalgo}). We write $\operatorname{supp}(\mathcal{O}_{\min})=[J]_{S}$. In this case, the intersection $\mathcal{O}\cap W_{J}$ is a disjoint union of conjugacy classes of $W_{J}$, and the group $\operatorname{Aut}_{W}(J)$ acts transitively on these classes.

Let $A=\mathbb{Q}[\q^{\pm 1/2}]$ as before, and let $H_{A}$ be the base change of $H$ under $\mathbb{Z}[\q^{\pm 1}]\hookrightarrow A$. Our objective in this section is to globalize the local pairings constructed on spherical parabolic subalgebras to form well-defined trace functionals on the entirety of $H_A$. The main result of this section establishes the existence of these global functions.

\begin{theorem} \label{thm:psi_global}
Let $\mathcal{O} \in Cl(W)$ be a spherical conjugacy class. Then there exists a unique $A$-linear trace map $\Psi_{\mathcal{O}}: \overline{H}_{A} \to A$ such that for any conjugacy class $\mathcal{O}^{\prime}\in Cl(W)$ with support $[J]=\operatorname{supp}(\mathcal{O}^{\prime}_{\min})$, we have
$$ \Psi_{\mathcal{O}}(T_{\mathcal{O}^{\prime}}) = 
\begin{cases} 
\displaystyle \sum_{\substack{\mathcal{O}_{2}\in Cl(W_{J}) \\ \mathcal{O}_{2}\subseteq\mathcal{O}}}\psi_J(\mathcal{O}_{1}\otimes\mathcal{O}_{2}), & \text{if } J \text{ is spherical;} \\ 
0, & \text{otherwise,} 
\end{cases} $$
where $\mathcal{O}_{1}\in Cl(W_{J})$ is chosen such that $\mathcal{O}_{1}\subseteq\mathcal{O}^{\prime}$, and $\psi_J$ is the generic pairing for the spherical parabolic subsystem $(W_J, J)$ from Definition \ref{def:psi_generic}.
\end{theorem}

The uniqueness of $\Psi_{\mathcal{O}}$ follows immediately from Proposition \ref{prop:cocenterSpanning}, as the canonical elements $T_{\mathcal{O}^{\prime}}$ span the cocenter $\overline{H}_{A}$. The crux of the matter is establishing existence. To achieve this, we will rely heavily on the orbital integral machinery for Kac-Moody groups developed in Section \ref{subsec:lifting-ss}.

\subsection{Globalizing the pairing $\varphi$}
Let $q \gg 0$ be a prime power such that for every spherical subset $J \subseteq S$, every $\mathbb F_q$-rational maximal torus in $\underline{M}_{J}=M_J/Z(M_J)^{\circ}$ contains at least one $\mathbb{F}_q$-rational regular element (cf.\ Proposition \ref{prop:rsExists}). 

For every spherical conjugacy class $\mathcal{O} \in Cl(W)$, we fix a minimal length element $w_{\mathcal{O}} \in \mathcal{O}_{\min}$ and set $J := \operatorname{supp}(w_{\mathcal{O}})$. We then choose a regular semisimple element $\tilde{s}_{\mathcal{O}} \in M_J(\mathbb{F}_q)$ whose type in $Cl(W_J)$ is exactly the $W_J$-conjugacy class of $w_{\mathcal{O}}$. Associated to $J$ and $\tilde{s}_{\mathcal{O}}$, we construct the $G(\mathbb{F}_q)$-conjugation invariant set $\hat{C}_{\mathcal{O}}$ as in Section~\ref{subsec:lifting-ss}. We denote the orbital integral of $\hat{C}_{\mathcal{O}}$ by $\Phi_{\mathcal{O}} : H_{\mathbb{C}} \to \mathbb{C}$, which factors through the cocenter $\overline{H}_{\mathbb{C}}$ by Proposition~\ref{prop:orbitalIntegralLift}. Note that, a priori, $\Phi_{\CO}$ depends on the choices made above.

To express $\Phi_{\mathcal{O}}$ in terms of our pairing, we define the \emph{rank} of a conjugacy class to be the cardinality of its minimal support.

\begin{proposition} \label{prop:PhiValues}
Let $w \in W$ be a minimal length element of its conjugacy class, with support $K := \operatorname{supp}(w) \subseteq S$.
\begin{enumerate}
    \item If $K$ is non-spherical, then $\Phi_{\mathcal{O}}(T_w) = 0$.
    \item If $K$ is spherical, let $\mathcal{O}_1 \in Cl(W_K)$ be the $W_K$-conjugacy class of $w$. Then we can express the orbital integral as
    $$ \Phi_{\mathcal{O}}(T_w) = \sum_{\mathcal{O}_2 \in Cl(W_K)} c_{\mathcal{O}_2} \varphi(\mathcal{O}_1, \mathcal{O}_2), $$
    where $\varphi$ is the pairing for $W_K$ from Definition \ref{def:phi}, and $c_{\mathcal{O}_2} \in \mathbb{Q}_{\ge 0}$ are constants (depending on $\mathcal{O}$ and $K$, but not on $w$) satisfying:
    \begin{itemize}
        \item The coefficients $c_{\mathcal{O}_2}$ are invariant under the action of $\operatorname{Aut}(W, K)$ on $Cl(W_K)$.
        \item If $\mathcal{O}_2 \subseteq \mathcal{O}$, then $c_{\mathcal{O}_2} > 0$.
        \item If $c_{\mathcal{O}_2} > 0$ but $\mathcal{O}_2 \not\subseteq \mathcal{O}$, then $\operatorname{rank}(\mathcal{O}_2) > \operatorname{rank}(\mathcal{O})$.
    \end{itemize}
\end{enumerate}
\end{proposition}

\begin{proof}
Part (1) follows immediately from Proposition ~\ref{prop:orbitalIntegralLift} (1). For part (2), applying Proposition ~\ref{prop:orbitalIntegralLift} (2) yields
$$ \Phi_{\mathcal{O}}(T_w) = \mu(U_K(\mathbb{F}_q)) \sum_{g \in \hat{D}} T_w(g), $$
where $\hat{D} \subseteq M_K(\mathbb{F}_q)$ is the set of elements that are $M_K(\mathbb{F}_q)$-conjugate to an element of $C^{\prime}U_{J^{\prime}}(\mathbb{F}_q)$ for some pair $(J^{\prime}, \underline{C}^{\prime}) \sim (J, \underline{C})$ with $J^{\prime} \subseteq K$.

For an element $g \in M_K(\mathbb{F}_q)$ with Jordan decomposition $g = su = us$, Lemma \ref{lem:Philemma} dictates that $g \in \hat{D}$ if and only if its semisimple part $s$ belongs to 
$$ D := \bigcup_{\substack{(J^{\prime}, \underline{C}^{\prime}) \sim (J, \underline{C}) \\ J^{\prime} \subseteq K}} M_K(\mathbb{F}_q) \cdot C^{\prime}. $$
By \eqref{eq:semisimplePartIndicatorFunction}, the indicator function for $D$ allows us to rewrite the sum over $\hat{D}$ in terms of Deligne-Lusztig characters. Following the identical trace extraction used in the proof of Proposition \ref{prop:HtraceOnRS}, we deduce
$$ \Phi_{\mathcal{O}}(T_w) = \mu(U_K(\mathbb{F}_q)) \sum_{\mathcal{O}_2 \in Cl(W_K)} \frac{\#\mathcal{O}_2 \cdot \#(\mathcal{T}_{\mathcal{O}_2}(\mathbb{F}_q) \cap D)}{\#W_K \cdot \#\mathcal{T}_{\mathcal{O}_2}(\mathbb{F}_q)} \operatorname{tr}(T_w \mid R_{\mathcal{O}_2}^1). $$
Rewriting this via Definition \ref{def:phi}, we obtain the desired sum with coefficients
$$ c_{\mathcal{O}_2} := \mu(U_K(\mathbb{F}_q)) q^{l(w_{0, K})} \frac{\#(\mathcal{T}_{\mathcal{O}_2}(\mathbb{F}_q) \cap D)}{\#\mathcal{T}_{\mathcal{O}_2}(\mathbb{F}_q)}. $$
These coefficients are manifestly non-negative and $\operatorname{Aut}(W, K)$-invariant. If $\mathcal{O}_2 \subseteq \mathcal{O}$, then $\mathcal{T}_{\mathcal{O}_2}$ intersects $C$ by construction, so $\mathcal{T}_{\mathcal{O}_2}(\mathbb{F}_q) \cap D \ne \emptyset$, making $c_{\mathcal{O}_2} > 0$.

Finally, suppose $c_{\mathcal{O}_2} > 0$ but $\mathcal{O}_2 \not\subseteq \mathcal{O}$. Then $\mathcal{T}_{\mathcal{O}_2}(\mathbb{F}_q) \cap D \ne \emptyset$, meaning there exists $(J^{\prime}, \underline{C}^{\prime}) \sim (J, \underline{C})$ and an element $s \in C^{\prime}$ contained in a torus of type $\mathcal{O}_2$. Let $L = C_{M_K}(s)^\circ$ and $T^{\prime} = Z(L)^\circ$. Because $T^{\prime}$ is contained in every torus containing $s$, we have $T^{\prime} \subseteq M_{J^{\prime}}$, and its projection to $\underline{M}_{J^{\prime}}$ contains a regular semisimple element. Thus, ${T}^{\prime} Z(M_{J^{\prime}})^\circ$ is a maximal torus inside $M_{J^{\prime}}$ containing $s$. 
Let $J_2 \subseteq K$ be the minimal support of $\mathcal{O}_2$. The torus $\mathcal{T}_{\mathcal{O}_2}$ is $M_K(\mathbb{F}_q)$-conjugate to a torus inside $P_{J_2} \cap M_K$, which implies $T^{\prime}$ is contained in some $M_K(\mathbb{F}_q)$-conjugate $P^{\prime}$ of $P_{J_2} \cap M_K$. Consequently, $(P^{\prime} \cap M_{J^{\prime}}) Z(M_{J^{\prime}})^\circ$ forms a parabolic subgroup of $M_{J^{\prime}}$ containing $T^{\prime}$. Because $s$ is regular elliptic in $M_{J^{\prime}}$, this parabolic must be the entirety of $M_{J^{\prime}}$. This dimensional constraint forces $\#J_2 > \#J^{\prime} = \#J$, establishing that $\operatorname{rank}(\mathcal{O}_2) > \operatorname{rank}(\mathcal{O})$.
\end{proof}

The filtered structure of the coefficients $c_{\mathcal{O}_2}$ allows us to invert the system and isolate the pairing $\varphi$.

\begin{theorem} \label{thm:phi_global}
Let $\hat{\mathcal{O}} \in Cl(W)$ be a spherical conjugacy class. There exists a linear combination $\varphi_{\hat{\mathcal{O}}}: \overline{H}_{\mathbb{C}} \to \mathbb{C}$ of the orbital integrals $\{\Phi_{\mathcal{O}} \mid \operatorname{rank}(\mathcal{O}) \le \operatorname{rank}(\hat{\mathcal{O}})\}$ such that for any spherical subset $K \subseteq S$ and minimal length element $w \in W_K$,
$$ \varphi_{\hat{\mathcal{O}}}(T_w) = \sum_{\substack{\mathcal{O}_2 \in Cl(W_K) \\ \mathcal{O}_2 \subseteq \hat{\mathcal{O}}}} \varphi(\mathcal{O}_1, \mathcal{O}_2), $$
where $\mathcal{O}_1 \in Cl(W_K)$ is the $W_K$-conjugacy class of $w$, and $\varphi$ is the pairing for $W_K$.
\end{theorem}

\begin{proof}
Consider the transition matrix expressing the evaluations of the orbital integrals $\Phi_{\mathcal{O}}$ on the elements $T_w$ in terms of the pairings $\varphi(\mathcal{O}_1, \mathcal{O}_2)$. If we order the spherical conjugacy classes of $W$ in descending order of their rank, Proposition \ref{prop:PhiValues} guarantees that this transition matrix is block upper-triangular. 

Crucially, on the diagonal blocks where $\operatorname{rank}(\mathcal{O}) = \operatorname{rank}(\hat{\mathcal{O}})$, the only non-zero contributions come from classes $\mathcal{O}_2$ where $\mathcal{O}_2 \subseteq \mathcal{O}$, and these coefficients $c_{\mathcal{O}_2}$ are strictly positive. The function $\varphi_{\hat{\mathcal{O}}}$ is simply the linear combination obtained by inverting this matrix to isolate the $\hat{\mathcal{O}}$ component.
\end{proof}

\subsection{Proof of Theorem~\ref{thm:psi_global}}
The final step in constructing the global trace maps $\Psi_{\mathcal{O}}$ is to lift the evaluations of the Kac-Moody orbital integrals from specific finite fields $\mathbb{F}_q$ to the generic ground ring $A = \mathbb{Q}[\mathbf{q}^{\pm 1/2}]$. To prevent ambiguity, we will strictly use $\mathbf{q}$ for the formal parameter in $A$, and $q$ for a prime power. For any $h \in H_A$, let $h|_q \in H_{\mathbb{C}}$ denote its specialization under the map $\mathbf{q}^{1/2} \mapsto \sqrt{q}$.

The main mechanism for the proof is establishing the following lemma, which leverages the fact that a polynomial identity is rigidly determined by its evaluations at infinitely many points.

\begin{lemma} \label{lem:psi_global}
Let $\mathcal{O} \in Cl(W)$ be a spherical conjugacy class and let $h \in H_A$. There exists a unique element $p_h \in A$ such that for all sufficiently large prime powers $q$, we have
$$ \varphi_{\mathcal{O}}(h|_{\q=q}) = p_h|_{\q=q}, $$
where $\varphi_{\mathcal{O}}: H_{\mathbb{C}} \to \mathbb{C}$ is the linear combination of Kac-Moody orbital integrals constructed in Theorem \ref{thm:phi_global}.
\end{lemma}

\begin{proof}
The uniqueness of $p_h$ is clear: if two Laurent polynomials in $\mathbf{q}^{1/2}$ agree when evaluated at infinitely many prime powers $q$, they must be identical in $A$.

To prove existence, we first define the target value of the pairing. Let $\mathcal{O}^{\prime}\in Cl(W)$ be any conjugacy class. If $\mathcal{O}^{\prime}$ is not spherical, we set $\Psi(\mathcal{O},\mathcal{O}^{\prime}) := 0$. If $\mathcal{O}^{\prime}$ is spherical, we choose a minimal length element $w\in (\mathcal{O}^{\prime})_{\min}$. Its support $J := \operatorname{supp}(w)$ is a spherical subset. Denoting by $\mathcal{O}_{1}$ the $W_J$-conjugacy class of $w$, we set
$$ \Psi(\mathcal{O},\mathcal{O}^{\prime}) := \sum_{\substack{\mathcal{O}_{2}\in Cl(W_{J}) \\ \mathcal{O}_{2}\subseteq\mathcal{O}}} \psi_{J}(\mathcal{O}_{1}\otimes\mathcal{O}_{2}) \in A, $$
where $\psi_{J}$ is the generic pairing for the parabolic subsystem $W_{J}$ from Definition \ref{def:psi_generic}. Note that $\Psi(\mathcal{O},\mathcal{O}^{\prime})$ is independent of the choice of $w$, since any valid isomorphism in $\operatorname{Iso}_{W}(J,J^{\prime})$ maps the generic pairing $\psi_{J}$ to $\psi_{J^{\prime}}$.

By Proposition \ref{prop:cocenterSpanning}, the canonical elements $T_{\mathcal{O}'}$ span the cocenter. Thus, we can find a finite collection of conjugacy classes $\mathcal{O}_{1},\dots,\mathcal{O}_{n}\in Cl(W)$ and coefficients $c_{1},\dots,c_{n}\in A$ such that $h$ can be written as
$$ h \equiv \sum_{i=1}^n c_{i}T_{\mathcal{O}_{i}} \pmod{[H_A, H_A]}. $$
We define our candidate polynomial as
$$ p_h := \sum_{i=1}^n c_{i} \Psi(\mathcal{O},\mathcal{O}_{i}) \in A. $$

We must verify that this $p_h$ correctly tracks the orbital integral evaluations. For any sufficiently large prime power $q$, the elements $h|_q$ and $\sum c_i|_q (T_{\mathcal{O}_i})|_q$ differ only by an element in the commutator subspace $[H_{\mathbb{C}}, H_{\mathbb{C}}]$. Because $\varphi_{\mathcal{O}}$ is a linear combination of orbital integrals, Proposition \ref{prop:orbital_integral_cocenter} guarantees that it vanishes completely on commutators. Thus,
$$ \varphi_{\mathcal{O}}(h|_{\q=q}) = \sum_{i=1}^n c_i|_{\q=q} \, \varphi_{\mathcal{O}}((T_{\mathcal{O}_i})|_{\q=q}). $$
By Theorem \ref{thm:phi_global} and Proposition \ref{prop:PhiValues}(1), the value of $\varphi_{\mathcal{O}}((T_{\mathcal{O}_i})|_{\q=q})$ exactly matches the specialization $\Psi(\mathcal{O}, \mathcal{O}_i)|_{\q=q}$. Therefore, $\varphi_{\mathcal{O}}(h|_{\q=q}) = p_h|_{\q=q}$, which proves the existence of the polynomial.
\end{proof}

With this lemma in hand, the proof of Theorem \ref{thm:psi_global} is straightforward.

\begin{proof}[Proof of Theorem \ref{thm:psi_global}]
Given a spherical conjugacy class $\mathcal{O} \in Cl(W)$, we define the map $\Psi_{\mathcal{O}}: H_A \to A$ by setting $\Psi_{\mathcal{O}}(h) = p_h$, where $p_h$ is the unique element in $A$ described in Lemma \ref{lem:psi_global}. The map is clearly $A$-linear. We must verify it satisfies the required properties.

First, we check that $\Psi_{\mathcal{O}}$ factors through the cocenter $\overline{H}_A$. If $h \in [H_A, H_A]$, then for all prime powers $q$, its specialization $h|_{\q=q}$ lies in $[H_{\mathbb{C}}, H_{\mathbb{C}}]$. Because orbital integrals vanish on commutators, we have $\varphi_{\mathcal{O}}(h|_{\q=q}) = 0$ for infinitely many $q$. The uniqueness of the polynomial in Lemma \ref{lem:psi_global} forces $\Psi_{\mathcal{O}}(h) = p_h = 0$.

Second, we verify the evaluation on basis elements. For any conjugacy class $\mathcal{O}' \in Cl(W)$, applying the definition of $\Psi_{\mathcal{O}}$ to the element $h = T_{\mathcal{O}'}$ trivially yields an expansion with a single term (where the coefficient is $1$). Thus, $\Psi_{\mathcal{O}}(T_{\mathcal{O}'})$ is exactly the polynomial $\Psi(\mathcal{O}, \mathcal{O}')$ we defined in the proof of Lemma \ref{lem:psi_global}. This perfectly matches the piecewise formula required by Theorem \ref{thm:psi_global}, concluding the proof.
\end{proof}

\section{Invertibility of Lusztig's pairing}\label{sec:perfect_pairing}

\subsection{Main result}

In this section, we study the matrix representing the generic pairing from Proposition \ref{prop:lusztigPairingConsequences}. We are mainly interested in the invertibility of this matrix, as the functions $\psi_{\mathcal{O}}$ can be globalized to arbitrary Kac-Moody groups (Theorem \ref{thm:psi_global}). 

Our main result for this section is the following:

\begin{theorem} \label{thm:perfectPairing}
Let $\mathcal{D}$ be an irreducible root datum of finite type, and let $M$ be the matrix representing the generic pairing $\psi$ from Proposition \ref{prop:lusztigPairingConsequences}. 
\begin{enumerate}
    \item If $\mathcal{D}$ is not of type $E_7$ or $E_8$, then the matrix $M$ is invertible.
    \item If $\mathcal{D}$ is of type $E_7$, then the kernel of $M$ is one-dimensional. Furthermore, the corresponding one-dimensional kernel $K_1 \subset \mathbb{Q}[Cl(W)]$ from Proposition \ref{prop:lusztigPairingConsequences} has trivial intersection with the subspace spanned by $Cl(W) \setminus \{w_0\}$, where $w_0 \in W$ is the longest element.
\end{enumerate}
\end{theorem}

We note the following immediate application of Theorem~\ref{thm:perfectPairing} and Proposition~\ref{prop:lusztigPairingConsequences}.

\begin{corollary}\label{cor:perfectPairingConsequences}
Let $(W,S)$ be a finite Weyl group with no irreducible component of type $E_8$,  and let $\CO_2\in \Cl(W)$ be an elliptic conjugacy class. Then there exists coefficients $c_{\CO_2,\CO'}\in\mathbb Q$ for all elliptic classes $\CO'\in\Cl(W)$ with the following property: If $\CO_1\in \Cl(W)$ is elliptic, and there does not exist an irreducible component $J\subseteq S$ of type $E_7$ such that the projection of $\CO_1$ to $W_J$ is equal to $\{w_{0,J}\}$, then
    \begin{align*}
        \sum_{\CO'\in \Cl(W)\text{ elliptic}} \psi_1(\CO_1\otimes \CO')c_{\CO_2,\CO'}=\delta_{\CO_1,\CO_2}.
    \end{align*}
\end{corollary}
\begin{proof}
    It suffices to prove this corollary separately for each irreducible component of $(W, S)$. For components not of type $E_7, E_8$, this follows from Proposition~\ref{prop:lusztigPairingConsequences} (1), together with the information that the matrix considered there is invertible by Theorem~\ref{thm:perfectPairing}, together with the information that it is block diagonal from Theorem~\ref{thm:lusztigPairingOverview}.

    Consider now the $E_7$ case. Let $\Cl^{\mathrm{el}}(W)$ be the set of elliptic classes and let $V=\mathbb Q[\Cl^{\mathrm{el}}(W)]$ be the free $\mathbb Q$-vector space over this set. Let $U$ be the $\mathbb Q$-subspace of $V^\ast$ spanned by the functions
    \begin{align*}
        \Bigl(\CO_1\mapsto \psi_1(\CO_1\otimes\CO_2)\Bigr)_{\CO_2\in \Cl(W)^{\mathrm{el}}}.
    \end{align*}
    Then we saw in Theorem~\ref{thm:perfectPairing} together with Theorem~\ref{thm:lusztigPairingOverview} that $K = \bigcap_{\lambda\in U}\ker \lambda$ is one-dimensional, with an explicit basis vector listed there. 
    Given any $\CO_2\in \Cl^{\mathrm{el}}(W)$, consider $\lambda_{\CO_2}\in V^\ast$ defined by $\lambda_{\CO_2}(\CO_1) = \delta_{\CO_1,\CO_2}$ for all $\CO_1\in \Cl^{\mathrm{el}}(W)$.
    
    Consider the class $\{w_0\}\in \Cl^{\mathrm{el}}(W)$. Then $\lambda_{\{w_0\}}\vert_K$ is non-vanishing. Thus, $V^\ast$ is a direct sum of $\mathbb Q\lambda_{\{w_0\}}$ and $U$. Thus, for every $\CO_2\in\Cl^{\mathrm{el}}(W)$, there exists some $q\in \mathbb Q$ with $\lambda_{\CO_2} - q\lambda_{\{w_0\}}\in U$. This finishes the proof.
\end{proof}

This theorem is easily reduced to the case where $(W, S)$ is irreducible. We therefore study the exceptional and classical types separately.
\subsection{Lusztig pairings in exceptional types}\label{subsec:exceptional}

In exceptional types, the matrix in Proposition~\ref{prop:lusztigPairingConsequences} can be computed using the \textsc{CHEVIE} computer algebra system \cite{Chevie}. This yields an invertible matrix in types $G_2, F_4$ and $E_6$.

For example, this is the matrix for $G_2$ (omitting zero values):
\begin{align*}
    \begin{pmatrix}
1&&&&\\&1&&&\\&&2/3&-1/3&1/3\\&&-1/3&2/3&1/3\\&&1/3&1/3&1/6
    \end{pmatrix}.
\end{align*}This is the matrix for $E_6$:
\begin{equation*}
\setcounter{MaxMatrixCols}{30}
\left(
    \begin{smallmatrix}
1&&&&&&&&&&&&&&&&&&&&&&&\\&1&&&&&&&&&&&&&&&&&&&&&&\\&&2/3&1/3&0&1/3&-1/3&&&&&&&&&&&&&&&&&\\&&1/3&1/6&-1/2&1/6&1/3&&&&&&&&&&&&&&&&&\\&&0&-1/2&1/2&1/2&0&&&&&&&&&&&&&&&&&\\&&1/3&1/6&1/2&1/6&1/3&&&&&&&&&&&&&&&&&\\&&-1/3&1/3&0&1/3&2/3&&&&&&&&&&&&&&&&&\\&&&&&&&1&&&&&&&&&&&&&&&&\\&&&&&&&&1&&&&&&&&&&&&&&&\\&&&&&&&&&1/2&-1/2&1/2&&&&&&&&&&&&\\&&&&&&&&&-1/2&1/2&1/2&&&&&&&&&&&&\\&&&&&&&&&1/2&1/2&1/2&&&&&&&&&&&&\\&&&&&&&&&&&&1/2&-1/2&1/2&&&&&&&&&\\&&&&&&&&&&&&-1/2&1/2&1/2&&&&&&&&&\\&&&&&&&&&&&&1/2&1/2&1/2&&&&&&&&&\\&&&&&&&&&&&&&&&1&&&&&&&&\\&&&&&&&&&&&&&&&&1&&&&&&&\\&&&&&&&&&&&&&&&&&1&&&&&&\\&&&&&&&&&&&&&&&&&&1&&&&&\\&&&&&&&&&&&&&&&&&&&1&&&&\\&&&&&&&&&&&&&&&&&&&&1&&&\\&&&&&&&&&&&&&&&&&&&&&1&&\\&&&&&&&&&&&&&&&&&&&&&&1&\\&&&&&&&&&&&&&&&&&&&&&&&1\\&&&&&&&&&&&&&&&&&&&&&&&&1\\
\end{smallmatrix}
    \right).
\end{equation*}
For $F_4$ the matrix is block diagonal again (according to the families), and we get the following distinct diagonal blocks:
\begin{align*}
    (1),\left(\begin{smallmatrix}
1/2&-1/2&1/2\\-1/2&1/2&1/2\\1/2&1/2&1/2\end{smallmatrix}\right),
\left(\begin{smallmatrix}
2/3&-1/3&1/3\\-1/3&1/6&1/12\\1/3&1/12&1/24\end{smallmatrix}\right),
    \left(\begin{smallmatrix}
3/8&-1/8&-1/8&3/8&-1/8&1/4&1/8&-1/4&1/4&-1/4\\-1/8&3/8&-1/8&-1/8&3/8&1/4&1/8&1/4&-1/4&-1/4\\-1/8&-1/8&3/8&-1/8&-1/8&1/4&1/8&-1/4&-1/4&1/4\\3/8&-1/8&-1/8&3/8&-1/8&1/4&1/8&1/4&-1/4&1/4\\-1/8&3/8&-1/8&-1/8&3/8&1/4&1/8&-1/4&1/4&1/4\\1/4&1/4&1/4&1/4&1/4&1/6&1/12&0&0&0\\1/8&1/8&1/8&1/8&1/8&1/12&1/24&1/4&1/4&1/4\\-1/4&1/4&-1/4&1/4&-1/4&0&1/4&1/2&0&0\\1/4&-1/4&-1/4&-1/4&1/4&0&1/4&0&1/2&0\\-1/4&-1/4&1/4&1/4&1/4&0&1/4&0&0&1/2\end{smallmatrix}\right).
\end{align*}
Up to this point, all the matrices we have seen in this section are invertible. Thus, Theorem~\ref{thm:perfectPairing} is true for exceptional types.

For $E_7$, we have the following distinct diagonal blocks appearing:
\begin{align*}(1),
    \begin{pmatrix}       
1/2&-1/2&1/2\\-1/2&1/2&1/2\\1/2&1/2&1/2
    \end{pmatrix},
    \begin{pmatrix}
1/6&1/3&1/3&-1/2&1/6\\1/3&2/3&-1/3&0&1/3\\1/3&-1/3&2/3&0&1/3\\-1/2&0&0&1/2&1/2\\1/6&1/3&1/3&1/2&1/6
\end{pmatrix},\begin{pmatrix}
1/2&1/2\\1/2&1/2
\end{pmatrix}.
\end{align*}

The last block is the only singular one, and it appears only once. Thus, the kernel of the original matrix is one-dimensional. This is explained as follows: The Weyl group $W$ has two $512$-dimensional irreducible representations, which is the largest dimension for all irreducible representations. Denote them by $E_1, E_2\in \Irr(\mathbb Q[W])$.

We lift $E_1, E_2$ to representations of $G(\mathbb F_q)$ as usual, and denote their characters by $\rho_1$ and $\rho_2$. Then for every Deligne-Lusztig character $R^\theta_T$ of $G$, we get
    \begin{align*}
        \langle \rho_1, R^\theta_T\rangle = \langle \rho_2,  R^\theta_T\rangle.
    \end{align*}
    Therefore, regular semisimple classes cannot distinguish these two representations.
    The difference between $\rho_1$ and $\rho_2$, as functions on $G(\mathbb F_q)$, is given by a linear combination of two characteristic functions of cuspidal character sheaves, see \cite[Corollary~4.14]{Lu-CS3} and \cite[Proof of Proposition 20.3 (c)]{Lu-CS4}. Each cuspidal character sheaf's support is a single conjugacy class, containing an element of the form $su$ such that $s\in G(\mathbb F_q)$ is semisimple, $C_G(s)^{\circ}$ has Cartan type $A_1 \times A_3\times A_3$ and $u\in C_G(s)^{\circ}$ is regular unipotent in this reductive group. For example if $\sqrt{-1} \in \mathbb F_q$ and $G$ is adjoint, then we can choose $s\in T(\mathbb F_q)$ such that $\alpha_i(s)=1$ for $i=1,2,3,5,6,7$ and $\alpha_4(s)=\sqrt{-1}$. Then $C_G(s)^{\circ}$ contains the root subgroups of the simple roots $\alpha_1,\alpha_2, \alpha_3,\alpha_5, \alpha_6, \alpha_7$ as well as the highest root $\theta$. The details are very nicely elaborated in \cite[Section~6]{Ge24}.

    We are in case (2) of Proposition~\ref{prop:lusztigPairingConsequences} . 
    Using \textsc{Chevie}, we compute the map corresponding to the specialization of the kernel generator at $q=1$:
$$ \mathbb{Q}[Cl(W)] \longrightarrow \mathbb{Q}, \quad \mathcal{O} \mapsto \operatorname{tr}(w_{\mathcal{O}}\mid E_1) - \operatorname{tr}(w_{\mathcal{O}}\mid E_2). $$
The kernel $K_1$ is one-dimensional, spanned by the explicit linear combination of classes obtained from this map. A direct calculation confirms that all the conjugacy classes that occur here with non-zero coefficients are elliptic and consist of elements of odd lengths. In particular, we obtain a strictly non-zero coefficient for the longest element $w_0\in W$. This verifies the intersection property required by Theorem \ref{thm:perfectPairing}(2).

    For type $E_8$, the kernel of the matrix in Proposition~\ref{prop:lusztigPairingConsequences}  is two dimensional. The issue has a similar explanation as in type $E_7$, stemming from four $4096$-dimensional irreducible representations of the Hecke algebra. For the purpose of this paper, we do not need the Lusztig's pairing for type $E_8$.

\subsection{Lusztig pairings in classical types via Lusztig symbols}\label{sec:symbols}

We now study the Fourier transform matrix for groups of classical Cartan types. As discussed in Section \ref{sec:6.3}, Lusztig organizes the irreducible representations $E \in \operatorname{Irr}(\mathbb{Q}[W])$ into families, and the matrix $M = (\{x_E, x_{E'}\})$ is block-diagonal with respect to this partition \cite[\S 4.21]{Lu-Characters}.

\textbf{Type $A$}

This case is straightforward. The families in type $A$ consist only of singletons. Consequently, the matrix $M$ is the identity matrix, which is trivially invertible.

\textbf{Type $B$}

The unipotent representations were studied in \cite[\S 4.5]{Lu-Characters}. The representations in a family are indexed by {\it symbols}
\begin{align*}
    \Lambda_M = \begin{pmatrix}Z_2\sqcup(Z_1-M)\\Z_2\sqcup M\end{pmatrix},
\end{align*}
where $Z_1\sqcup Z_2$ is equal to a certain fixed set (which determines the family), and $M$ runs over all $d$-element subsets of $Z_1$, where $\# Z_1 = 2d+1$. 

We denote by $\binom{Z_1}d\subseteq 2^{Z_1}$ the set of $d$-element subsets of $Z_1$. The pairing of two symbols is $$\{\Lambda_{M_1}, \Lambda_{M_2}\}=2^{-d}(-1)^{\# (M_1^{\#}\cap M_2^{\#})}.$$ Here $M_i^{\#} = (M_i\cup M_0) - (M_i\cap M_0)$ for some fixed set $M_0$ depending on the family. Thus 
\begin{align*} 
\#(M_1^{\#} \cap M_2^{\#}) &=\#\bigl((M_1 \cap M_2) \cup M_0-(M_1 \cap M_0)-(M_2 \cap M_0)\bigr) \\ &=\#(M_1 \cap M_2)+\#(M_0)-\#(M_1 \cap M_2 \cap M_0) \\ & \quad -\#(M_1 \cap M_0)-\#(M_2 \cap M_0)+\#(M_1 \cap M_2 \cap M_0) \\ &= \#(M_1 \cap M_2)+\#(M_0)-\#(M_1 \cap M_0)-\#(M_2 \cap M_0).
\end{align*}

Hence the matrix $(\{\Lambda_{M_1}, \Lambda_{M_2}\})_{M_1, M_2\in \binom{Z_1}d}$ can be transformed by left multiplication of the diagonal matrix $\diag(2^d (-1)^{\#(M_1 \cap M_0)+\#(M_0)})_{M_1 \in \binom{\{1,\dotsc,2d+1\}}{d}}$ and by right multiplication of the diagonal matrix $\diag((-1)^{\#(M_2 \cap M_0)})_{M_2 \in \binom{\{1,\dotsc,2d+1\}}{d}}$ into the matrix
    \begin{align*}
        \mathcal M_{d}:=\Bigl( (-1)^{\#(A\cap B)}\Bigr)_{A,B\in \binom{\{1,\dotsc,2d+1\}}{d}}
    \end{align*}
We will study this matrix in the next subsection.

\textbf{Type $C$}

The representation theory of the Hecke algebra for type $C_n$ is identical to that of type $B_n$. The analysis of the unipotent families is completely analogous and ultimately reduces to the invertibility of the exact same matrix $\mathcal{M}_{d}$.

\textbf{Type $D$}

We follow \cite[\S 4.6]{Lu-Characters}. Families with more than one member are again parametrized by symbols
\begin{align*}
    \Lambda_M = \begin{pmatrix}Z_2\sqcup(Z_1-M)\\Z_2\sqcup M\end{pmatrix},
\end{align*}
similar to the ones above, the main difference being that the symbol is still considered the same if the two rows are swapped. Here, $Z_1\sqcup Z_2$ is equal to a certain fixed set determined by the family, and $M\subseteq Z_1$ is a subset of cardinality $\# M = d$, where $\# Z_1 = 2d$. The equivalence relation is $M \sim (Z_1-M)$. The pairing is
\begin{align*}
    \{\Lambda_{M_1}, \Lambda_{M_2}\} = 2^{-d} (-1)^{\# (M_1^{\#} \cap M_2^{\#})}.
\end{align*}

The matrix $(\{\Lambda_{M_1}, \Lambda_{M_2}\})_{M_1, M_2\in \binom{Z_1}d/\sim}$ can be transformed by left and right multiplication of diagonal matrices into the matrix
    \begin{align*}
        \tilde{\mathcal M}_{d}:=\Bigl( (-1)^{\#(A\cap B)}\Bigr)_{A,B\in \binom{\{1,\dotsc,2d\}}{d}/\sim}
    \end{align*}

To connect this to the odd-cardinality case, choose a fixed element $z_1\in Z_1$. Under the equivalence relation on $\binom{Z_1}{d}$, every equivalence class contains exactly one representative $M$ such that $z_1\notin M$. This establishes a canonical bijection between $\binom{Z_1}{d}/\sim$ and the subsets $\binom{Z_1 \setminus \{z_1\}}{d}$, which is exactly $\binom{\{1,\dotsc,2d-1\}}{d}$. Under this bijection, the matrix $\tilde{\mathcal{M}}_{d}$ is identical to $\mathcal{M}_{d-1}$. Therefore, the invertibility of the pairing in type $D$ also completely reduces to the invertibility of $\mathcal{M}_d$.

\subsection{Characteristic polynomial of a $\pm 1$-matrix}
    Let $0\leq d\leq n$ be two integers. Set $I = \{1,2,\dotsc,n\}$ and $S = \binom Id = \{A\subseteq I\mid \# A = d\}$. Let $V$ be the free $\mathbb C$-vector space with basis $S$. Define a linear map $f : V\to V$ by
    \begin{align*}
        f(A) = \sum_{B\in S} (-1)^{\#(A\cap B)} \text{ for } A \in V.
    \end{align*}

    The matrix of $f$ with respect to the basis $S$ is
\[
\mathcal{M}_{n,d}:=\bigl((-1)^{\#(A\cap B)}\bigr)_{A,B\in S}.
\]

The main result of this section is a computation of the characteristic polynomial of $f$. As a special case, we shall prove that $f$ (and hence $\mathcal{M}_{n,d}$) is \textbf{invertible} when $n=2d+1$.
This is the matrix that appears for types $B$ and $C$; for type $D$ the relevant matrix
reduces to an odd‑case matrix of smaller size (see \S\ref{sec:symbols}),
so invertibility in the odd case suffices for all classical types.

For a subset $Z\subseteq I$, we define
\[
v_Z:=\sum_{\substack{A\in S\\ Z\subseteq A}}A\in V.
\]
For $-1\leq k\leq d$, let $V_k\subseteq V$ be the subspace spanned by all $v_Z$ with $\#Z=k$. Then
\[
0=V_{-1} \subseteq V_0\subseteq V_1\subseteq\cdots\subseteq V_d=V.
\]

\begin{proposition}\label{prop:charpoly}
    The characteristic polynomial of $f$ is given by
    \begin{align*}
        \chi_f = \prod_{k=0}^d (x-c_k)^{\binom {n} k - \binom n{k-1}},
    \end{align*}
    where the eigenvalues are $c_k=\sum_{\ell_1,\ell_2\geq 0}(-1)^{k+\ell_2} \binom k{\ell_1}\binom{d-\ell_1}{\ell_2}\binom{n-d-k+\ell_1}{d-k-\ell_2}$.
\end{proposition}

\begin{proof}
We have $\dim V_k / V_{k-1} = \binom nk - \binom n{k-1}$.
Our goal is to show that $f$ preserves each $V_k$ and its action on $V_k / V_{k-1}$ equals multiplication by $c_k$. 

Let us now evaluate the function $f$ on elements of $V_k$ for $k=-1,\dotsc,d$. Let $Z\in \binom Ik$.
For subsets $Z'\subseteq Z\subseteq I$, we define
    \begin{align*}
        v_{Z',Z} := \sum_{\substack{A\in S\\A\cap Z = Z'}} A \in V.
    \end{align*}
    Then it is an easy exercise in M\"obius inversion to show that
    \begin{align}
        v_{Z',Z} = \sum_{Z'\subseteq Y\subseteq Z} (-1)^{\#(Y\setminus Z')} v_Y\in V_k.\label{eq:vZprimeZ}
    \end{align}
    Alternatively, it is elementary to check \eqref{eq:vZprimeZ} by hand: The coefficient of $A\in S$ in the right-hand side of \eqref{eq:vZprimeZ} is
    \begin{align*}
        \sum_{Z'\subseteq Y\subseteq Z\cap A} (-1)^{\#(Y\setminus Z')}.
    \end{align*}This is easily seen to be $1$ if $Z\cap A=Z'$ and $0$ if $Z'\not\subseteq A$. If $Z'\subsetneq Z\cap A$, we get an alternating sum whose value is zero.
    
 Thus, we may compute
    \begin{align*}
        f(v_Z) &= \sum_{Z\subseteq A\in S} f(A) = \sum_{\substack{Z\subseteq A\in S\\ B\in S}}(-1)^{\#(A\cap B)}B
        \\&=\sum_{B\in S}(-1)^{\# (B\cap Z)}B\sum_{Z\subseteq A\in S}(-1)^{\#(A\cap B\setminus Z)} 
        \\&=\sum_{b\in S}(-1)^{\#(B\cap Z)}  B\sum_{A'\in \binom{I\setminus Z}{d-k}} (-1)^{\#(A'\cap B\setminus Z)}
        \\&= \sum_{B\in S} (-1)^{\#(B\cap Z)}B\sum_{\substack{A_1\subseteq B\setminus Z\\ A_2\subseteq I\setminus (B\cup Z)\\ \# A_1 + \# A_2 = d-k}} (-1)^{\#A_1}
        \\&=\sum_{B\in S} B\sum_{\ell\geq 0} (-1)^{\#(B\cap Z)+\ell} \binom{\#(B\setminus Z)}{\ell} \binom{n-\# (B\cup Z)}{d-k-\ell}.
    \end{align*}
    Write now $\# (B\setminus Z) = d - \# (B\cap Z)$ and $\#(B\cup Z) = k+d-\# (B\cap Z)$. Therefore, the coefficient of $B$ in the above sum only depends on the cardinality $B\cap Z =: Z'$. We therefore get
    \begin{align*}
        f(v_Z) = \sum_{Z'\subseteq Z} v_{Z',Z}\sum_{\ell\geq 0} (-1)^{\#Z'+\ell} \binom{d-\# Z'}{\ell} \binom{n-d-k+\# Z'}{d-k-\ell}.
    \end{align*}

    By \eqref{eq:vZprimeZ}, it follows that $f(V_Z)\in V_k$. By linearity, we get $f(V_k)\subseteq V_k$ as claimed.

    Finally, we study the value of $f(v_Z)$ modulo $V_{k-1}$ (if $k\geq 0$). Using \eqref{eq:vZprimeZ}, we get
    $v_{Z',Z} \in (-1)^{k-\#Z'} v_Z + V_{k-1}$. Using the above formula for $f(v_Z)$, we get
    \begin{align*}
        f(v_Z) \in v_{Z}\sum_{Z'\subseteq Z} \sum_{\ell\geq 0} (-1)^{k+\ell} \binom{d-\# Z'}{\ell} \binom{n-d-k+\# Z'}{d-k-\ell} + V_{k-1}.
    \end{align*}
    The right-hand side can now be more concisely written as $c_k v_Z+ V_{k-1}$. We are done.
\end{proof}

\begin{lemma}\label{lem:coeff}
The number $c_k$ is equal to the coefficient of $x^{d-k}$ in
    \begin{align*}
        (-1)^{d+k} 2^k (1-x)^{d-k} (1+x)^{n-d-k}.
    \end{align*}
\end{lemma}
\begin{proof}
For $\ell_1$, set $P_{d,k,\ell_1}(x) := (1-x)^{d-\ell_1}(1+x)^{n-d-k+\ell_1}\in \mathbb C[x]$. By the binomial theorem, we get
    \begin{align*}
        P_{d,k,\ell_1}(x) = \sum_{m_1, m_2\geq 0}x^{m_1} (-1)^{m_1}\binom{d-\ell_1}{m_1} x^{m_2} \binom{n-d-k+\ell_1}{m_2}.
    \end{align*}
    Therefore, the coefficient of $x^{d-k}$ in $P_{d,k,\ell_1}(x)$ is given by
    \begin{align*}
        \text{coeff.\ of }x^{d-k}\text{ in }P_{d,k,\ell_1}(x)
        =\sum_{\ell_2}(-1)^{\ell_2}\binom{d-\ell_1}{\ell_2} \binom{n-d-k+\ell_1}{d-k-\ell_2}.
    \end{align*}
    Hence we get
    \begin{align*}
        (-1)^{d+k}{c_k} &= \text{ coeff.\ of $x^{d-k}$ in }\sum_{\ell_1\geq 0}\binom k{\ell_1} (1-x)^{d-\ell_1} (1+x)^{n-d-k+\ell_1}
        \\&=\text{ coeff.\ of $x^{d-k}$ in }(1-x)^{d-k}(1+x)^{n-d-k} \sum_{\ell_1\geq 0} \binom k{\ell_1} (1-x)^{k-\ell_1} (1+x)^{\ell_1}
        \\&=\text{ coeff.\ of $x^{d-k}$ in }(1-x)^{d-k}(1+x)^{n-d-k} [(1-x) + (1+x)]^k
        \\&=\text{ coeff.\ of $x^{d-k}$ in }(1-x)^{d-k} (1+x)^{n-d-k} 2^k.
    \end{align*}
    This finishes the proof.
\end{proof}

\begin{corollary}
    Let $n=2 d+1$. Then $f$ is an isomorphism. 
\end{corollary}

\begin{proof}
For $n = 2d+1$, the polynomial in Lemma \ref{lem:coeff} becomes
$$ (1-x)^{d-k} (1+x)^{d-k+1} = (1+x) (1-x^2)^{d-k}. $$
The polynomial $(1-x^2)^{d-k}$ is even, meaning it has non-zero coefficients only for the even exponents $x^0, x^2,\dotsc, x^{2(d-k)}$. When multiplied by $(1+x)$, we obtain a polynomial where every coefficient from $x^0$ up to $x^{2(d-k)+1}$ is non-zero (with alternating signs). 

Note that for $0 \le k \le d$, we have $0 \le d-k \le 2(d-k)+1$). By Lemma \ref{lem:coeff}, this implies $c_k \neq 0$ for all $0 \le k \le d$. Because all of its eigenvalues $c_k$ are non-zero, the map $f$ (and consequently the matrix $\mathcal{M}_d$) is invertible.
\end{proof}

As a consequence of this combinatorial inversion and the checks for exceptional types in Section \ref{subsec:exceptional}, the proof of Theorem \ref{thm:perfectPairing} is complete.

\section{A special construction for $E_7$}\label{sec:E7}

\subsection{The $E_7$ anomaly and valid pairs}
The final obstacle to completing the proof of our main theorem comes from conjugacy classes $\mathcal{O}\in Cl(W)$ whose support contains at least one irreducible component of type $E_7$. Recall that each such conjugacy class $\mathcal{O}$ defines a canonical element $T_{\mathcal{O}}\in \overline{H}$. However, as established in Section \ref{thm:perfectPairing}, the orbital integral method for spherical elliptic pairs is insufficient to establish the linear independence of these elements. Specifically, the Hecke algebra of type $E_7$ possesses two irreducible $512$-dimensional representations whose characters are indistinguishable by orbital integrals of regular semisimple elements. 

Fortunately, these two degenerate characters evaluate differently on $T_{w_0}$, where $w_0$ is the longest element of the Weyl group of $E_7$. In this section, we construct a new class of functions on the cocenter $\overline{H}$ tailored specifically to exploit this difference. This provides the final ingredient needed to complement what the orbital integrals can achieve.

In order to exploit this difference at the longest element, we will use a modified, truncated form of parabolic induction. Throughout this section, we let $(W, S)$ be the Weyl group of an arbitrary Kac-Moody group and let $I_1\subseteq S$ be a subset of type $E_7$. We define its centralizer and the combined parabolic set as:
\begin{align*}
    I_2 &:= \{s\in S\mid ss' = s' s \text{ for all } s'\in I_1\},\\
    I &:= I_1\cup I_2.
\end{align*}
For any spherical subset $J\subseteq S$, we let $w_{0,J}\in W_J$ denote the longest element of $W_J$. If $J$ is irreducible, we denote the highest root of the positive root system $\Phi_J^+$ of $(W_J, J)$ by $\theta_J$.

If we were to apply the standard parabolic induction from Section \ref{sec:par-ind} to $W_I$, elements from an ambient $E_8$ parabolic subgroup could cause an infinite sum. Therefore, we restrict our induction to a highly specific subset of elements.

\begin{definition} \label{def:valid_pair}
    With notation as above, we define a \emph{valid pair} to be a tuple $(v_1, v_2)\in W^I\times W_{I_2}$ satisfying the following property: For every cardinality eight set $J$ with $I_1\subsetneq J\subseteq S$ of type $E_8$ such that $J\subseteq C_W(\operatorname{supp}(v_2))$ (i.e., every simple reflection in $J$ commutes with every simple reflection in $\operatorname{supp}(v_2)$), we require $v_1\theta_J\in \Phi^+$.
\end{definition}

By strictly bounding our summation to these valid pairs, we can define a truncated parabolic induction method that targets the $E_7$ block. Now we state the main result of this section. 

\begin{theorem}\label{thm:E7_lift}
    We continue the notation as above. Let $f: H_{I_2} \to \mathbb{Z}[\q^{\pm 1}]$ be a $\mathbb Z[\q^{\pm 1}]$-linear function which vanishes on commutators $[H_{I_2}, H_{I_2}]$. Then the function
    \begin{align*}
        F : H\to \mathbb Z[\q^{\pm 1}],\qquad h\mapsto \sum_{(v_1, v_2)\text{ valid}} \q^{-\ell(v_1) - \ell(v_2)} f(T_{v_2}) \tau(T_{(v_1 v_2 w_{0,I_1})^{-1}} h T_{v_1})
    \end{align*}
    is well-defined, i.e.\ the sum has only finitely many non-zero terms for each $h\in H$.
    It satisfies the following properties:
    \begin{enumerate}
    \item We have $F(h)=0$ for all $h\in [H,H]$. Thus, we get an induced map $F : \overline H\to \BZ[\q^{\pm 1}]$.
    \item For $w\in W$, we get
    \begin{align*}
        F(T_w) \equiv f(T_{w_2})\pmod{\q-1}
    \end{align*}
    for any element $w_2\in W_{I_2}$ such that $w_{0,I_1} w_2$ is $W$-conjugate to $w$. If no such element $w_2$ exists, then $F(T_w)\equiv 0\pmod{\q-1}$.
    \end{enumerate}
\end{theorem}

For the proof, we need some preparation.

\subsection{Finiteness conditions}

Our goal in this subsection is to establish that the sum defining the map $F$ in Theorem \ref{thm:E7_lift} is well-defined (i.e., finite). The main result is the following finiteness bound on valid pairs.

\begin{proposition}\label{prop:E7_lift_finiteness}
    Given $w\in W$, here are only finitely many valid pairs $(v_1, v_2)$ which satisfy $\tau(T_{v_1 v_2 w_{0,I_1}} T_w T_{v_1})\neq 0$.
\end{proposition}

To prove this proposition, we must establish a sequence of geometric bounds on Coxeter groups. We begin with a general combinatorial lemma regarding integer sequences.

\begin{lemma}\label{lem:bound}
    Equip the $n$-fold product of the non-negative integers $\mathbb{Z}_{\geq 0}^n$ with a partial order by declaring $(v_1,\dotsc,v_n)\leq (u_1,\dotsc,u_n)$ iff $v_i\leq u_i$ for all $i$.
    \begin{enumerate}
        \item Each subset $A\subseteq \mathbb{Z}_{\geq 0}^n$ only has finitely many minima.
        \item If $A\subseteq \mathbb{Z}_{\geq 0}^n$ is an infinite subset, then there exists an infinite increasing chain $v^{(1)}<v^{(2)}<\cdots$ of elements in $A$.
    \end{enumerate}
\end{lemma}

\begin{proof}
    \begin{enumerate}
        \item We proceed by induction on $n$, the cases $n=0,1$ being clear. Assume $n\geq 2$ and that the claim holds for $n-1$. For $i=1,\dotsc,n$, we let $\pi_i : \mathbb{Z}_{\geq 0}^n\to \mathbb{Z}_{\geq 0}$ be the projection to the $i$-th factor. We may assume $A\neq\emptyset$; fix an element $a^\ast\in A$. Pick a minimal element $m$ of $A$. Since $a^\ast \in A$, we cannot have $a^\ast < m$. Thus, every minimal element $m$ of $A$ satisfies $\pi_i(m)\leq \pi_i(a^\ast)$ for at least one $i\in \{1,\dotsc,n\}$. Moreover, if we set $p := \pi_i(m)\in\mathbb{Z}_{\geq 0}$, then $m$ is minimal in the fiber $\pi_i^{-1}(p)\cap A$.

        Consider the set
        \begin{align*}
            A_{i,p} := \{(a_1,\dotsc,a_{i-1},a_{i+1},\dotsc,a_n)\in\mathbb{Z}_{\geq 0}^{n-1}\mid (a_1,\dotsc,a_{i-1}, p, a_{i+1},\dotsc,a_n)\in A\}.
        \end{align*}
        By the inductive hypothesis, $A_{i,p}$ only has finitely many minima. By comparing elements between $\pi_i^{-1}(p)\cap A$ and $A_{i,p}$ (which correspond under inserting or deleting the $i$-th coordinate), we conclude that each $\pi_i^{-1}(p)\cap A$ only has finitely many minima. As established above, the set of minimal elements of $A$ is contained in the union of all minima of $\pi_i^{-1}(p)\cap A$ over the finite set of pairs $(i,p)$ with $i \in \{1,\dotsc,n\}$ and $p \in \{0, 1, \dotsc, \pi_i(a^\ast)\}$. Since the union of finitely many finite sets is finite, we are done.

        \item This is a standard application of K\"onig's lemma. For each $a\in A$, let $A_{\geq a} := \{b\in A\mid a\leq b\}$. Then $A$ is the union of the sets $A_{\geq a}$ where $a$ ranges over the (finitely many) minimal elements of $A$. Since $A$ is infinite, there must exist a minimal element $v^{(1)}$ of $A$ such that $A_{\geq v^{(1)}}$ is infinite.
        
        Set $A^{(1)} := A_{\geq v^{(1)}}\setminus \{v^{(1)}\}$. Again, this infinite set has finitely many minima (by part (1)), so we can find a minimal element $v^{(2)}\in A^{(1)}$ such that $A_{\geq v^{(2)}}$ is infinite. Setting $A^{(2)} := A_{\geq v^{(2)}}\setminus \{v^{(2)}\}$ and repeating this procedure indefinitely yields the desired infinite increasing chain. \qedhere
    \end{enumerate}
\end{proof}

Recall that the \emph{height} function $\mathrm{ht}: \mathbb{Z}\Phi\to \mathbb{Z}$ is the unique $\mathbb{Z}$-linear map sending each simple root to $1$. Thus, positive roots have positive height and negative roots have negative height.

\begin{lemma}\label{lem:const}
    Let $C\in\mathbb{Z}$ be any constant. Then the set
    \begin{align*}
        \{w\in W\mid \forall s\in S:~\mathrm{ht}(w\alpha_s)\geq C\}
    \end{align*}
    is finite.
\end{lemma}

\begin{proof}
    Let $Z\subseteq W$ be the set under consideration and $\{\alpha_1,\dotsc,\alpha_n\} = \{\alpha_s\mid s\in S\}$ be an enumeration of the simple roots. Define a map
    \begin{align*}
        \vartheta : Z\to \mathbb{Z}^{n^2},\qquad w\mapsto (\text{the $\alpha_i$-coefficient of } w\alpha_j)_{i,j\in\{1,\dotsc,n\}}.
    \end{align*}
    Because the simple roots form a basis for the geometric representation, the map $\vartheta$ is injective. Denote its image by $A$. 
    
    We claim that the coordinates of all elements in $A$ are uniformly bounded from below. Indeed, if $w\alpha_j$ is a positive root, all of its coefficients are non-negative. If $w\alpha_j$ is a negative root, all of its coefficients are non-positive. Because the height of a root is the sum of its coefficients, the condition $\mathrm{ht}(w\alpha_j) \geq C$ forces every individual coefficient of a negative root to be bounded below by $C$. Thus, all coordinates of elements in $A$ are bounded below by $\min(0, C)$. 
    
    By shifting $A$ component-wise by $-\min(0,C)$, we obtain a subset of $\mathbb{Z}_{\geq 0}^{n^2}$. If $Z$ were infinite, then $A$ would be infinite, and by Lemma \ref{lem:bound}, we would find an infinite increasing chain $\vartheta(w_1) < \vartheta(w_2) < \cdots$ in $A$.

    Consider two elements $w_1, w_2\in W$ with $\vartheta(w_1) < \vartheta(w_2)$. By definition, this means that for every simple root $\alpha_j$, the difference $w_2\alpha_j - w_1\alpha_j$ is a non-negative linear combination of simple roots. Because any positive root $\alpha \in \Phi^+$ is a non-negative linear combination of simple roots, the difference $w_2\alpha - w_1\alpha$ must also be a non-negative linear combination of simple roots. 
    
    In particular, if $w_2\alpha$ is a negative root, then $w_1\alpha$ must also be a negative root. This implies an inclusion of inversion sets:
    $$ \{\alpha \in \Phi^+ \mid w_2\alpha \in \Phi^-\} \subseteq \{\alpha \in \Phi^+ \mid w_1\alpha \in \Phi^-\}. $$
    Since the length of an element is exactly the cardinality of this inversion set, it follows that $l(w_2) \leq l(w_1)$.

    This argument shows that if $Z$ were infinite, we would find an infinite sequence of elements in $Z$ with monotonically decreasing (and thus eventually bounded) lengths. Since a Coxeter group contains only finitely many elements of any bounded length, this is clearly impossible.
\end{proof}

We now apply these general bounds to the specific centralizers in type $E_7$.

\begin{lemma}\label{lem:validPairCentralizerFiniteness}
    Let $v_1\in W^I, v_2\in W_{I_2}$. Then there exist only finitely many valid pairs $(v_1', v_2)$ with $v_1' \in v_1 C_W(w_{0,I_1}v_2)$.
\end{lemma}

\begin{proof}
Recall that the maximal height of positive roots in $E_7$ is $63$. Let $C = \min\{\operatorname{ht}(\alpha)\mid \alpha\in \Phi^-\text{ and } v_1\alpha\in \Phi^+\}\cup\{-64\}$. 
    Let $z\in C_W(w_{0,I_1}v_2)$ such that $(v_1 z, v_2)$ is a valid pair. By Lemma \ref{lem:const}, it suffices to show that $\operatorname{ht}(z\alpha)\geq C$ for all simple roots $\alpha$.

    Aiming for a contradiction, assume this is not the case. Then we can find a simple root $\alpha$ with $\operatorname{ht}(z\alpha)<C$. By the definition of $C$, this forces $v_1z\alpha$ to be a negative root. By assumption, $v_1 z\in W^I$, which means $v_1 z$ sends all simple roots in $I$ to positive roots. Thus, we must have $\alpha\notin I$.

    \textbf{Step 1: We prove that $z\in W^{I_1}$ centralizes every element in $I_1$.} Let us decompose $z = z_1 z_2$ with $z_1\in W^{I}$ and $z_2\in W_{I}$. Since $W_{I}$ is contained in the centralizer of $w_{0,I_1}$, we have $z_1\in C_W(w_{0,I_1} z_2v_2z_2^{-1})$ and $z_2 v_2 z_2^{-1}\in W_{I_2}$. Apply now Lemma~\ref{lem:parabolicConjugation} (1) to $z_1$ and any maximal Bruhat order chain
    \begin{align*}
        1 = w_0\lessdot w_1\lessdot\cdots \lessdot w_\ell = w_{0,I_1} z_2v_2z_2^{-1}.
    \end{align*}
    Then it follows from said lemma that $z_1 w_1 z_1^{-1}\leq w_{0,I_1} z_2v_2z_2^{-1}$ has length $1$, so $z_1 w_1 z_1^{-1}\in \supp(w_{0,I_1} z_2v_2z_2^{-1})$.
    This shows that $z_1 \big(I_1 \cup \supp(z_2v_2z_2^{-1})\big) (z_1)^{-1} = I_1 \cup \supp(z_2v_2z_2^{-1})$. Because $z_2v_2z_2^{-1} \in W_{I_2}$, the simple reflections in $I_1$ commute with those in $\supp(z_2v_2z_2^{-1})$. Thus, $I_1$ forms an isolated connected component of type $E_7$ within $I_1 \cup \supp(z_2v_2z_2^{-1})$. By Section \ref{sec:LS} (a), conjugation by $z_1$ must map this $E_7$ component to itself and its restriction to this $E_7$ component is the trivial action. We conclude that $z_1 s z_1^{-1} = s$ for all $s\in I_1$. In particular, $z_1\Phi_{I_1}^+ = \Phi_{I_1}^+$. Now, the fact that $v_1\in W^{I_1}$ guarantees $v_1 z_1 \in W^{I_1}$. Since $(v_1 z, v_2)$ is a valid pair, we know $v_1 z \in W^I \subseteq W^{I_1}$. The equation $v_1 z = (v_1 z_1) z_2$ then implies $z_2 \in W_{I_2}$. Thus, $z \in W^{I_1}$, and $z$ centralizes every simple reflection in $I_1$.

    \textbf{Step 2: Construction of $J$.} The fact that $\alpha\notin I$ implies that $\langle \alpha^\vee,\beta\rangle\neq 0$ for at least one simple root $\beta$ belonging to $I_1$. So $J = I_1 \cup \{s_\alpha\}$ is a connected subset of the Coxeter diagram of $(W, S)$ containing $I_1$ as a strict subset. We claim that $w_{0,I_1} z s < w_{0,I_1} z$ for all $s\in J$. Indeed, if $s\in I_1$, then because $z$ centralizes $I_1$ and $z \in W^{I_1}$, we have $w_{0,I_1} z s = w_{0,I_1} s z < w_{0,I_1} z$, so the claim follows. Furthermore, because $\operatorname{ht}(z\alpha) < -64$, we conclude that $w_{0,I_1} z\alpha \in \Phi^-$. This implies $w_{0,I_1} z s_\alpha < w_{0,I_1} z$. This establishes the claim for all $s\in J$. Therefore, $J$ is a spherical subset of $S$. As $J$ is connected and contains $E_7$ as a proper subset, $J$ is of type $E_8$. 

    To establish the desired contradiction, we must show that $J$ satisfies the prerequisite condition of Definition \ref{def:valid_pair} for the valid pair $(v_1 z, v_2)$. Specifically, if we can prove that every simple reflection in $J$ commutes with $\operatorname{supp}(v_2)$, then the definition of a valid pair strictly dictates that $(v_1 z)\theta_J \in \Phi^+$. This would directly contradict our established finding that $v_1 z \theta_J \in \Phi^-$. 
    
    \textbf{Step 3: Completing the proof} Since the simple reflections in $I_1 \subset J$ already commute with $\operatorname{supp}(v_2)$ by construction, it remains only to prove that $s_\alpha$ centralizes $\operatorname{supp}(v_2)$. This is done as follows. We decompose $I_2 = I_3\sqcup I_4$ as
    \begin{align*}
        I_3 &:= \{s\in I_2\mid s s_\alpha\neq s_\alpha s\},\\
        I_4 &:= \{s\in I_2\mid s s_\alpha = s_\alpha s\}.
    \end{align*}
    Note that every $z'\in z W_{I_4}$ centralizes $w_{0,I_1}$ and satisfies $z'\alpha = z\alpha\in \Phi^-$. Now the set $D := \{s\in S\mid w_{0,I_1} z's<w_{0,I_1} z'\}$ contains $J$ and is spherical. Since $J$ is of type $E_8$, it follows that $D\cap I_3=\emptyset$. So if we decompose $z$ again as $z = z_3 z_4$ with $z_3\in W^{I_4}$ and $z_4\in W_{I_4}$, this argument shows that $z_3\in W^{I_3}$, and we simply get $z_3\in W^{I_1}$ as well from $z\in W^{I_1}$. Thus, $z_3\in W^I$.

    Again, the set $D' = \{s\in S\mid w_{0,I_1}z_3s<w_{0,I_1}z_3\}$ is a spherical subset of $S$ containing $J$ and being contained in $S\setminus I_2$. It follows that $D' = J$ as above. Hence, $\alpha$ is the only simple root with $z_3\alpha\in\Phi^-$.

    Decompose $z_3 = z_5 z_6$ with $z_5\in W_I$ and $z_6\in {}^I W$. The fact that $z_3\alpha\in \Phi^-$ implies $z_6\alpha\in \Phi^-\cup \Phi_I$. Then $z_6\neq 1$ because $z_3\alpha \in \Phi^-$ and $z_5\alpha\in \Phi^+$. Hence there exists a simple root $\alpha'$ with $z_6\alpha'\in \Phi^-$, and any such root will satisfy $z_3\alpha'\in \Phi^-$ since the product $z_3 = z_5 z_6$ is length additive. In view of the above argument, we see that $\alpha' = \alpha$ is the only simple root with $z_6\alpha'\in \Phi^-$. Finally, the fact that $z = z_5 z_6 z_4\in C_W(w_{0,I_1} v_2)$ shows that conjugation by $z_6$ maps $z_4^{-1} w_{0,I_1}v_2 z_4\in W_I$ to $z_5^{-1} w_{0,I_1} v_2 z_5\in W_I$. The fact that $z_6\in {}^I W^I$ thus implies that $z_6^{-1} s z_6\in I$ for all $s\in \operatorname{supp}(z_4^{-1} v_2 z_4)$ by Lemma~\ref{lem:parabolicConjugation}.

    Note that $\operatorname{supp}(v_2)\cap I_3 = \operatorname{supp}(z_4^{-1} v_2 z_4)\cap I_3$ from the fact that $z_4\in W_{I_4}$. Now given any simple reflection $s_1\in \operatorname{supp}(v_2)\cap I_3$, we get that $s_2 = z_6^{-1} s_1 z_6$ lies in $I$ by Lemma \ref{lem:parabolicConjugation}. Consider now the element $w := w_{0,I_1} s_2 z_6\in W$, which is a length additive product. We have $w\alpha\in \Phi^-$ since $z_6\alpha\in \Phi^-$. Moreover, we have $w\beta\in \Phi^-$ for all simple roots $\beta$ whose corresponding simple reflection lies in $I_1\cup\{s_1\}$, as we also can write $w$ as either length additive product
    \begin{align*}
        w = s_2 z_6 w_{0,I_1} = w_{0,I_1} z_6 s_1\in W.
    \end{align*}
    So the set $\{s\in S\mid ws<w\}$ contains $J\cup \{s_1\}$. This is not spherical. This contradiction shows that $\operatorname{supp}(v_2)\cap I_3=\emptyset$, finishing the proof.
\end{proof}

With these geometric bounds established, the proof of the main finiteness condition follows immediately.

\begin{proof}[Proof of Proposition \ref{prop:E7_lift_finiteness}]
    Suppose $(v_1, v_2)$ is a valid pair satisfying the condition. In the Hecke algebra, the product $T_w T_{v_1}$ expands as a linear combination $\sum_{z\leq w} c_z T_{zv_1}$ with $c_z \in \mathbb{Z}[q]$. The assumption $\tau(T_{v_1 v_2 w_{0,I_1}} T_w T_{v_1})\neq 0$ implies there exists an element $z\leq w$ such that $zv_1 = v_1 v_2 w_{0,I_1}$, or equivalently, $z = v_1 v_2 w_{0,I_1} v_1^{-1}$.
    
    Because $v_1 \in W^I$ and $v_2 w_{0,I_1} \in W_I$, the product $v_1 (v_2 w_{0,I_1})$ is length-additive. This allows us to bound the lengths:
    \begin{equation*}
        l(w) \geq l(z) \geq l(zv_1) - l(v_1) = l(v_1 v_2 w_{0,I_1}) - l(v_1) = l(v_2) + l(w_{0,I_1}).
    \end{equation*}
    In particular, this forces $l(v_2) \leq l(w)$ and $l(z) \leq l(w)$. Since $W$ contains only finitely many elements of bounded length, there are only finitely many possibilities for both $v_2$ and $z$.
    
    Let us fix such a choice of $v_2$ and $z$, and study the set of possible first components:
    \begin{equation*}
        V := \{v_1\in W^I \mid (v_1, v_2) \text{ is a valid pair and } z = v_1 w_{0,I_1} v_2 v_1^{-1}\}.
    \end{equation*}
    If $v_1, v_1' \in V$, then $(v_1')^{-1} z v_1' = v_1^{-1} z v_1 = w_{0,I_1} v_2$, which means $(v_1')^{-1} v_1$ commutes with $w_{0,I_1} v_2$. Thus, $V$ is contained in at most one right $C_W(w_{0,I_1} v_2)$-coset. By Lemma \ref{lem:validPairCentralizerFiniteness}, there can only be finitely many valid pairs within this coset, meaning $V$ is finite. Since there are only finitely many choices for the parameters $v_2$ and $z$, the total number of such valid pairs is finite.
\end{proof}

\subsection{Descent to the cocenter}
\begin{lemma}\label{lem:valid_reflection}
    Let $(v_1, v_2)$ be a valid pair and $s\in S$. Then exactly one of the following three claims is true.
    \begin{enumerate}
        \item The pair $(sv_1, v_2)$ is valid or
        \item the element $s' := v_1^{-1} s v_1$ lies in $I_2$ and $(v_1, v_2')$ is valid for any $v_2'\in\{s'v_2, v_2s', s' v_2 s'\}$ or
        \item there exists no valid pair $(v_1', v_2')$ with $v_1' v_2' = s v_1 v_2$, and we have
        \begin{align*}
            T_s T_{v_1} T_{(v_1 w_{0,I_1} v_2)^{-1}} = T_{v_1} T_{(v_1 w_{0,I_1} v_2)^{-1}} T_s
        \end{align*}
        in $H$.
    \end{enumerate}
    We write
    \begin{align*}
        (v_2, v_2)\in X_s^{(1)}\text{ resp.\ }X_s^{(2)}\text{ resp.\ }X_s^{(3)}
    \end{align*}
    in each respective case.
\end{lemma}
\begin{proof}
    First consider the case where $v_1^{-1} s v_1\in W_I$. Since $v_1\in W^I$, this means that $s' := v_1^{-1} s v_1$ must lie in $I$ by Lemma~\ref{lem:parabolicConjugation}.

    If $s'\in I_1$, then we get $sv_1 v_2\notin W^{I_1}$, so the first two cases are definitely not true. However, we get
    \begin{align*}
        T_s T_{v_1} T_{(v_1 w_{0,I_1} v_2)^{-1}} = T_{v_1} T_{s'}T_{(v_1 w_{0,I_1} v_2)^{-1}}
        =T_{v_1} T_{(v_1 w_{0,I_1} v_2)^{-1}}T_{s}
    \end{align*}
    as claimed.

    If $s'\in I_2$, then we consider $v_2'\in\{s'v_2, v_2s', s' v_2 s'\}$. We wish to show that $(v_1, v_2')$ is valid. For this, consider a set $J\supsetneq I_1$ of type $E_8$ which commutes with every simple reflection in $\supp(v_2')$. We have to show that $v_1\theta_J\in \Phi^+$. If it is true that $J$ commutes with every simple reflection in $\supp(v_2)$, then this is clear by the assumption that $(v_1, v_2)$ is valid. So assume this is not the case. Since $\supp(v_2')\cup \{s'\}= \supp(v_2)\cup \{s'\}$, this is only possible if $J$ does \emph{not} commute with $s'$ and $s'\in \supp(v_2)\setminus \supp(v_2')$.

    Let us decompose $v_1 = v_{1,1} v_{1,2}$ with $v_{1,1} \in W^{J\cup\{s'\}}$ and $v_{1,2}\in W_{J\cup\{s'\}}$. Let $\alpha_{s'}$ denote the simple root corresponding to $s'$. We note that $v_{1,2}\alpha_{s'}\in \Phi_{J\cup\{s'\}}$ is a root whose image under $v_{1,1}\in W^{J\cup\{s'\}}$ is simple. Thus, $v_{1,2}\alpha_{s'}$ is simple.

    Consider the element $\gamma := \theta_J + \alpha_{s'}\in \mathbb Z\Phi_{J\cup\{s'\}}$. Its pairing with all roots in $\Phi_{I_1}$ is zero (since $s'\in I_2$). For the unique simple reflection $s''\in J\setminus I_1$ and its simple root $\beta := \alpha_{s''}$, we get $\langle \beta^\vee,\theta_J\rangle=1$ (using an explicit calculation for $E_8$). Now the fact $\langle \beta^\vee,\alpha_{s'}\rangle\leq -1$ shows $\langle \beta^\vee,\gamma\rangle\leq 0$. Similarly, the fact that the pairing of $\alpha_{s'}$ with all roots in $\Phi_{I_1}$ is zero implies $\langle \alpha_{s'},\theta_J\rangle=2\langle \alpha_{s'},\beta\rangle \leq -2$. Hence $\langle \alpha_{s'},\gamma\rangle\leq 0$. We conclude that $\gamma$ is a positive imaginary root from \cite[Lemma~6.10]{Ma18}.

    Now consider $v_{1,2}\theta_J = v_{1,2}\gamma - v_{1,2}\alpha_{s'}$, where $v_{1,2}\gamma$ is positive imaginary and $v_{1,2}\alpha_{s'}$ is simple. It follows that $v_{1,2}\theta_J$ is a positive root in $\Phi_{J\cup\{s'\}}^+$. Hence $v_1\theta_J = v_{1,1} v_{1,2}\theta_J\in \Phi^+$ since $v_{1,1}\in W^{J\cup\{s'\}}$. 

    This shows that $(v_1, v_2')$ is a valid pair as claimed.
    
    It remains to consider the case where $v_1^{-1} s v_1\notin W_I$. Then $s v_1$ lies in $W^I$. We are done if $(sv_1, v_2)$ is valid. So let us assume this is not valid. This is only possible if there exists a set $I_1\subsetneq J\subseteq S$ of type $E_8$ which commutes with every simple reflection of $\supp(v_2)$ such that $sv_1\theta_J\in \Phi^-$. Since $v_1\theta_J\in \Phi^+$ by assumption, this means that $\theta_J = v_1^{-1}\alpha_s$.

    Now we decompose $v_1 = v_{1,1} v_{1,2}$ with $v_{1,1}\in W^J$ and $v_{1,2}\in W_J$. We note that $v_{1,2}\theta_J\in \Phi_J$ is a root whose image under $v_{1,1}$ is simple. Thus, $v_{1,2}\theta_J$ is a simple root again, which we denote by $\alpha\in \Phi_J$.

    We claim that the product $v_{1,2} w_{0,I_1} v_{1,2}^{-1}$ is length additive. Indeed, since $v_{1,2}\in W^I$, the product $v_{1,2} w_{0,I_1}$ is length additive. 
Now consider the usual length formula:
    \begin{align*}
        \ell(v_{1,2} w_{0,I_1} v_{1,2}^{-1}) &= \ell(v_{1,2} w_{0,I_1}) + \ell(v_{1,2}) \\
        &\quad - 2\#\{\beta\in \Phi^+\mid v_{1,2}\beta \in \Phi^- \text{ and } v_{1,2} w_{0,I_1} \beta \in \Phi^-\}.
    \end{align*}
    It remains to show that there is no $\beta\in \Phi^+$ with $v_{1,2}\beta, v_{1,2} w_{0,I_1}\beta\in \Phi^-$. Since $\supp(v_{1,2})\subseteq J$ and $v_{1,2}\in W^I$, such a root $\beta$ would lie in $\Phi_J^+$ but not in $\Phi_{I_1}$. Thus, $w_{0,I_1}\beta$ is a positive root and $\gamma := w_{0,I_1}\beta + \beta$ is a $\mathbb Z_{\geq 0}$-linear combination of simple roots in $\Phi_J^+$. It follows that $v_{1,2}\gamma$ is a $\mathbb Z_{\leq 0}$-linear combination of such roots. Moreover, the pairing of $\gamma$ with each root in $\Phi_{I_1}$ is zero. It follows that $\gamma$ is a $\mathbb Z$-multiple of $\theta_J$ (which also satisfies this property, and we use that the $E_8$ Cartan matrix is non-degenerate). We get a contradiction to the fact that $v_{1,2}\theta_J\in \Phi^+$ (from $(v_1, v_2)$ being valid and the observation that $v_1 = v_{1,1} v_{1,2}$ is length additive).

    The contradiction shows that $v_{1,2} w_{0,I_1} v_{1,2}^{-1}$ is indeed a length additive product. It is contained in $W_J$, which commutes with $\supp(v_2)$ by assumption. Hence we can compute
    \begin{align*}
        T_s T_{v_1} T_{(v_1 w_{0,I_1} v_2)^{-1}} &= T_{sv_1} T_{v_2^{-1}} T_{w_{0,I_1}} T_{v_1^{-1}}
        \\&=T_{v_{1,1} s_\alpha v_{1,2}} T_{v_2^{-1}} T_{w_{0,I_1}} T_{v_{1,2}^{-1}} T_{v_{1,1}^{-1}}
        \\&= T_{v_{1,1}} T_{s_{\alpha}} T_{v_{1,2}} T_{w_{0,I_1}} T_{v_{1,2}^{-1}} T_{v_2^{-1}} T_{v_{1,1}^{-1}}
        \\&= T_{v_{1,1}} T_{s_\alpha} T_{v_{1,2} w_{0,I_1} v_{1,2}^{-1}} T_{v_2^{-1}} T_{v_{1,1}^{-1}}.
    \end{align*}
    Observe that $i := v_{1,2} w_{0,I_1} v_{1,2}^{-1}\in W_J$ is an involution which commutes with $s_\alpha\in J$. We get $\ell(s_\alpha i) = \ell((s_\alpha i)^{-1}) = \ell(i s_\alpha)\in \ell(i)\pm 1$. If this is equal to $\ell(i)+1$, then we get
    \begin{align*}
        T_{s_\alpha} T_i = T_{s_\alpha i} = T_{is_\alpha} = T_i T_{s_\alpha}.
    \end{align*}
    If this is equal to $\ell(i)-1$, we get
    \begin{align*}
        T_{s_\alpha} T_i = \q T_{s_\alpha i} + (\q-1) i = \q T_{i s_\alpha} + (\q-1) i = T_i T_{s_\alpha}.
    \end{align*}

    So it is always true that $T_{s_\alpha}$ commutes with $T_i$. Continuing our calculation, and noting that $T_{s_\alpha}$ also commutes with $T_{v_2^{-1}}$ (since $s_\alpha \in J$ and $J$ commutes with $\supp(v_2)$):
    \begin{align*}
        T_{v_{1,1}} T_{s_\alpha} T_{v_{1,2} w_{0,I_1} v_{1,2}^{-1}} T_{v_2^{-1}} T_{v_{1,1}^{-1}}
        &=T_{v_{1,1}} T_{v_{1,2} w_{0,I_1} v_{1,2}^{-1}} T_{s_\alpha} T_{v_2^{-1}} T_{v_{1,1}^{-1}}
        \\&=T_{v_{1,1}} T_{v_{1,2} w_{0,I_1} v_{1,2}^{-1}} T_{v_2^{-1}} T_{s_\alpha}T_{v_{1,1}^{-1}}
        \\&=T_{v_{1,1}} T_{v_{1,2}} T_{w_{0,I_1}} T_{v_{1,2}^{-1}} T_{v_2^{-1}} T_{s_\alpha} T_{v_{1,1}^{-1}}
        \\&= T_{v_1} T_{w_{0,I_1}} T_{v_2^{-1}} T_{v_{1,2}^{-1} s_\alpha v_{1,1}^{-1}}
        \\&= T_{v_1} T_{w_{0,I_1}} T_{v_2^{-1}} T_{v_1^{-1} s}
        \\&= T_{v_1} T_{(v_1 w_{0,I_1} v_2)^{-1}} T_s.
    \end{align*}
    
    We see that only the third case is satisfied. This finishes the proof.
\end{proof}

\subsection{Proof of Theorem~\ref{thm:E7_lift}}
In this subsection, we complete the proof of Theorem~\ref{thm:E7_lift}.

Well-definedness of the map $F$ is proved in Proposition~\ref{prop:E7_lift_finiteness}. Let us next verify the fact that $F\vert_{[H,H]}=0$.
    
Since $H$ is generated, as an algebra over $\BZ[\q^{\pm 1}]$, by the elements $T_s$ for $s\in S$, it suffices to show that $F(h T_s) = F(T_s h)$ holds for all $s\in S$ and all $h\in H$.

        For a valid pair $(v_1, v_2)$, let us use the shorthand notation
        \begin{align*}
            F_{v_1, v_2}(h') := \q^{-\ell(v_1)-\ell(v_2)} f(T_{v_2}) \tau(T_{(v_1 w_{0,I_1} v_2)^{-1}} h' T_{v_1}),
        \end{align*}
        so that $F(h') = \sum_{\substack{(v_1, v_2)\text{ valid}}}F_{v_1, v_2}(h')$. Let us consider the sub-sums
        \begin{align*}
            F^{(i)}(h') &:= \sum_{(v_1, v_2)\in X_s^{(i)}} F_{v_1, v_2}(h')
        \end{align*}
        for $i=1,2,3$, where $X_s^{(i)}$ is the set defined in Lemma~\ref{lem:valid_reflection}. By the same lemma, we get $F = F^{(1)} + F^{(2)} + F^{(3)}$. We show that each of these three functions takes the same value on $T_s h$ as it takes on $h T_s$.

\textbf{Step 1.} First consider $F^{(1)}$. Its summands can be grouped into pairs according to the multiplication by $s$. Explicitly, for $(v_1, v_2)\in X_s^{(1)}$, we also get $(sv_1, v_2)\in X_s^{(1)}$ and hence
        \begin{align*}
            F^{(1)}(h') &= \sum_{\substack{(v_1, v_2)\in X_s^{(1)}\\ sv_1>v_1}} (F_{v_1, v_2}(h') + F_{sv_1, v_2}(h')).
        \end{align*}
        Each summand is given by
        \begin{align*}
        &F_{v_1, v_2}(h') + F_{sv_1, v_2}(h')
        \\&=\q \tau(T_{(v_1 w_{0,I_1} v_2)^{-1}} h T_s T_{v_1}) + \tau(T_{(v_1 w_{0,I_1} v_2)^{-1}} T_s h T_{s}^2 T_{v_1})
            \\&= \tau(T_{(v_1 w_{0,I_1} v_2)^{-1}} (\q h T_s + qT_s h + (\q-1) T_s h T_s) T_{v_1})
            \\&=\q \tau(T_{(v_1 w_{0,I_1} v_2)^{-1}}  T_s h T_{v_1}) + \tau(T_{(v_1 w_{0,I_1} v_2)^{-1}} T_s^2 h T_s T_{v_1}).
        \end{align*}
        Multiplication by $\q^{-\ell(sv_1)-\ell(v_2)} f(T_{v_2})$ yields $F_{v_1, v_2}(h T_s) + F_{sv_1, v_2}(h T_s) = F_{v_1, v_2}(T_s h) + F_{sv_1, v_2}(T_s h)$. We conclude $F^{(1)}(h T_s) = F^{(1)}(T_s h)$ as claimed.

\textbf{Step 2.} Next, we consider $F^{(2)}$. Given $(v_1, v_2)\in X_s^{(2)}$, set $s' = v_1^{-1} s v_1\in I_2$. Then also $(v_1, \sigma_1 v_2 \sigma_2)\in X_s^{(2)}$ for all $\sigma_1, \sigma_2\in \langle s'\rangle = \{1,s'\}$. It therefore suffices to show that the sum
\begin{align*}
    \hat F_{(v_1, v_2)} := \sum_{\sigma_1, \sigma_2\in \langle s'\rangle} F_{v_1, \sigma_1 v_2 \sigma_2}: H\to \BZ[\q^{\pm 1}]
\end{align*}
takes the same value on $T_s h$ as it takes on $h T_s$. For this, we define an $\mathbb Z[\q^{\pm 1}]$-linear map
\begin{align*}
    \pi_{v_1, v_2} : H\to H_{I_2},~T_w\mapsto\begin{cases} T_{v_2'}&\text{ if }\exists v_2'\in \langle s'\rangle v_2 \langle s'\rangle\text{ s.th.\ }w = v_1 w_{0,I_1} v_2',\\0,&\text{otherwise}\end{cases}
\end{align*}
Since $v_1 w_{0,I_1}\in W^{I_2}$, we see that $\pi_{v_1, v_2}(hT_{s'}) = \pi_{v_1, v_2}(h)T_{s'}$ holds for all $h\in H$. Moreover, the simple observation $\ell(sv_1) = \ell(v_1 s') = \ell(v_1)+1$ implies that
\begin{align*}
    T_s T_{v_1} = T_{sv_1} = T_{v_1s'} = T_{v_1} T_{s'}
\end{align*}
and similarly $T_s T_{v_1 w_{0,I_1} v_2} = T_{v_1 w_{0,I_1}} T_{s'} T_{v_2}$. We conclude that
\begin{align*}
    \pi_{v_1, v_2}(T_s h) = T_{s'}\pi_{v_1, v_2}(h)
\end{align*}
holds for all $h\in H$.

Now $\hat F_{(v_1, v_2)}(h')$ is a $\mathbb Z_{>0}$-multiple of $f\circ \pi_{v_1, v_2}(h' T_{v_1})$ by the usual properties of $\tau$. We get
\begin{align*}
    f(\pi_{v_1, v_2}(T_s h T_{v_1})) &= f(T_{s'} \pi_{v_1, v_2}(h T_{v_1}))
    =f(\pi_{v_1, v_2}(hT_{v_1}) T_{s'}) \\&= f(\pi_{v_1, v_2}(hT_{v_1} T_{s'}))
    =f(\pi_{v_1, v_2}(h T_s T_{v_1})).
\end{align*}
In particular, we also get $\hat F_{v_1, v_2}(T_s h) = \hat F_{v_1, v_2}(h T_s)$.

\textbf{Step 3.} We still have to study the function $F^{(3)}$. By Lemma~\ref{lem:valid_reflection}, this is thankfully very easy. If $(v_1, v_2)\in X_s^{(3)}$, we conclude
        \begin{align*}
            &\tau(T_{(v_1 w_{0,I_1} v_2)^{-1}} h T_s T_{v_1}) = \tau(T_s T_{v_1} T_{(v_1 w_{0,I_1} v_2)^{-1}} h_1)
            \\&\underset{\text{L\ref{lem:valid_reflection}}}=\tau(T_{v_1} T_{(v_1 w_{0,I_1} v_2)^{-1}} T_s h) = \tau(T_{(v_1 w_{0,I_1} v_2)^{-1}} T_s h T_{v_1}).
        \end{align*}
        Multiplied by $\q^{-\ell(v_1)-\ell(v_2)} f(T_{v_2})$, we see $F_{v_1, v_2}(h T_s) = F_{v_1, v_2}(T_s h)$. We conclude that $F^{(3)}(h T_s) = F^{(3)}(T_s h)$.

\textbf{Step 4.} The last remaining claim of Theorem~\ref{thm:E7_lift} concerns the specialization of $\q$ to $1$. Under this specialization, we see that $F(T_w)$ becomes
        \begin{align*}
            F(T_w)_{\q=1} = \sum_{\substack{(v_1, v_2)\text{ valid}\\v_2 w_{0,I_1} = v_1^{-1} w v_1}} f(T_{v_2})_{\q=1}.
        \end{align*}
        We observe that if $w$ is not $W$-conjugate to any element in $W_{I_2} w_{0,I_1}$, then $F(T_w)_{\q=1}=0$. Thus, let us assume we are given an element $w_2\in W_{I_2}$ such that $w$ is conjugate to $w_{0,I_1} w_2$.
        
        By (1), the value of $F(T_w)_{\q=1}$ only depends on the $W$-conjugacy class of $w$. Similarly, the value of $f(T_{w_2})_{\q=1}$ only depends on the $W_{I_2}$-conjugacy class of $w_2$. So let us assume that $w = w_{0,I_1} w_2$ and that $w_2$ is of minimal length in its $W_{I_2}$-conjugacy class.
        Then $(1, w_2)$ is a valid pair, whose contribution to the above sum is $f(T_{w_2})_{\q=1}$.

        Note from the above that $w\in W_I$ has minimal length in its $W_I$-conjugacy class, hence it has minimal length in its $W$-conjugacy class (Theorem~\ref{min}).
        If $(v_1, v_2)\neq (1, w_2)$ is another valid pair with $v_1^{-1} w v_1 = v_2 w_{0,I_1}$, then we get $v_1\neq 1$. Since $v_1\in W^I$ and $w, v_2 w_{0,I_1}\in W_I$ with $w$ having minimal length in its $W$-conjugacy class, it follows from Lemma~\ref{lem:parabolicConjugation} that $v_1 s v_1^{-1}\in I$ for all $i\in J := \supp(v_2 w_{0,I_1})$. Therefore, $v_1$ is a (length additive) product of elementary conjugators in the sense of Proposition~\ref{prop:LSalgo}.
        
        Explicitly, this means that there exists $s\in S\setminus J$ such that the irreducible component $C$ of $J\sqcup\{s\}$ containing $s$ is spherical, and the elementary conjugator $w = w_{0,C\cap J} w_{0,C}$ satisfies $\ell(v_1 w) = \ell(v_1) - \ell(w)$. In particular, $v_1s<v_1$ and hence $s\notin I$ (using $v_1\in W^I$). The only possibility for this is if $C = I_1\sqcup\{s\}$ is of type $E_8$, and contained in $C_W(\supp(v_2))$, as otherwise $C$ will not be spherical. But then $w \theta_C\in \Phi^-$ implies $v_1 \theta_C\in \Phi^-$, contradicting the assumption of a valid pair.

        The contradiction shows that $(1, w_2)$ is the only valid pair contributing to our above formula for $F(T_w)_{\q=1}$.
This finishes the proof of Theorem~\ref{thm:E7_lift}.

\section{Proof of the Main Theorem}\label{sec:proofOfMainThm}

Throughout this section, $(W,S)$ denotes the Weyl group of the
fixed split Kac--Moody group $G$. All parabolic subsystems
$(W_I,I)$ occurring in the induction are again crystallographic
Weyl groups.

\subsection{Weakly separating property}
By definition, the set of canonical elements $\{T_{\mathcal{O}}\mid \mathcal{O}\in Cl(W)\}\subseteq \overline{H}$ is linearly independent if and only if the following separating property is satisfied:

\smallskip\noindent\textbf{Property (SP):} For each $\mathcal{O}\in Cl(W)$, there exists a function $f_{\mathcal{O}} : \overline{H} \to \mathbb{Z}[\q^{\pm 1}]$ such that $f_{\mathcal{O}}(T_{\mathcal{O}'}) =\delta_{\mathcal{O},\mathcal{O}'}$ for all $\mathcal{O}'\in Cl(W)$. 

\smallskip
We introduce the following weaker version of the separating property, evaluated over the extended ring $A = \mathbb{Q}[\q^{\pm 1/2}]$: 

\smallskip\noindent\textbf{Property (WSP):} For each $\mathcal{O}\in Cl(W)$, there exists an $A$-linear function $f_{\mathcal{O}}: \overline{H}_A\to A$ such that for all $\mathcal{O}'\in Cl(W)$, we have
\begin{align*}
    f_{\mathcal{O}}(T_{\mathcal{O}'}) \equiv \delta_{\mathcal{O},\mathcal{O}'}\pmod{\q-1}.
\end{align*}

To execute our induction, we must carefully isolate irreducible components of type $E_8$. Let $J \subseteq S$. We say that $J$ is \emph{small} if $(W_J, J)$ is a spherical Coxeter system and $J$ does not contain any irreducible component of type $E_8$. 

\smallskip\noindent\textbf{Property (WSP) for small $J$:} 
For each $\mathcal{O}\in Cl(W)$ with $\mathcal{O} \cap W_J \neq \emptyset$, there exists an $A$-linear function $f_{\mathcal{O}}: \overline{H}_A\to A$ such that for all $\mathcal{O}'\in Cl(W)$, we have
\begin{align*}
    f_{\mathcal{O}}(T_{\mathcal{O}'}) \equiv \delta_{\mathcal{O},\mathcal{O}'}\pmod{\q-1}.
\end{align*}

\subsection{(WSP) $\Rightarrow$ (SP)}\label{subsec:almostIndependenceImpliesIndependence}
We assume that Property (WSP) holds for a Coxeter group $(W, S)$. 

Suppose that $c_1 T_{\mathcal{O}_1}+\cdots+c_n T_{\mathcal{O}_n}=0$ for pairwise distinct $\mathcal{O}_1,\dotsc,\mathcal{O}_n\in Cl(W)$ and coefficients $c_1,\dotsc,c_n\in \mathbb{Z}[\q^{\pm 1}]$. Let $f_{\mathcal{O}_1},\dotsc, f_{\mathcal{O}_n}$ be the corresponding functions given by Property (WSP). Define the matrix $M\in A^{n\times n}$ by
\begin{align*}
    M = (f_{\mathcal{O}_i}(T_{\mathcal{O}_j}))_{i,j=1,\dotsc,n}.
\end{align*}
By the definition of Property (WSP), evaluating this matrix modulo $\q-1$ yields the identity matrix. Because $\det M \equiv 1 \pmod{\q-1}$, the determinant $\det M \in A$ cannot be the zero polynomial.

We naturally identify $\mathbb{Z}[\q^{\pm 1}]$ as a subring of $A$. Let $v = (c_1,\dotsc,c_n)^T\in A^n$. Applying the linear maps $f_{\mathcal{O}_i}$ to our initial linear dependence equation yields the matrix equation $Mv=0$. Because the determinant of $M$ is non-zero and $A$ is an integral domain, the matrix $M$ is invertible over the fraction field of $A$, which forces $v=0$. Hence $c_1=\cdots=c_n=0$, establishing the linear independence required by Property (SP).

\subsection{(WSP) for small $J$ and for smaller $(W', S')$ $\Rightarrow$ (WSP) for $(W, S)$}

We first establish the case where $(W, S)$ is not irreducible, i.e.\ where there exists a decomposition into non-empty subsets $S = S_1\sqcup S_2$ with $s_1 s_2 = s_2 s_1$ for all $s_1\in S_1, s_2\in S_2$. Then the Hecke algebra decomposes into the tensor product $H_A = H_{S_1,A}\otimes_A H_{S_2, A}$, and each conjugacy class $\CO\in \Cl(W)$ is of the form $\CO = \CO_1\CO_2$ with $\CO_i\in \Cl(W_{S_i})$. Similarly, each $w\in W$ can be uniquely written as $w = w_1 w_2$ with $w_i\in W_{S_i}$. It remains to set
\begin{align*}
    f_{\CO}(T_w) := f_{\CO_1}(T_{w_1}) f_{\CO_2}(T_{w_2}),
\end{align*}
where $f_{\CO_i}$ is the function obtained from the inductive assumption that (WSP) holds for the smaller parabolic Weyl group $(W_{S_i}, S_i)$. This finishes the proof if $(W, S)$ is not irreducible.

Assume now that $(W, S)$ is irreducible. Let $\CO\in \Cl(W)$ be a conjugacy class. Let $J = \supp(w)$ for some minimal length representative $w \in \CO_{\min}$. If $J=S$, then we obtain a desired function $f_{\CO}$ from Theorem~\ref{thm:almostEllipticClassPolynomial}. If $J$ is \emph{small}, then we obtain a desired function $f_{\CO}$ from our assumption.
    
Suppose that $J\neq S$ but it is not small. Then, there exists an irreducible component $J'\subseteq J$ of $J$ such that $J'$ is of type $E_8$ or $W_{J'}$ is infinite.

Set $I= \{s\in S\mid ss' = s's \text{ for all } s' \in J'\}\cup J'$. Since $J' \subset J \subsetneq S$ and $(W, S)$ is irreducible, $I$ is a proper subset of $S$. By inductive hypothesis on $(W_I, I)$, we obtain a function $f : \overline H_I\to A$ with $f(T_{\CO'}) = \d_{\CO, \CO'}$.

By the Lusztig-Spaltenstein algorithm (see Proposition \ref{prop:LSalgo}), for every $K \subset S$ that is $J$-conjugate to $J$, we have $J' \subset K$. Therefore $N_W(K)\subseteq W_I$. Let us base-change Theorem~\ref{thm:classPolynomialLifting} and its proof from $\mathbb Z[\q^{\pm 1}]$ to $A$, then apply it to $(I, J, f)$. This yields a function $\indFunctionName : \overline H\to A$, which is our desired function $f_{\CO}$. 

\subsection{Inductive proof of Property (WSP) for small $J$}
We now prove Property~\textup{(WSP)} for small subsets of Weyl groups of Kac--Moody groups. We proceed by induction on $\#S$, assuming the assertion for all proper parabolic Weyl subsystems $(W_K,K)$.


Fix a small subset $J \subseteq S$ (meaning $W_J$ is a finite Coxeter group containing no irreducible component of type $E_8$). To isolate the anomalies caused by type $E_7$ components, we partition the relevant conjugacy classes. Let
\[
\mathcal{C}_{1}=\{\mathcal{O}_1 \in Cl(W_J) \mid \forall J' \subsetneq J : \mathcal{O}_1 \cap W_{J'} = \emptyset\}
\]
denote the set of elliptic conjugacy classes in $W_J$. For any irreducible component $J' \subseteq J$, we write $W_J = W_{J'} \times W_{J \setminus J'}$ and denote by $\pi_{J'}$ the projection map $W_J \to W_{J'}$. We define the subset of anomalous local classes as
\[
\mathcal{C}_{2}=\{\mathcal{O}_1 \in \mathcal{C}_{1} \mid \exists J' \subseteq J \text{ an irreduc. comp. of type } E_7 \text{ s.t. } \pi_{J'}(\mathcal{O}_1) = \{w_{0,J'}\}\}.
\]

We call the conjugacy classes of $W_J$ the ``local classes'' and the conjugacy classes of $W$ the ``global classes''. 

We globalize the local class sets $C_1$ and $C_2$ to the ambient Weyl group $W$ as follows:
\begin{align*}
\mathcal{C}_{3}&=\{\mathcal{O} \in Cl(W) \mid \exists \mathcal{O}_1 \in \mathcal{C}_1 \text{ such that } \mathcal{O}_1 \subseteq \mathcal{O}\}, \\
\mathcal{C}_{4}&=\{\mathcal{O} \in Cl(W) \mid \exists \mathcal{O}_2 \in \mathcal{C}_2 \text{ such that } \mathcal{O}_2 \subseteq \mathcal{O}\}.
\end{align*}
Thus, $\mathcal{C}_3$ contains all global conjugacy classes whose minimal support is $W$-conjugate to $J$, while $\mathcal{C}_4 \subseteq \mathcal{C}_3$ contains the globalized $E_7$-anomalous classes.

From our method of orbital integrals, we obtain the following generic separation result for non-anomalous classes.

\begin{proposition}\label{prop:smallwsp_oi}
For every $\mathcal{O}\in\mathcal{C}_{3}$, there exists an $A$-linear trace function $f:\overline{H}_{A}\rightarrow A$ such that for all $\mathcal{O}^{\prime}\in Cl(W)\setminus\mathcal{C}_{4}$, we have
\[
f(T_{\mathcal{O}^{\prime}})\equiv\delta_{\mathcal{O},\mathcal{O}^{\prime}} \pmod{\q-1}.
\]
\end{proposition}

\begin{proof}
Pick a local elliptic class $\mathcal{O}_1 \in \mathcal{C}_1$ such that $\mathcal{O}_1 \subseteq \mathcal{O}$. By Corollary \ref{cor:perfectPairingConsequences} applied to the finite Coxeter group $W_J$, there exist rational coefficients $c_{\mathcal{O}_1'} \in \mathbb{Q}$ for $\mathcal{O}_1' \in \mathcal{C}_1$ such that for any non-anomalous class $\tilde{\mathcal{O}}_1 \in \mathcal{C}_1 \setminus \mathcal{C}_2$, we have
\[
\sum_{\mathcal{O}_1' \in \mathcal{C}_1} c_{\mathcal{O}_1'} \psi(\tilde{\mathcal{O}}_1 \otimes \mathcal{O}_1') = \delta_{\tilde{\mathcal{O}}_1, \mathcal{O}_1},
\]
where $\psi = \psi_J$ is the generic pairing for $(W_J, J)$. 

For each global class $\mathcal{O}_3' \in \mathcal{C}_3$, list the local classes in $\mathcal{C}_1$ contained in $\mathcal{O}_3'$ as $\mathcal{O}^{(1)}, \dots, \mathcal{O}^{(m)}$. These local classes form a single orbit under the action of $\operatorname{Aut}(W, J)$. We set $C_{\mathcal{O}_3'} := \sum_{i=1}^m c_{\mathcal{O}^{(i)}}$ to be the average of these coefficients, and construct the global function
\[
f := \#\{\tau(\CO_1)\mid\tau\in \Aut(W, J)\}\sum_{\mathcal{O}_3' \in \mathcal{C}_3} C_{\mathcal{O}_3'} \Psi_{\mathcal{O}_3'} : \overline{H}_A \longrightarrow A.
\]
By Theorem \ref{thm:psi_global} and Theorem \ref{thm:lusztigPairingOverview}, $f$ factors through $\overline{H}_A$ and satisfies $f(T_{\tilde{\mathcal{O}}}) \equiv 0 \pmod{\q-1}$ whenever $\tilde{\mathcal{O}} \notin \mathcal{C}_3$.

    It remains to study the case $\tilde \CO\in \mathcal C_3\setminus \mathcal C_4$. Choose $\tilde\CO_1\in \mathcal C_1$ with $\tilde\CO_1\subseteq \tilde \CO$. Then we compute using Theorem~\ref{thm:psi_global} that
    \begin{align*}
    C_{\CO_3'}\Psi_{\CO_3'}(T_{\tilde\CO})=\sum_{\mathcal C_1\ni\CO_1'\subseteq \CO_3'} C_{\CO_3'} \psi(\tilde \CO_1\otimes \CO_1').
    \end{align*}
    For any fixed $\CO_1^\ast\subseteq \CO_3'$, the set of conjugacy classes in $\mathcal C_1$ contained in $\CO_3'$ forms a single orbit under the action of $\Aut(W, J)$. Denote the cardinality of this orbit by $m$ as before. We conclude
    \begin{align*}
        C_{\CO_3'} \Psi_{\CO_3'}(T_{\tilde\CO}) = \frac m{\#\Aut(W, J)} \sum_{\tau\in \Aut(W, J)} C_{\CO_3'} \psi(\tilde \CO_1\otimes \tau \CO_1^\ast).
    \end{align*}
    By definition of $C_{\CO_3'}$, we get
    \begin{align*}
        C_{\CO_3'} \Psi_{\CO_3'}(T_{\tilde\CO}) &= \frac m{(\#\Aut(W, J))^2} \sum_{\tau_1,\tau_2\in \Aut(W, J)} c_{\tau_1\CO_1^\ast} \psi(\tilde \CO_1\otimes \tau_2 \CO_1^\ast)
        \\&\underset{\tau_2\mapsto \tau_2\tau_1}= \frac m{(\#\Aut(W, J))^2} \sum_{\tau_1,\tau_2\in \Aut(W, J)} c_{\tau_1\CO_1^\ast} \psi(\tilde \CO_1\otimes \tau_2 \tau_1\CO_1^\ast)
        \\&=\frac m{(\#\Aut(W, J))^2} \sum_{\tau_1,\tau_2\in \Aut(W, J)} c_{\tau_1\CO_1^\ast} \psi((\tau_2^{-1}\tilde \CO_1)\otimes  \tau_1\CO_1^\ast)
        \\&=\frac 1{\#\Aut(W, J)}\sum_{\tau\in \Aut(W, J)} \sum_{\mathcal C_1\ni \CO_1'\subseteq \CO_3} c_{\CO_1'}\psi((\tau\tilde \CO_1)\otimes \CO_1')
    \end{align*}
    using that $\psi$ is invariant under Coxeter diagram automorphisms of $(W_J, J)$.
    Taking the sum over all $\CO_3'\in \mathcal C_3$, we summarize
    \begin{align*}
        &\frac{f(T_{\tilde\CO})}{\#\{\tau(\CO_1)\mid\tau\in \Aut(W, J)\}} \\&= \sum_{\CO_3'\in \mathcal C_3}
        \frac 1{\#\Aut(W, J)} \sum_{\tau\in \Aut(W, J)}\sum_{\mathcal C_1\ni \CO_1'\subseteq \CO_3} c_{\CO_1'}\psi((\tau\tilde \CO_1)\otimes \CO_1')
        \\&= \frac 1{\#\Aut(W, J)} \sum_{\tau\in \Aut(W, J)}\sum_{\CO_1'\in\mathcal C_1} c_{\CO_1'}\psi((\tau\tilde \CO_1)\otimes \CO_1')
        \\&=\frac 1{\#\Aut(W,J)}\sum_{\tau \in \Aut(W, J)}\delta_{\tau\tilde \CO_1,\CO_1}
        \\&=\delta_{\tilde\CO,\CO}\frac{\#\{\tau\in \Aut(W, J)\mid \tau(\CO_1) = \CO_1\}}{\#\Aut(W, J)}.
    \end{align*}
    By the orbit-stabilizer formula, we indeed obtain $f(T_{\tilde\CO}) = \delta_{\CO,\tilde\CO}$.
\end{proof}

We now apply our special construction for type $E_7$ to handle the anomalous classes in $\mathcal{C}_4$.

\begin{proposition}\label{prop:smallwsp_special}
For every $\mathcal{O}\in\mathcal{C}_{4}$, there exists an $A$-linear trace function $F_{\mathcal{O}}:\overline{H}_A\rightarrow A$ such that for all $\mathcal{O}^{\prime}\in Cl(W)$, we have
\[
F_{\mathcal{O}}(T_{\mathcal{O}^{\prime}})\equiv\delta_{\mathcal{O},\mathcal{O}^{\prime}} \pmod{\q-1}.
\]
\end{proposition}

\begin{proof}
    Let $\CO_2\in \mathcal C_2$ with $\CO_2\subseteq \CO$. Let $I_1\subseteq J$ be an irreducible component of $J$ of type $E_7$ with $\pi_{I_1}(\CO_2) = \{w_{0,I_1}\}$. Define
    \begin{align*}
        I_2 = \{s\in S\mid ss' = s's\text{ for all }s'\in I_1\},
    \end{align*}
    so that $J_2 := J\setminus I_1\subseteq I_2$. Write $W_J = W_{I_1}\times W_{J_2}$, and let $\underline \CO_2$ be the image of $\CO_2$ in $W_{J_2}$. 
    
By the inductive hypothesis applied to the smaller Coxeter subsystem $(W_{I_2}, I_2)$ and the small subset $J_2 \subseteq I_2$, we obtain a function $f: H_{A, I_2} \to A$ satisfying $f(T_{\underline{\mathcal{O}}}) \equiv \delta_{\underline{\mathcal{O}}_2, \underline{\mathcal{O}}} \pmod{\q-1}$ for all $\underline{\mathcal{O}} \in Cl(W_{I_2})$. 

Lift it to a function $F : \overline{H_A}\to A$ as in Theorem~\ref{thm:E7_lift} (with a base change of the theorem and its proof from $\BZ[\q^{\pm 1}]$ to $A$). Then $F$ satisfies the desired conditions by Theorem~\ref{thm:E7_lift}.
\end{proof}

Combining Proposition~\ref{prop:smallwsp_oi} with Proposition~\ref{prop:smallwsp_special}, we can finish the proof of (WSP) for small sets $\tilde J$:

If $\CO\in \Cl(W)$ is given with $W_{\tilde J}\cap \CO\neq\emptyset$, consider a minimal subset $J\subseteq \tilde J$ with $W_{J}\cap \CO\neq\emptyset$. Then $J$ is small. With our notation as above, this means $\CO\in \mathcal C_3$.

From Proposition~\ref{prop:smallwsp_oi}, we get a function $g : \overline H_A\to A$ satisfying
\begin{align*}
    g(T_{\CO'}) \equiv \delta_{\CO,\CO'}\pmod{\q-1}
\end{align*}
holds for all $\CO'\in \mathcal C_3\setminus \mathcal C_4$. For each $\tilde \CO\in \mathcal C_4$, we get a function $F_{\tilde\CO} : \overline H_A\to A$ satisfying
\begin{align*}
    F_{\tilde\CO}(T_{\CO'}) \equiv \delta_{\tilde\CO,\CO'}\pmod{\q-1}
\end{align*}
for all $\CO'\in \mathcal C_3$.

The function $f$ required to get (WSP) for $\CO$ is now given by
\begin{align*}
    f = g - \sum_{\tilde\CO\in \mathcal C_4} (g(T_{\tilde\CO})_{\q=1} - \delta_{\tilde\CO,\CO})F_{\tilde\CO} : \overline H_A\to A.
\end{align*}
This finishes the proof of (WSP) for small sets $\tilde J$ and hence the proof of Theorem~\ref{thm:intro_main}.

\section{Deligne-Lusztig varieties of Kac-Moody groups}\label{sec:dl_var}

\subsection{Class polynomials}

As an immediate consequence of our main result, Theorem~\ref{thm:intro_main}, we can formally define the \emph{class polynomials} for arbitrary Kac--Moody Hecke algebras and rigorously establish their positivity properties.

\begin{proposition}\label{prop:class_polynomials}
    For each $w \in W$, there exist uniquely determined polynomials $f_{w, \mathcal O} \in \mathbb N[\mathbf q - 1]$ for $\mathcal O \in \operatorname{Cl}(W)$ such that $f_{w, \mathcal O} = 0$ for all but finitely many $\mathcal O$, and
    \begin{align*}
        T_w \equiv \sum_{\mathcal O \in \operatorname{Cl}(W)} f_{w, \mathcal O} \, T_{\mathcal O} \pmod{[H,H]}.
    \end{align*}
\end{proposition}

\begin{proof}
By Proposition~\ref{prop:cocenterSpanning}, the cocenter $\overline H$ is spanned by the elements $\{T_{\mathcal O}\}_{\mathcal O \in \operatorname{Cl}(W)}$. The polynomials $f_{w, \mathcal O}$ can be constructed recursively using the following reductions:

\begin{enumerate}
    \item If $w$ has minimal length in its conjugacy class $\mathcal O'$, then
    \begin{align*}
        f_{w, \mathcal O} = \delta_{\mathcal O, \mathcal O'}.
    \end{align*}
    
    \item If $s \in S$ satisfies $\ell(sws) = \ell(w) - 2$, then
    \begin{align*}
        f_{w, \mathcal O} = \mathbf q \, f_{sws, \mathcal O} + (\mathbf q - 1) \, f_{sw, \mathcal O}.
    \end{align*}
    
    \item If $s \in S$ satisfies $\ell(sws) = \ell(w)$, then
    \begin{align*}
        f_{w, \mathcal O} = f_{sws, \mathcal O}.
    \end{align*}
\end{enumerate}

These formulas follow directly from the quadratic relation in the Hecke algebra:
\[
T_s^2 = (\mathbf q - 1) T_s + \mathbf q,
\]
and the fact that, in the cocenter, $T_w$ is unchanged under the equivalence relation $w \approx sws$ when $\ell(sws) = \ell(w)$. The recursion terminates because the length $\ell(w)$ strictly decreases in the second case, while the first case provides the base of the induction.

The existence of the polynomials $f_{w, \mathcal O}$ is thus established; the recursive formulas also show that each $f_{w, \mathcal O} \in \mathbb N[\mathbf q - 1]$. Uniqueness follows from the linear independence of the elements $\{T_{\mathcal O}\}$ in $\overline H$, which is the content of Theorem~\ref{thm:intro_main}. 
\end{proof}

\subsection{Deligne-Lusztig varieties for Kac-Moody groups}\label{subsec:kac-moody-deligne-lusztig}

Let \(G^{\bullet} / \mathbb F_q\) be the Kac--Moody group associated to our root system, split over $\mathbb F_q$, where \(\bullet \in \{\min, \max\}\). 
Associated to it we have the \emph{(thin) Kac--Moody flag variety}, which we denote by \(G^{\bullet} / B^{\bullet}\); this is an ind-scheme in the sense of \cite[Definition~1.6]{Ka20}.

We can then define the \emph{Deligne--Lusztig variety} associated to \(w \in W\) and \(b \in G^{\bullet}(\overline{\mathbb F}_q)\) as the reduced locally closed subscheme of \(G^{\bullet}/B^{\bullet}\) whose geometric points are given by
\begin{align*}
X_w(b)(\overline{\mathbb F}_q) = \{\, g \in G^{\bullet}(\overline{\mathbb F}_q) / B^{\bullet}(\overline{\mathbb F}_q) \mid g^{-1} b \, \sigma(g) \in B^{\bullet}(\overline{\mathbb F}_q) \, w \, B^{\bullet}(\overline{\mathbb F}_q) \,\},
\end{align*}
where \(\sigma\) denotes the Frobenius automorphism of \(\overline{\mathbb F}_q / \mathbb F_q\).

Classical Deligne--Lusztig varieties were introduced by Deligne and Lusztig \cite{DL76} and play a central role in the study of representations of finite groups of Lie type. Affine Deligne--Lusztig varieties were introduced by Rapoport \cite{Ra}; they serve as group-theoretic models for the reduction of Shimura varieties and Shtukas, and play an important role in arithmetic geometry and the Langlands program.

The \emph{Deligne--Lusztig reduction method} allows one to decompose arbitrary Deligne--Lusztig varieties into smaller pieces, ultimately reducing to those attached to minimal-length elements. This method was first introduced by Deligne and Lusztig \cite{DL76} for finite types, and later generalized to the affine setting in \cite[Corollary~2.5.3]{GH10}. The general case for arbitrary Kac--Moody groups can be proved by similar arguments.

\begin{proposition}[Deligne--Lusztig reduction]\label{prop:KacMoodyDeligneLusztigReduction}
Let \(w \in W\) and \(b \in G^{\bullet}(\overline{\mathbb F}_q)\).
\begin{enumerate}
    \item If \(s \in S\) satisfies \(\ell(sws) = \ell(w)\), then \(X_w(b)\) is universally homeomorphic to \(X_{sws}(b)\).
    \item If \(s \in S\) satisfies \(\ell(sws) = \ell(w) - 2\), then there exists a decomposition
    \begin{align*}
        X_w(b) = X_1 \sqcup X_2
    \end{align*}
    such that \(X_1 \subseteq X_w(b)\) is open, \(X_2 \subseteq X_w(b)\) is closed, and there are maps
    \begin{align*}
        p_1 : X_1 \longrightarrow X_{sw}(b), \qquad p_2 : X_2 \longrightarrow X_{sws}(b)
    \end{align*}
    which are compositions of Zariski locally trivial \(\mathbb G_m\)-fibre bundles (for \(p_1\)) and \(\mathbb A^1\)-fibre bundles (for \(p_2\)) with universal homeomorphisms.
\end{enumerate}
\end{proposition}

\subsection{Dimension=degree theorem}

As an application of our main result on the cocenter of Hecke algebras, we establish the ``dimension=degree'' theorem for the Deligne-Lusztig variety \(X_w(1)\).

\begin{theorem}\label{thm:dim=deg}
    Let \(w \in W\) and set
    \begin{equation}\label{eq:sum-class}
        f_{w,[1]} = \sum_{\mathcal O} \mathbf q^{\ell(\mathcal O)} f_{w,\mathcal O},
    \end{equation}
    where the sum is taken over all conjugacy classes of finite-order elements in \(W\), and \(\ell(\mathcal O) = \min\{\ell(v) \mid v \in \mathcal O\}\) for each \(\mathcal O \in \operatorname{Cl}(W)\). Then
    \begin{align*}
        \dim X_w(1) = \deg_{\mathbf q} f_{w,[1]},
    \end{align*}
    with the convention that \(\dim \emptyset = -\infty = \deg 0\).
\end{theorem}

The proof of Theorem~\ref{thm:dim=deg} proceeds by induction on the length \(\ell(w)\), using the recursive structure of the class polynomials (Proposition~\ref{prop:class_polynomials}) and the Deligne-Lusztig reduction (Proposition~\ref{prop:KacMoodyDeligneLusztigReduction}). The base case, where \(w\) is of minimal length in its conjugacy class, is handled separately. The inductive step then follows by comparing the two decompositions—one algebraic, one geometric—and observing that the degree of a sum of polynomials with non-negative coefficients is the maximum of their degrees.

\begin{remark}
    When \(W\) is a finite Weyl group, the summation in \eqref{eq:sum-class} is over all conjugacy classes of \(W\). It is easy to show using the inductive formula for class polynomials that \(\deg f_{w,[1]} = \ell(w)\). By \cite{DL76}, \(X_w(1)\) is locally homeomorphic to the Bruhat cell, so in particular \(\dim X_w(1) = \ell(w)\). Thus, in the finite type case, both quantities coincide for elementary reasons.
\end{remark}

In the affine case, the situation is already more subtle. The variety \(X_w(1)\) is not always nonempty, and when it is nonempty, its dimension is governed by a complicated formula. Affine Deligne-Lusztig varieties serve as group-theoretic models for the intersection of Kottwitz-Rapoport strata with the basic Newton stratum in the reduction of Shimura varieties. Understanding their nonemptiness patterns and dimension formulas is a fundamental problem in arithmetic geometry. We refer to the survey article of the first author at ICM2018 \cite{He-ICM} for details.

It was discovered by the first author in \cite{He-Ann} that the dimension of affine Deligne-Lusztig varieties can be computed via the degree of the class polynomials in affine Hecke algebras. This provides a powerful and computationally effective tool for studying affine Deligne-Lusztig varieties. Theorem~\ref{thm:dim=deg} extends this correspondence to the full Kac-Moody setting, thereby establishing a uniform dimension formula of $X_w(1)$ for all split Kac-Moody groups.

\subsection{Base step}

Assume that \(w\) is of minimal length in its \(W\)-conjugacy class. By definition of the class polynomials, we have
\[
f_{w,[1]} = 
\begin{cases}
0, & \text{if } w \text{ has infinite order},\\[4pt]
\mathbf q^{\ell(w)}, & \text{if } w \text{ has finite order}.
\end{cases}
\]
It remains to establish the corresponding statements for the Deligne-Lusztig varieties.

\begin{proposition}\label{prop:DLMinLength}
    Let \(\mathcal O \in \operatorname{Cl}(W)\) and \(w \in \mathcal O_{\min}\).
    \begin{enumerate}
        \item If \(w\) has infinite order, then \(X_w(1) = \emptyset\).
        \item If \(w\) has finite order, then \(\dim X_w(1) = \ell(w)\).
    \end{enumerate}
\end{proposition}

\begin{proof}
    The proof is analogous to that of Proposition~\ref{prop:orbitalIntegralLift}.
    
    (1) In view of Theorem~\ref{thm:MarquisIndefiniteMinLength} and Proposition~\ref{prop:KacMoodyDeligneLusztigReduction}, we may assume that \(w = ab\) has the special form from Theorem~\ref{thm:MarquisIndefiniteMinLength}, with \(b\) straight and \(\ell(b) > 0\). Suppose \(gB(\overline{\mathbb F}_q) \in X_w(1)(\overline{\mathbb F}_q)\), and let \(v \in W\) with \(g \in B(\overline{\mathbb F}_q) v B(\overline{\mathbb F}_q)\). Write \(b' = g^{-1}\sigma(g) \in B(\overline{\mathbb F}_q) w B(\overline{\mathbb F}_q)\). Then for \(n \gg 0\),
    \[
        g^{-1}\sigma^n(g) = b' \sigma(b') \cdots \sigma^{n-1}(b') \in (B(\overline{\mathbb F}_q) w B(\overline{\mathbb F}_q))^n.
    \]
    This forces the coefficient of some \(T_{v_n}\) with \(\ell(v_n) \le 2\ell(v)\) in \((T_w)^n\) to be non-zero. But as in the proof of Proposition~\ref{prop:orbitalIntegralLift},
    \[
        (T_w)^n \in \sum_{\substack{u \in W \\ \ell(u) \ge n\ell(b)-\ell(w_{0,J})}} \mathbb Z[\mathbf q^{\pm 1}] T_u,
    \]
    which is impossible for \(n\) sufficiently large since \(\ell(b) > 0\). Hence \(X_w(1) = \emptyset\).
    
    (2) If \(w\) has finite order, let \(J = \operatorname{supp}(w) \subseteq S\). Then \(W_J\) is finite and the Levi subgroup \(M_J \le G^\bullet\) is reductive. Moreover, $w$ being elliptic in $W_J$ implies that the classical Deligne-Lusztig variety $X_w\subseteq M_J / (M_J\cap B^\bullet)$ is irreducible. The natural embedding
    \[
        M_J / (M_J \cap B^\bullet) \hookrightarrow G^\bullet / B^\bullet
    \]
    identifies the classical Deligne-Lusztig variety \(X_w\) with an irreducible component of \(X_w(1)\). Moreover, the \(G^\bullet(\mathbb F_q)\)-translates of this component cover \(X_w(1)\) (see the proof of Proposition~\ref{prop:orbitalIntegralLift} for the analogous argument). Since the classical Deligne-Lusztig variety has dimension \(\ell(w)\), we conclude that \(\dim X_w(1) = \ell(w)\). 
\end{proof}

\subsection{Inductive step}

Assume that \(w\) does not have minimal length in its conjugacy class, and that the theorem has been proved for all elements of smaller length. By Theorem~\ref{min}, we find a sequence
\[
w = w_1 \rightarrow_{s_1} w_2 \rightarrow_{s_2} \cdots \rightarrow_{s_{n-1}} w_n \rightarrow_{s_n} w_{n+1}
\]
with
\[
\ell(w_1) = \ell(w_2) = \cdots = \ell(w_n) = \ell(w_{n+1}) + 2.
\]
From Proposition~\ref{prop:class_polynomials}, we obtain
\[
f_{w_1,[1]} = \cdots = f_{w_n,[1]} = \mathbf q \, f_{w_{n+1},[1]} + (\mathbf q - 1) \, f_{s_n w_n,[1]}.
\]
Similarly, using Proposition~\ref{prop:KacMoodyDeligneLusztigReduction}, we have
\[
\dim X_{w_1}(1) = \cdots = \dim X_{w_n}(1) = 1 + \max\bigl(\dim X_{w_{n+1}}(1),\, \dim X_{s_n w_n}(1)\bigr).
\]

Now observe that
\[
\ell(s_n w_n) = \ell(w) - 1, \qquad \ell(w_{n+1}) = \ell(w) - 2.
\]
Thus, by the inductive hypothesis applied to \(s_n w_n\) and \(w_{n+1}\), we get
\[
1 + \dim X_{w_{n+1}}(1) = \deg_{\mathbf q}\bigl(\mathbf q \, f_{w_{n+1},[1]}\bigr),
\]
and
\[
1 + \dim X_{s_n w_n}(1) = \deg_{\mathbf q}\bigl((\mathbf q - 1) \, f_{s_n w_n,[1]}\bigr).
\]

Finally, since both \(\mathbf q \, f_{w_{n+1},[1]}\) and \((\mathbf q - 1) \, f_{s_n w_n,[1]}\) lie in \(\mathbb N[\mathbf q - 1]\), their sum has degree equal to the maximum of their individual degrees:
\[
\deg_{\mathbf q}\bigl(\mathbf q \, f_{w_{n+1},[1]} + (\mathbf q - 1) \, f_{s_n w_n,[1]}\bigr)
= \max\Bigl(
    \deg_{\mathbf q}\bigl(\mathbf q \, f_{w_{n+1},[1]}\bigr),
    \deg_{\mathbf q}\bigl((\mathbf q - 1) \, f_{s_n w_n,[1]}\bigr)
\Bigr).
\]

Combining these equalities yields \(\dim X_w(1) = \deg_{\mathbf q} f_{w,[1]}\). This completes the inductive step and the proof of Theorem~\ref{thm:dim=deg}.

\subsection{Parabolic Deligne--Lusztig varieties and partial class polynomials}

A natural generalization of Section~\ref{subsec:kac-moody-deligne-lusztig} is the study of \emph{parabolic Deligne--Lusztig varieties}. For \(J \subseteq S\), \(w,b \in G^\bullet(\overline{\mathbb F}_q)\), define
\[
X_w^J(b) := \{\, g \in P_J^\bullet(\overline{\mathbb F}_q) / B^\bullet(\overline{\mathbb F}_q) \mid g^{-1} b \, \sigma(g) \in B^\bullet(\overline{\mathbb F}_q) \, w \, B^\bullet(\overline{\mathbb F}_q) \,\}.
\]
The case \(J=S\) recovers the usual Deligne--Lusztig variety. For affine flag varieties, the parabolic Deligne-Lusztig varieties arise in the study of level-changing morphisms of Shimura varieties with different level structures. Following \cite{Vi-Ann}, one may reduce to \(b \in W^J\), so the geometry can be studied using the Hecke algebra of \(W\).

A dimension=degree theorem for such varieties requires \emph{partial class polynomials}, arising from the partial cocenter
\[
\overline H^J := H \,/\, \operatorname{span}_{\mathbb Z[\mathbf q^{\pm 1}]}\{\, h_1 h_2 - h_2 h_1 \mid h_1 \in H,\ h_2 \in H_J \,\}.
\]
The following result, joint with K. Wu, will appear in a forthcoming paper.

\begin{theorem}[to appear]\label{thm:partial-cocenter}
    Let \((W,S)\) be crystallographic and \(J \subseteq S\).
    \begin{enumerate}
        \item For any \(W_J\)-conjugacy class \(\mathcal O\) of elements in \(W\), and any \(w_1,w_2 \in \mathcal O_{\min}\), the images of \(T_{w_1}\) and \(T_{w_2}\) in \(\overline H^J\) coincide. Denote this common image by \(T_{\mathcal O}\).
        \item The elements \(T_{\mathcal O}\), as \(\mathcal O\) ranges over all \(W_J\)-conjugacy classes in \(W\), form a \(\mathbb Z[\mathbf q^{\pm 1}]\)-basis of \(\overline H^J\).
    \end{enumerate}
\end{theorem}

\subsection{Further directions}

The results of this paper raise several natural questions, including: 

\begin{itemize}
    \item \textbf{Non-split groups and unequal parameters}
A natural extension is to study non-split (quasi-split) Kac--Moody groups, whose Iwahori--Hecke algebras have unequal parameters. Extending our parabolic induction and trace modules to this setting would provide an algebraic foundation for a dimension=degree theorem for non-split Deligne--Lusztig varieties. It is an interesting question whether Lusztig's pairing remains invertible in the unequal-parameter case.

\item \textbf{Disconnected groups and twisted cocenters}
For disconnected Kac--Moody groups, the standard cocenter must be replaced by the \emph{twisted cocenter} \(\overline H_{\theta} := H / [H,H]_{\theta}\), where \([H,H]_{\theta}\) is generated by \(h h' - h' \theta(h)\) for a diagram automorphism \(\theta\). While \(\theta\)-twisted conjugacy classes provide a natural spanning set, their linear independence remains open and is a prerequisite for a twisted Deligne--Lusztig theory.

\item \textbf{Infinite non-crystallographic Coxeter groups.}
The generic Hecke algebra can be defined for every Coxeter system.
For finite non-crystallographic Coxeter groups, including types
$H_3$, $H_4$, and $I_2(m)$, the cocenter basis theorem follows from the classical theory of finite Hecke algebras. For non-finite, non-crystallographic Coxeter groups, since no associated Kac--Moody group is available in this setting, our geometric methods do not apply in this setting, and establishing linear independence of \(\{T_{\mathcal O}\}\) would require new purely algebraic trace functionals.

\item \textbf{Non-invertible specializations}
At specializations such as \(\mathbf q = 0\) or roots of unity, the Hecke algebra degenerates to a $0$-Hecke algebra. The structure of the cocenter in these cases is of independent combinatorial interest and may reveal connections to  the representation theory of 0-Hecke algebras.

\item \textbf{Deligne--Lusztig varieties: further extensions}
The dimension=degree theorem for \(X_w(1)\) can be extended in several directions: the ``dimension=degree'' theorem for arbitrary \(b\), the description of the irreducible components, and the formula for the ``normalized cardinality'', and the upper and lower bound of $\dim X_w(b)$ with respect to $\ell(w)$, etc. 
\end{itemize}

We hope the methods and questions developed here will stimulate further research at the interface of Hecke algebras, Coxeter groups, and Kac--Moody groups. Especially the first two points reveal natural open questions about finite groups of Lie type.

\end{document}